\documentclass[11pt]{article}
\usepackage[utf8]{inputenc}
\usepackage{amsfonts,amsmath,amsthm,amssymb}
\usepackage{mathtools}
\usepackage[english]{babel}
\usepackage[table]{xcolor}
\usepackage{hyperref}
\usepackage{dsfont}
\usepackage{fullpage}
\usepackage{mathtools}
\usepackage{tikz}
\usepackage{pgfplots}
\usepackage{enumitem}
\usepackage{empheq}

\usepackage[numbers,sort&compress]{natbib}

\usepackage{blindtext}

\pgfplotsset{width=10cm,compat=1.9}

\newcommand\numberthis{\addtocounter{equation}{1}\tag{\theequation}}

\newcommand{\norm}[1]{\left| \left|  #1 \right| \right|}
\newcommand{\pars}[1]{\left(  #1 \right)}
\newcommand{\set}[1]{\left\{ #1 \right\}}
\newcommand{\rr}{\mathbb{R}}
\newcommand{\nn}{\mathbb{N}}
\newcommand{\esper}[1]{\mathbb{E} \left[  #1 \right]}
\newcommand{\probd}[2]{\mathbb{P}_{\delta_{\left(  #1 \right)}} \left(  #2 \right)}
\newcommand{\esperd}[2]{\mathbb{E}_{\delta_{\left(  #1 \right)}} \left[  #2 \right]}
\newcommand{\ap}[2]{\left\langle #1, #2 \right\rangle}

\newtheorem{theorem}{Theorem}[section]
\newtheorem{lemma}[theorem]{Lemma}
\newtheorem{proposition}[theorem]{Proposition}
\newtheorem{corollary}[theorem]{Corollary}

\theoremstyle{definition}
\newtheorem{remark}[theorem]{Remark}
\newtheorem{definition}[theorem]{Definition}
\newtheorem{assumptions}[theorem]{Assumptions}
\newtheorem{example}[theorem]{Example}

\title{Phenotypic plasticity trade-offs in an age-structured model of bacterial growth under stress}
\author{
Meriem El Karoui\thanks{University of Edinburgh, School of Biological Sciences, Institute of Cell Biology; E-mail: \texttt{meriem.elkaroui@ed.ac.uk}}\and 
{Ignacio Madrid Canales}\thanks{Ecole Polytechnique, CNRS, Institut polytechnique de Paris, Inria, route de Saclay, 91128 Palaiseau Cedex-France; E-mail: \texttt{ignacio.madrid-canales@polytechnique.edu}} 
\and {Sylvie M\'el\'eard}\thanks{Ecole Polytechnique et Institut Universitaire de France, CNRS, Institut polytechnique de Paris, Inria, route de Saclay, 91128 Palaiseau Cedex-France; E-mail: \texttt{sylvie.meleard@polytechnique.edu}} }
 
\date{\today}

\begin{document}

\maketitle

\begin{abstract}
Under low concentrations of antibiotics causing DNA damage, \textit{Escherichia coli} bacteria can trigger stochastically a stress response known as the SOS response. While the expression of this stress response can make individual cells transiently able to overcome antibiotic treatment, it can also delay cell division, thus impacting the whole population's ability to grow and survive. In order to study the trade-offs that emerge from this phenomenon, we propose a bi-type age-structured population model that captures the phenotypic plasticity observed in the stress response. Individuals can belong to two types: either a fast-dividing but prone to death ``vulnerable" type, or a slow-dividing but ``tolerant" type. We study the survival probability of the population issued from a single cell as well as the population growth rate in constant and periodic environments. We show that the sensitivity of these two different notions of fitness with respect to the parameters describing the phenotypic plasticity differs between the stochastic approach (survival probability) and the deterministic approach (population growth rate). Moreover, under a more realistic configuration of periodic stress, our results indicate that optimal population growth can only be achieved through fine-tuning simultaneously both the induction of the stress response and the repair efficiency of the damage caused by the antibiotic.   
\end{abstract}



\section{Introduction and stochastic model}

Under the presence of antibiotics and other stress factors,  bacteria can exhibit a dynamic and heterogeneous expression of stress-response genes. In the case of \textit{Escherichia coli} growing under a sublethal concentration of an antibiotic causing DNA damage such as ciprofloxacin, the detection of DNA breaks on the chromosome triggers the initiation of the DNA damage response, which is called SOS response \cite{Witkin1967,Little1982,Ojkic2020}. The intensity of this response depends on the amount of damage, but it also exhibits high heterogeneity among individuals, even in a isogenic population, because of the stochastic expression of several factors \cite{Jones2021,Alnahhas2023,JaramilloRiveri2022}.   Observations at single-cell level of this stress response have shown that the heterogeneity of the SOS response is strongly dependent on the growth conditions \cite{JaramilloRiveri2022}.  Moreover, for each individual cell, the SOS response can fluctuate substantially in time, often as a sequence of pulses, transitioning from periods of apparent SOS inactivity to periods of strong SOS response \cite{Friedman2005,Lagage2020,Shimoni2009}. A distinct signature of the high SOS expression is the highly perturbed division dynamics caused by the SOS-dependent expression of a cell division inhibitor that induces, in turn, inter-division times which are much longer than in non-SOS-inducing cells. At the same time, however, cell elongation is not repressed by the SOS response, leading to the production of cells several times longer than the normally observed ones, a phenomenon known as filamentation. Importantly, SOS-inducing cells are able to produce non-SOS-inducing offspring, which results in a subset of cells that divide normally and could be able to rapidly take over the population once the stress is stopped \cite{wehrens2018,Raghunathan2020,Karasz2022}. Interestingly, similar phenotypic variability has also been shown in other stress response systems, impacting cell elongation or division with important consequences for the survival of the population \cite{Sampaio2022}.  Collectively, these observations suggest strong links between single-cell phenotypic heterogeneity and population-level stress survival, which might be a key to explaining antibiotic tolerance \cite{Alnahhas2023,Balaban2004,Acar2008}. 

Nonetheless, at least theoretically, the unconstrained heterogeneous expression of stress responses might lead to poor population-level performance, especially if the intensity of the stress response is anti-correlated to the mechanisms that usually contribute to the population fitness, such as fast division (inhibited by the SOS response) and volume increase (inhibited by other general stress responses).  Furthermore, the expression of stress response genes has been shown to be generally within the noisiest in \textit{E. coli}'s proteome \cite{Silander2012}. Thus, expression of the stress response might allow cells to survive in the short term but may come at the expense of their capacity to grow and divide, giving rise to a trade-off between individual stress expression and population growth. 

Here, to shed some conceptual and quantitative light on this question, we propose a minimal stochastic model in continuous time of a structured population that preserves the main elements that characterise the stress responses described above. First, each individual cell is characterised by their age $a \geq 0$ and a type $i \in \{0,1\}$. The type $i=0$ stands for the \textit{vulnerable cells} which do not induce a stress response, and are therefore fast dividing but can die at division with probability $p \in [0,1]$. Otherwise, with probability $1-p$, such a cell will produce two new identical daughter cells. Although there is no biological reason for death to happen at division, it seems a reasonable approximation to account in a minimal and very simple way the fact that the considered stresses (i.e. antibiotics) target only proliferating bacteria. On the other hand, type $i=1$ are \textit{tolerant cells} which react to the stress by inducing a response (e.g. by repairing their DNA in the case of the SOS response) which transiently protects them but leads them to divide after much longer times. 

We include age as a variable for the following reasons. Firstly, we want to account for memory-effects for both normally dividing and slowly dividing cells. Indeed, the distributions of division times, even in the absence of antibiotic and under ideal conditions, are not exponential and are much better explained by probability laws with memory, thus requiring the inclusion of an age-like structure in the construction of the model \cite{Doumic2015}. Secondly, we use age as a proxy for other cellular characteristics. For example, in the case of SOS-inducing cells, filamentous cells correspond to old individuals of type $1$. We, therefore, suppose that for each individual cell, the division mechanism is triggered by its age in a phenotype-dependent way. In particular, we suppose that the distribution of the interdivision times of cells of type $i$ are driven by an age-dependent division rate $\beta_i$ in the sense that for all individuals we have
\[
\mathbb{P} \pars{\textrm{Divide at age } < a + \Delta a \big| \textrm{Type } = i, \textrm{Age } \geq a} = \beta_i(a) \cdot o(\Delta a).
\]

\begin{figure}[h!]
    \centering
    \includegraphics[width=0.7\textwidth]{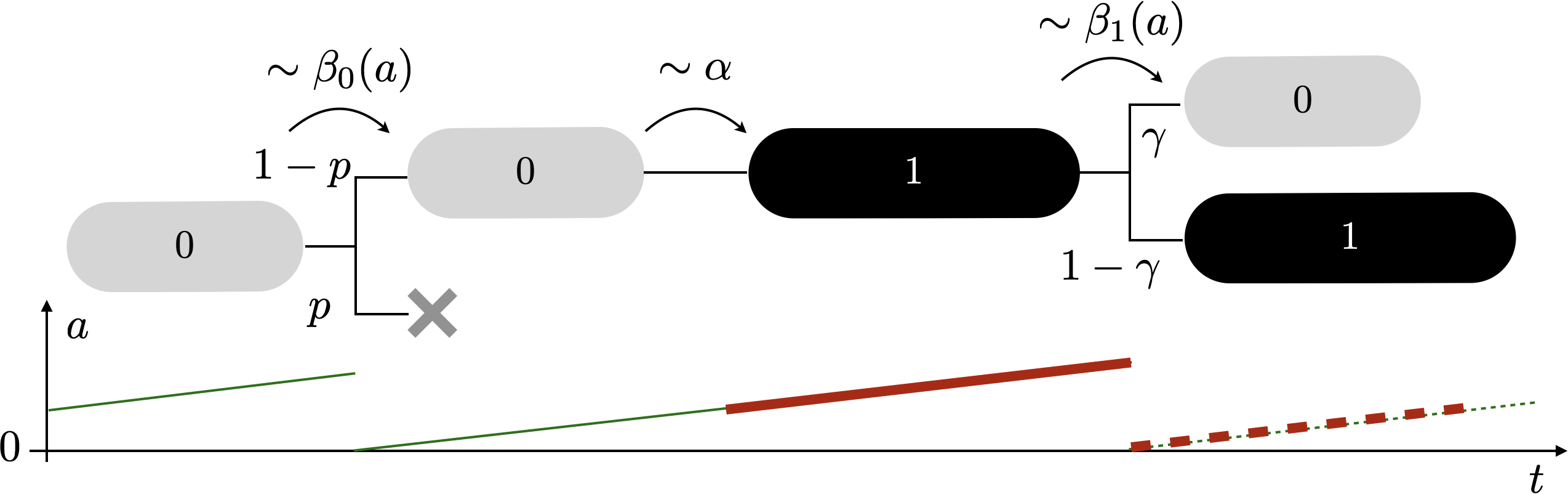}
    \caption{Schematic representation of the cell division mechanism and age.}
    \label{fig:tree}
\end{figure}

Following the approach introduced by \cite{Fournier2004,Tran2008}, we represent the population as a measure-valued stochastic process $(Z_t)_{t \geq 0}$ in continuous time, which at each time $t \geq 0$ can be written as a point measure
\begin{equation}
Z_t = \sum_{k=1}^{N_t} \delta_{(A_k(t), I_k(t))}
\label{eq:defZt}
\end{equation}
where $N_t = \ap{Z_t}{1}$ is the total number of cells alive at time $t$, and $(A_k(t), I_k(t)) \in \rr_+ \times \set{0,1}$ for $k \in \set{1,...,N_t}$ is the age and phenotype at time $t$ of cell number $k$, for any fixed arbitrary order (e.g. numbered using Neveu's notation, or by lexicographical order). In Appendix \ref{app-process} we define rigorously the paths of $Z_t$ as the solution of a Stochastic Differential Equation, which describes the desired dynamics conditionally to a given initial population.

We suppose that  a cell switch between type 0 and 1 occurs after a random exponential time with rate parameter $\alpha \geq 0$. In particular, this means that the $0 \to 1 $ switch is memoryless and can take place at any moment of the cell cycle, independently of the age of the cell. Importantly, the switch and division dynamics are supposed independent, so it is possible for a cell of type 0 to never switch, if its division time (determined by the value of $\beta_0$ at its current age) occurs before the intended switching time (determined by $\alpha$). On the other hand, we suppose that during the lifetime of a cell of type 1, switching back from type $1$ to type $0$ is impossible. However, when a cell of type $1$ divides, its offspring can become of type $0$ as consequence of a successful DNA repair. Specifically, we suppose that each daughter resulting by the division of a cell of type $1$ can be born with type $0$ with probability $\gamma \in [0,1]$, or otherwise keep type $1$ with probability $1 - \gamma$. Thus $\gamma$ can be interpreted as a measure of the DNA repair efficiency from one generation to the next one.

Table \ref{tab:params} summarises all these parameters and Fig \ref{fig:tree} represents the described dynamics.

\begin{table}[h]
    \centering
    \begin{tabular}{|c|l|}
    \hline
    \textbf{Parameter} & \textbf{Definition} \\
    \hline 
     $\beta_i(a) \in \rr_+$    & Division rate at age $a$ and type $i$ \\
    \hline
    $\alpha \in \rr_+$  & Switching rate from type 0 to 1 (phenotypic plasticity) \\
    \hline
     $p \in [0,1]$    & Probability of death at division for cells of type 0 (environmental effect) \\
    \hline
    $\gamma \in [0,1]$     & Probability that a type 1 divides and produces a type 0 (DNA repair effect) \\
    \hline
    \end{tabular}
    \caption{Summary of the model parameters.}
    \label{tab:params}
\end{table}

Several models have been studied in the context of other similar phenomena relating to seed banks and others (see the review by Lennon et al. \cite{Lennon2021} and the references therein). In the specific case of microbiological populations,  more recently, Blath et al. \cite{Blath2020, Blath2021} studied bi-type non-structured populations of \textit{active} and \textit{dormant} cells in a discrete-time setting and analysed the effect of different switching strategies in a randomly varying environment. In this work, we generalise these fundamental ideas to the case of stress response, where the division of tolerant individuals is not stopped but only affected, and give a detailed study of the fitness sensibility with respect to the parameters from both a probabilistic and deterministic approach.

Sections \ref{sec:extinctionProb} to \ref{sec:floquet} give the main results of the paper and discuss their biological implications, whilst we postpone all the proofs and technical details to Sections \ref{sec:proofs2} to \ref{sec:proofs4}. In Section~\ref{sec:extinctionProb} we derive the extinction probability of a population initiated by a single cell in an environment with a constant stress signal. Moreover, we obtain an explicit condition on the model parameters under which this population establishes with positive probability.  We relate this result with the capacity of \textit{evolutionary rescue} \cite{Bell2017, Carja2019} of phenotype-switching populations.
In Section \ref{sec:MtDynamics} we show that under the survival condition, we can observe a Malthusian behaviour for the first-moment semigroup of the stochastic process, characterised by an exponential population growth rate $\lambda > 0$ and a stationary distribution of types and ages. Moreover, we show that there is an equivalence relation between the criterion for the establishment of the population obtained by stochastic and deterministic approaches, which is, in general, not trivial for infinite-dimensional branching processes. Indeed, the measure of \textit{fitness} obtained from each approach, corresponds to different biological realities, as discussed in \cite{Metz1992, Durinx2005, Dieckmann1996}. In the stochastic case it corresponds to the probability that the branching process initiated by a single individual survives forever, while in the macroscopic case it is the asymptotic rate of exponential growth of the population. Our results extend the parallel results of Campillo et al. in \cite{Campillo2016,Campillo2017} for a mono-type growth-fragmentation-death case. 

However, in contrast with \cite{Campillo2017}, we show in Section \ref{sec:dlambda} that the extinction probability and the population growth rate do not always vary in the same direction with respect to variations in the parameter space. In particular, if the stress is low enough ($p < \bar p$ for some critical value $\bar p < 1/2$), increasing $\gamma$ will lead to a decreased survival probability of the population, but at the same time, to a higher population growth rate. Our proofs show that the loss of the classical monotonic behaviour arises from the crucial Assumption \ref{ass:domination} that vulnerable individuals divide faster than tolerant individuals. In a framework of size-structured cell growth, Calvez et al. in \cite{Calvez2012} have also shown such loss of monotonicity and Cloez et al.  in \cite{Cloez2021}, studied the variation of the population growth rate parameter with respect to an asymmetry factor.

Our results show that the optimal parameters in the sense of any of the two notions of fitness only correspond to extreme strategies $\gamma = 0$ or $\gamma = 1$. From both the point of view of the population survival probability, the creation of prone-to-death type 1 individuals is always detrimental, favouring a null recovery probability $\gamma = 0$. The same occurs from both the point of view of the population growth rate under high stress ($p > \bar p$). In that situation, an optimal population is expected to produce only type 1 ``tolerant" individuals with $\gamma = 0$, since any creation of cells of type 0 is also detrimental for growth. Under low stress ($p \le \bar p$), from the point of view of the population growth rate, it is optimal to produce only fast-dividing type 0 individuals, favouring $\gamma =1$.  

A more realistic aspect of such switching strategies appears when the environment (i.e., the stress parameter $p$) is allowed to fluctuate in time. In Section \ref{sec:floquet}, using Floquet's theory (cf. \cite{Clairambault2007, Clairambault2009}), we extend our results to the case where we consider a $T$-periodic signal $p(t)$. Although, in that case, our results cannot be equally quantitative, we show that the sign of the fitness variation with respect to the model parameters is not constant. Numerical simulations show that from the point of view of population growth rate, the optimal parameters change drastically and might require the simultaneous fine-tuning of both the recovery ($\gamma$) and switching ($\alpha$). We conclude with an outlook on the biological consequences of general stress-response strategies, particularly the SOS response on survival probabilities and population growth rate.

\section{Model assumptions and first properties}
\label{sec:construction}

We will work under the following set of assumptions.
\begin{assumptions} \hfill
\begin{enumerate}[label=(A\arabic*)]
    \item \label{eq:integrabilityBeta} Division times are almost surely finite: For all $i \in \{0,1\}$,
       $\  \int_0^{+\infty} \beta_i(a) da = + \infty$.
    
    \item Division rates are uniformly bounded: $\exists \; \bar b > 0$ such that $\forall i \in \set{0,1}, a \geq 0, \  \beta_i(a) \leq \bar b$.
    \item Division rates are uniformly bounded by below from a certain age: $\exists \; a_0 > 0, \underline b > 0$ such that $\forall i \in \set{0,1}$, $ \beta_i(a) \geq \underline b$ whenever $a \geq a_0$.
\end{enumerate}
\label{ass:assumptions}
\end{assumptions}

Under assumption \ref{eq:integrabilityBeta} we define the survival probabilities for an individual of type 0 or 1 not to experience any event between ages $s$ and $t$, given that it has already survived until age $s$:
\begin{align}\label{psi}
    \psi_0(s,t) &= \exp \pars { - \int_s^t (\alpha + \beta_0(u)) du  }, \\
    \psi_1(s,t) &= \exp \pars { - \int_s^t \beta_1(u) du  },
\end{align}
and we define the $\star$ composition between the two survival functions as
\begin{align} \label{convol}
\psi_0 \star \psi_1 (s,t) = \int_s^t \psi_0(s,u) \psi_1(u,t) du.
\end{align}
In particular, $\alpha \psi_0 \star \psi_1 (s,t) $ represents the probability for an individual of type 0 at initial age $s$, to switch to type 1 at some point between ages $s$ and $t$ and survive the remaining time until it has reached age $t$, given that it has already survived until age $s$.

\bigskip

We provide now some useful definitions that will be used in the next paragraphs.

\begin{definition}
\begin{enumerate}
    \item For any finite positive point measure $\mu$, set $\mathbb{P}_{\mu}$ the probability under the initial condition $Z_0 = \mu$, and $\mathbb{E}_{\mu}$ the respective expectation.
    \item Let $T$ be the first jump of $Z$.
    \item Let $(I_1, I_2)$ be the types of the two daughters obtained after the first jump, if it was a division.
\end{enumerate}
\end{definition}
Lemma \ref{lemma:probasvarias} stated below characterises the probability laws of these random variables, and will be used to compute the probability that a population initiated by a single bacterium goes extinct. 

\begin{lemma} 
\label{lemma:probasvarias} Under Assumptions \ref{ass:assumptions}, 
\begin{enumerate}
    \item The probability that one cell of initial state $(a, i)$ dies before time $t_0 \geq 0$, instead of dividing or switching before that, is
\begin{equation}
    \probd{a,i}{Z_{T} = 0, T \leq t_0} = \mathds{1}_{i=0} \ p \int_{0}^{t_0} \beta_i(a+t) \exp \pars{ - \int_{0}^t \beta_i(a+u) du - \alpha t} dt .
    \label{eq:probaExtinction1}
\end{equation}
\item For any bounded measurable function $h: \rr_+ \to \rr$, the conditional law of the switching time is characterised by
\begin{equation}
    \esperd{a,0}{h(T) \mathds{1}_{Z_{T} = \delta_{(a+T,1)}, T \leq t_0}} =  \int_{0}^{t_0} h(t) \ \alpha \exp \pars{ - \int_{0}^t \beta_0(a+u) du - \alpha t} dt .
    \label{eq:probaExtinction2}
\end{equation}

\item For any bounded measurable function $h: \set{0,1} \times \set{0,1} \to \rr$, the conditional law of the daughter types after division is characterised by
\begin{multline}
     \esperd{a,i}{h(I_1, I_2) \mathds{1}_{Z_T = \delta_{\pars{0,I_1}} + \delta_{\pars{0,I_2}}, T \leq t_0} }   = \\ \int_{0}^{t_0} \beta_i(a+t) \exp \pars{ - \int_{0}^t \beta_i(a+u) du - (1-i)\alpha t} dt  \ \Big\{ \mathds{1}_{i=0} \ (1-p) h(0,0)   \\
     + \mathds{1}_{i=1} \Big( \gamma ^2 h(0,0) +  (1-\gamma)^2 h(1,1) + \gamma(1-\gamma)(h(0,1) + h(1,0))   \Big) \Big\} 
     \label{eq:probaExtinction3}
\end{multline}
\end{enumerate}
\end{lemma}
\begin{proof}
By the construction introduced in Definition \ref{def:pathwise}, if the initial population consists on only one individual, $Z_0 = \delta_{(a,i)}$, then $T$ is the first jump time of the process
\[
J_t = \int_0^t \int_{\set{1} \times \rr_+ \times [0,1]^2 } \mathds{1}_{ \{ z \leq \alpha(1-i) + \beta_i(a+u) \} } \mathcal{N} (du, di, dz, d\omega),
\]
which, by definition of $\mathcal{N}$, is a non-homogeneous Poisson process whose time-dependent rate is then given by $t \mapsto  \pars{ (1-i) \alpha + \beta_i(a+t) } $. Therefore $T$ has the probability distribution 
\begin{equation}
    \probd{a,i}{T > t} = \probd{a,i}{J_t = 0} = \exp \pars{- \int_0^t \pars{(1-i) \alpha + \beta_i(a+u)} du }.
\end{equation}
From Assumptions \ref{ass:assumptions}~\ref{eq:integrabilityBeta} we can then deduce that $\probd{a,i}{T < +\infty} = 1 $. Moreover, by differentiation we get that $T$ admits the probability density function
\begin{equation}
    \probd{a,i}{T \in [t + dt[} =\pars{(1-i) \alpha + \beta_i(a+t)} \exp \pars{- \int_0^t \pars{(1-i) \alpha + \beta_i(a+u)} du } dt
    \label{eq:marginalT}.
\end{equation}
Now, the value of the process $Z_t$ at time $t = T$, i.e. just after the first jump, is given by
\begin{equation}
    Z_{T} = \begin{dcases}
    0 & \textrm{(death) with probability   } \frac{(1-i) p\beta_i(a+T)}{(1-i)\alpha + \beta_i(a+T)} \\
    \delta_{\pars{0,I_1}} + \delta_{\pars{0,I_2}}  & \textrm{(division) with probability   } \frac{((1-i)(1-p) + i) \beta_i(a+T)}{(1-i)\alpha + \beta_i(a+T)} \\
    \delta_{\pars{T,1}} & \textrm{(switch) with probability   } \frac{(1-i)\alpha}{(1-i)\alpha + \beta_i(a+T)} 
    \end{dcases} \ .
    \label{eq:conditionalZt}
\end{equation}
The equations of Lemma \ref{lemma:probasvarias} are obtained by computing joint probabilities and expectations using the marginal probability density of $T$, obtained from ~\eqref{eq:marginalT}, and the conditional probability of $Z_T$ given $T$, ~\eqref{eq:conditionalZt}. For example, 
we have
\begin{align*}
    \probd{a,i}{Z_{T} = 0, T \leq t_0} &= \esperd{a,i}{ \mathbb{P}(Z_T = 0 | T) \mathds{1}_{T \leq t_0} } \\
    &= \esperd{a,i}{\frac{(1-i) p\beta_i(a+T)}{(1-i)\alpha + \beta_i(a+T)} \mathds{1}_{T \leq t_0}} \\
    &= \mathds{1}_{i=0} \int_0^{t_0}  \frac{ p \beta_i(a+t)}{\alpha + \beta_i(a+t)} \pars{\alpha + \beta_i(a+t)} \exp \pars{- \int_0^t \pars{\alpha + \beta_i(a+u)} du } dt.
\end{align*}
 We obtain analogously the conditional law of the switching time. For the conditional law of the daughter types, we note that, if $i=0$, then $I_1 = I_2 = 0$, and that, if $i=1$, then $I_1$ and $I_2$ are independent Bernoulli random variables of parameter $1-\gamma$. 
\end{proof}

\bigskip
We are interested by the case motivated in Section 1, in which individuals of type 1, while tolerant, divide slower than individuals of type 0. We will only use this assumption to study the sensitivity of the population fitness with respect to the model parameters, starting in Section \ref{sec:dlambda}.

\begin{assumptions} 

 Let $T_{div}$ be the time of division of a non-switching cell, this is,
    \[
    \probd{0,i}{T_{div} \geq a} = \exp \pars {- \int_0^a \beta_i(s) ds }.
    \]
    We suppose that the type 0 division time is stochastically dominated (in first order) by the type 1 division time:
    \[
    \probd{0,0}{T_{div} \geq a} < \probd{0,1}{T_{div} \geq a}.
    \]
\label{ass:domination}
\end{assumptions}

 \begin{remark}
 \label{rmk:domination}
     For individuals of type 0, $T_1 = \textrm{min} \pars{T_{div}, \; T_{switch}}$, where $T_{switch}$ is a Exponential random variable of parameter $\alpha$. Therefore, for all $\alpha \geq 0$, the first event time $T_1$ is stochastically larger for individuals of type 1 than for individuals of type 0, which implies $\psi_0(0,a) < \psi_1(0,a)$. Moreover, for all non-decreasing function $g$ we we have $\esperd{0,0}{g(T_{1})} < \esperd{0,1}{g(T_{1})} $.
 \end{remark}

\section{Conditions for population establishment}
\label{sec:extinctionProb}

We now give conditions for the parameters $\alpha, \gamma$ such that for two given division rates $\beta_0, \beta_1$ and a death probability $p$, the population initiated by a single initial cell establishes indefinitely. We begin by computing the probability that the population initiated by a single cell goes extinct. 

\begin{definition}
Let $\pi(a,i)$ be the extinction probability of a population initiated by a single cell of type $i \in \{0,1\}$ and age $a \geq 0$. This is
\begin{equation}
    \pi(a,i) := \probd{a,i}{\exists t > 0 : N_t = 0}.
    \label{eq:pis_def}
\end{equation}
Set $\pi_i = \pi(0,i)$ the extinction probability associated with a single initial cell of age 0. 
\end{definition}

We will focus now on what happens when the initial cell has age 0. We will see, nevertheless, that the extinction probabilities of a population issued from a single cell of any initial age $ a > 0$ can be obtained explicitly from the two extinction probabilities $\pi_0$ and $\pi_1$. 

\begin{theorem} Let $T_1$ be the time of the first jump event. For all $a \geq 0$, let $q_a$ the probability that an individual of type 0 switches before dividing, conditionally to have already survived until age $a$:
\begin{equation*}
    q_a = \probd{0,0}{Z_{T_1} = \delta_{(T_1,1)} | T_1 \geq a } = \probd{0,a}{Z_{T_1} = \delta_{(T_1,1)} } = \int_0^{+\infty} \alpha \psi_0(a,t) dt
\end{equation*}
and let in particular
\begin{equation}
q = q_0 = \probd{0,0}{Z_{T_1} = \delta_{(T_1,1)}} = \int_0^{+\infty} \alpha \psi_0(0,t) dt    .  \label{eq:r0} 
\end{equation}
The vector of extinction probabilities $(\pi_0, \pi_1)$ is the smallest solution on $[0,1]$ of the quadratic system
\begin{subequations}
\begin{empheq}[left=\empheqlbrace]{align}
    \pi_0 &= (1-q) p + (1 - q)(1-p) \pi_0^2 + q \pi_1, 
    \label{eq:pi0} \\
    \pi_1 &= \pars{\gamma \pi_0 + (1-\gamma) \pi_1 }^2,
    \label{eq:pi1}
  \end{empheq}
\end{subequations}
in the sense that for any other admissible solution $\tilde \pi$, we have $\pi_i \leq \tilde \pi_i$ for both $i \in \{0,1\}$. Moreover, for any $a \geq 0$ we can obtain $\pi(a,i)$ as explicit functions of $\pi_0$ and $\pi_1$ (which justifies our analysis focused in the initial condition 0), as given by
\begin{align}
        \pi(a,0) &= (1-q_{a})p + (1-q_{a})(1-p) \pi_0^2 + q_{a} \pi_1 
        \label{eq:pi0a} \\
    \pi(a,1) &= \pi_1.
    \label{eq:pi1a}
  \end{align}
\label{thm:pisystem}
\end{theorem}
The proof of this theorem is postponed to Section \ref{sec:proofs2}.

\begin{remark}
Since cells of type 1 do not die in their generation, the division rate $\beta_1$ does not play any role in the extinction probability. 
\end{remark}

\begin{remark}
Thanks to the integrability assumption \ref{eq:integrabilityBeta}, we can differentiate the integral \eqref{eq:r0} to obtain $\left. \partial_{\alpha}q \right|_{\alpha=0} =  \esperd{0,0}{T_{\textrm{div}}}$, where $T_{\textrm{div}}$ is the division time of cells of type 0 (which depends only on $\beta_0$). Therefore, at least for small values of $\alpha$, the sensitivity of $q$ with respect to $\alpha$ is proportional to the mean division time of vulnerable cells.  We give a more practical example below.
\end{remark}

\begin{example}[Gamma distributed inter-division times] Suppose that $\beta_0$ is such that for some $a_0, b_0 > 0$, we can write for all $t \geq 0$
\[
\beta_0(t) \exp \pars{- \int_0^t \beta_0(u) du } = \frac{b_0^{a_0}}{\Gamma(a_0)} t^{a_0-1} e^{- b_0 t} .
\]
Then, the inter-division times are Gamma random variables of shape parameter $a_0$ and rate parameter $b_0$. This has been shown to be a good parametric model to explain the distributions of division ages \cite{Golubev2016}. An integration by parts of ~\eqref{eq:r0} shows that under this assumption,
\[
q = 1 - \pars{1 + \frac{\alpha}{b_0} }^{- a_0}.
\]
The general shape for $q$ as function of $\alpha$ is given in Fig. \ref{fig:r0_gamma}. Calibrating the values of $a_0$ and $b_0$ can modify the sensibility of $q$ with respect to $\alpha$. In particular, choosing $a_0 = 1$ reduces to the memoryless case $\beta_0 \equiv b_0$, where age does not affect the division times, which are then identically distributed exponential random variables of rate parameter $b_0$. Generally, notice that the derivative at the origin of $q$ with respect to $\alpha$ is equal to $a_0/ b_0$, which is the expected division time in the Gamma case. 
\label{ex:gammaBeta0}
\end{example}

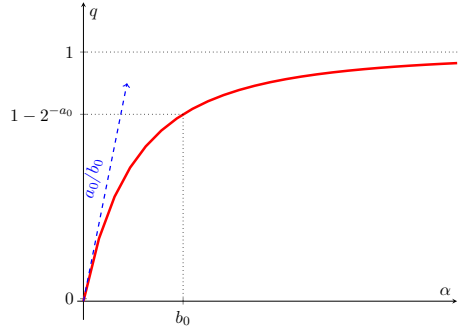
\begin{figure}
    \centering

\begin{tikzpicture}[scale=0.6]
\begin{axis}[
    axis lines = middle,
    xlabel={$\alpha$},
    ylabel={$q$},
    xmin=-0.25, xmax=15, xtick={4}, xticklabels={$b_0$},
    ymin=-0.05, ymax=1.2/1.5, ytick={0.01/1.5,0.75/1.5, 1/1.5}, yticklabels={0,$1-2^{-a_0}$,1}]
\addplot[red, ultra thick, domain=0:15] {(1 - (1 + x/4)^(-2))/1.5};
\addplot[blue, thick, dashed, domain=0:1.75, ->] {x/2/1.5} node[above, sloped, pos = .55] {$a_0/b_0$} ;
\addplot[dotted,mark=none] coordinates {(4, 0) (4, 0.75/1.5)};
\addplot[dotted,mark=none] coordinates {(0, 0.75/1.5) (4, 0.75/1.5)};
\addplot[dotted,mark=none] coordinates {(0, 1/1.5) (15, 1/1.5)};
\end{axis}

\end{tikzpicture}
    \caption{Form of $q$ as function of $\alpha$ in the case of division times following a Gamma distribution of parameters $(a_0, b_0)$ (Example \ref{ex:gammaBeta0}). Note that $q$ is always an increasing function of $\alpha$.}
    \label{fig:r0_gamma}
\end{figure}

\bigskip

In particular, we can characterise explicitly the subcritical region in which extinction happens almost surely as a function of $\gamma$, $p$ and the probability $q$ introduced above. The proof of Theorem \ref{prop:survival} is postponed to Section \ref{sec:proofs2}. 

\begin{theorem}
\label{prop:survival}
The population initiated by a single cell of age 0 survives with positive probability if and only if
\begin{equation}
    \set{p \leq \frac{1}{2}} \textrm{  or  } \set{ p > \frac{1}{2} \textrm{ and } \gamma < \frac{1}{2} \pars {1 + \frac{q}{(2p-1)(1-q)}} }.
    \label{eq:extinctionCondition}
\end{equation}
\end{theorem}

\begin{corollary}
Let $p > 1/2$. Then the extinction probability of the population is at most $\frac{1}{2}(1 - \log 2)$, a value which is attained in the extreme case $p=1$ where individuals of type 0 die almost surely at each division attempt.
\end{corollary}
\begin{proof}
The proof consists simply in calculating the area of the extinction region delimited by the complement of condition \eqref{eq:extinctionCondition}. Indeed, the probability that the population traits belong to this region (so that the establishment probability is 0) is equal to
\begin{align*}
    \iint_{[0,1]^2} \mathds{1}_{\set{\gamma > \frac{1}{2} \pars{1 + \frac{q}{(2p-1)(1-q)}}   }} dq d\gamma &= \int_0^{\frac{2p-1}{2p}} dq \int_{\frac{1}{2} \pars{1 + \frac{q}{(2p-1)(1-q)}} }^1 d \gamma \\
    &= \frac{1}{2} \pars{p^2 + \frac{1}{2p-1} \log \pars{1 - \frac{p(2p-1)}{2}} }.
\end{align*}
We can see easily that it is an increasing function of $p$ for $p > 1/2$ and thus its maximum its attained for $p=1$ which gives the value $\frac{1}{2}(1 - \log 2)$.
\end{proof}

\begin{remark}
A maximum \emph{evolutionary risk} of $\frac{1}{2}(1 - \log 2) \approx 15.34 \%$ might seem not very restrictive. However, one could argue that keeping a high value of $q$, i.e. a high phenotypic plasticity, could be associated with a high energetic cost and other constraints not included in our model, which would enlarge the zone of non viability. 
\end{remark}
\begin{remark}
Note that the previous computation  can also be interesting in an evolutionary framework, since it gives a measure of the capacity of random evolutionary rescue. 
 If a mutant cell appears with $p>1/2$ and new characteristics $(\gamma',q')$ that are independently and uniformly distributed on the square $[0,1]$, the  survival probability of its subpopulation will be at least $\frac{1}{2}(1 + \log 2)$ and minimum for $p=1$. 
\end{remark} 
 
\section{Long-time behaviour of the population and links with the population establishment condition}
\label{sec:MtDynamics}

In the following we establish some fundamental links between necessary and sufficient conditions for the establishment of a population issued from a single cell from both a probabilistic approach based on the trajectories of $Z_t$, and for a deterministic approach based on the long-time behaviour of the semigroup $M_t$, which describes the expected value of the population dynamics, as defined below. Dichotomy properties linking the survival probability with the behaviour of $M_t$ have been studied in size-structured models by \cite{Campillo2017}. Numerical studies have been performed in the same spirit by \cite{Fritsch2017}. We show in particular that under the survival conditions there is a positive Malthusian parameter $\lambda > 0$ such that the rescaled dynamics $e^{-\lambda t}M_t$ converge to a non-zero stationary measure, and that this convergence is at exponential rate.

\begin{definition}[First-moment semigroup and vector representations]
\label{def:Mt}
Let us define over the space of bounded Borel functions $\mathcal{B}_b(\rr_+ \times \set{0,1})$, the first-moment semigroup $M_t: \mathcal{B}_b(\rr_+ \times \set{0,1}) \to \mathcal{B}_b(\rr_+ \times \set{0,1})$ by
\begin{equation}
    M_t f(a,i) = \esperd{a,i}{\ap{Z_t}{f}}, \ \forall (a,i) \in \rr_+ \times \set{0,1}
\end{equation}
and for all signed Borel measure $\mu \in \mathcal{M}(\rr_+ \times \set{0,1})$ we define $\mu M_t \in \mathcal{M}(\rr_+ \times \set{0,1})$ as the measure which, for all $f \in \mathcal{B}_b(\rr_+ \times \set{0,1})$, verifies the duality relation
\begin{equation}
    \ap{\mu M_t}{f} = \ap{\mu}{M_t f} = \int_{\rr_+ \times \set{0,1}} M_t f(a,i) \mu(da, di).
    \label{eq:duality}
\end{equation}
We give also a vector representation, which will be useful in the sequel. Let us write $\mathbf{f} = (f(\cdot, 0), f(\cdot, 1)) \in (\mathcal{B}_b( \rr_+))^2$ and define for all $a \geq 0$ the matrix semigroup $\mathbf{M}_t : (\mathcal{B}_b( \rr_+))^2 \to (\mathcal{B}_b( \rr_+))^2$ as
\[
\mathbf{M}_t \mathbf{f}(a) = (M_t f(a,0), \ M_t f(a,1)) \in \rr^2. \\
\]
Analogously, for all Borel set $A$ of $\rr_+$, we let $\boldsymbol{\mu}(A) = ( \mu(A \times \set{0}) , \ \mu(A \times \set{1}) ) \in (\mathcal{M}(\rr_+))^2$ and define
\[
\boldsymbol{\mu} \mathbf{M}_t (A) = ( \mu M_t (A \times \set{0}) , \ \mu M_t (A \times \set{1}) ) \in \rr^2.
\]
\end{definition}

\bigskip

Taking expectations in the semi-martingale decomposition \eqref{eq:semimartinagaledecomposition} associated with $Z_t$, we easily show that the infinitesimal generator $\mathcal{Q}$ associated to $\mathbf{M}_t$ such that for all $a \geq 0$
\[
\frac{d}{dt}  \mathbf{M}_t \mathbf{f} (a) = \mathbf{M}_t \pars{ \mathcal{Q} \mathbf{f}} (a) = \mathcal{Q} \pars{\mathbf{M}_t \mathbf{f} } (a),
\]
is given by
\begin{equation}
    \mathcal{Q} \mathbf{f}(a) = \mathbf{f}'(a) - \mathbf{D}(a) \mathbf{f}(a) + 2 \mathbf{B}(a) \mathbf{f}(0)
    \label{eq:Q}
\end{equation}
where $\mathbf{f}'(a) := (\partial_{a} f(a,0), \partial_a f(a,1))$ and
\begin{equation}
\mathbf{B}(a) = \begin{bmatrix}
(1-p) \beta_0(a) & 0  \\
\gamma \beta_1(a) & (1- \gamma) \beta_1(a) 
\end{bmatrix} \textrm{ and} \quad  \mathbf{D}(a) = \begin{bmatrix}
\alpha + \beta_0(a) & - \alpha  \\
0 & \beta_1(a)
\end{bmatrix}.
\label{eq:matricesBD}
\end{equation}

An useful alternative approach is the following representation of $\mathbf{M}_t \mathbf{f}$ as the mild solution to a renewal equation:

\begin{proposition}[Forward Equation]
\label{prop:FE}
For all test function $\mathbf{f} \in (\mathcal{B}_b( \rr_+))^2$ in the form introduced above, the right action of the semigroup $\mathbf{M}_t \mathbf{f}$ is solution to the renewal equation, for all $ a \geq 0$
\begin{equation}
    \mathbf{M}_t \mathbf{f}(a) = \boldsymbol{\Psi}(a,a+t) \mathbf{f}(a+t) + 2 \int_0^t \mathbf{K}(a,a+s) \mathbf{M}_{t-s} \mathbf{f}(0) ds ,
    \label{eq:FE}
\end{equation}
where the matrix $\boldsymbol{\Psi}(s,t)$ is given by
\begin{equation}
\label{eq:PsiMatrix}
    \boldsymbol{\Psi}(s,t) = \begin{bmatrix}
    \psi_0(s,t) & \alpha \psi_0 \star \psi_1(s, t) \\
    0 & \psi_1(s,t)
    \end{bmatrix}
\end{equation}
and the kernel $\mathbf{K}$ is given by
\begin{align}
    \mathbf{K}(s,t) &= \begin{bmatrix}
    (1-p) \beta_0(t) \psi_0(s,t) + \gamma \alpha  \beta_1(t) \psi_0 \star \psi_1(s,t) &  (1-\gamma) \alpha \beta_1(t) \psi_0 \star \psi_1(s,t) \\
    \gamma \beta_1(t) \psi_1(s,t)  & (1-\gamma) \beta_1(t) \psi_1(s,t)
    \end{bmatrix}.  \label{eq:K} 
\end{align}
\end{proposition}

\smallskip \noindent
Recall that $\psi_0$ and $\psi_1$ have been defined respectively in (2) and (3). The first column of the matrix kernel $\mathbf{K}$ corresponds to the possible outcomes for individuals of type 0, which can persist with probability $(1-p)$ or switch at rate $\alpha$ and then give offspring of type $0$ or $1$ with probabilities $\gamma$ and $1 - \gamma$ respectively. The second column corresponds to the individuals of type 1, whose contributions are pondered by $\gamma$ or $1 - \gamma$ as recalled above. 

\begin{remark}
In particular, by setting $a=0$, $t \mapsto \mathbf{M}_t \mathbf{f}(0)$ is the unique fixed point of
\begin{equation}
    \mathbf{M}_t \mathbf{f}(0) = \boldsymbol{\Psi}(0,t) \mathbf{f}(t) + 2 \int_0^t \mathbf{K}(0,s) \mathbf{M}_{t-s} \mathbf{f}(0) ds,
    \label{eq:FE0}
\end{equation}
and then we obtain $ \mathbf{M}_t \mathbf{f}(a)$ for all $a \geq 0$ injecting this fixed point into the integral term of ~\eqref{eq:FE}. 
\end{remark}

\smallskip \noindent
The proof of Proposition \ref{prop:FE} is postponed to Section \ref{sec:proofs3}. 
Analogously, the left action of the semigroup can immediately be identified to the measure solution to the following PDE.

\begin{proposition}[Multitype renewal PDE]
For all initial vector measure $\boldsymbol{\mu} \in (\mathcal{M}(\rr_+))^2$, the vector measure $\boldsymbol{\mu}\mathbf{M}_t$ is equal to the measure-valued solution $\mathbf{n}(t,\cdot)$ of the multitype McKendrick–von Foerster Equation
\begin{equation}
\begin{cases}
\partial_t \mathbf{n}(t,a) &= - \partial_a \mathbf{n}(t,a) - \mathbf{D}^\top(a) \mathbf{n}(t,a) \\
\mathbf{n}(t,0) &= 2 \int_0^{+\infty} \mathbf{B}^\top(a) \mathbf{n}(t,da) \\
\mathbf{n}(0,\cdot) &= \boldsymbol{\mu}
\end{cases}
\label{eq:PDE}
\end{equation}
with $\mathbf{B}^\top$ and $\mathbf{D}^\top$ the transposed matrices of the ones defined by \eqref{eq:matricesBD}.

Moreover, if $\boldsymbol{\mu}$ are absolutely continuous with respect to the Lebesgue measure, then $\boldsymbol{\mu}\mathbf{M}_t$ is a strong solution to ~\eqref{eq:PDE}.
\end{proposition}

Then, solving ~\eqref{eq:PDE} by variation of parameters we obtain that the vector measure $\boldsymbol{\mu}\mathbf{M}_t$ admits, for any test function $f \in \mathcal{B}_b(\rr_+)$, the representation
\begin{equation}
    \ap{\boldsymbol{\mu}\mathbf{M}_t}{f} = \int_0^t 
    f(a) \boldsymbol{\Psi}^\top(0,a) \boldsymbol{\eta}(t-a) da + \int_{t}^{+\infty} f(a) \boldsymbol{\Psi}^\top(a-t,a) \boldsymbol{\mu}(a-t) da.
\end{equation}
The function $\boldsymbol{\Psi}(s,t)$ appears in this new context as the fundamental matrix solution to the ODE
\begin{subequations}
\begin{empheq}[left=\empheqlbrace]{align}
\partial_t \boldsymbol{\Psi}^\top(s,t) &= - \mathbf{D}^\top(t) \boldsymbol{\Psi}^\top(s,t) , \quad t > s \\
\boldsymbol{\Psi}(s,s) &= \mathbf{I}
\end{empheq}
\end{subequations}
with $\mathbf{I}$ the $2 \times 2$ identity matrix, and $\boldsymbol{\eta}$ is defined by the boundary condition of ~\eqref{eq:PDE}, giving
\begin{equation}
    \boldsymbol{\eta}(t) = 2 \int_0^t \mathbf{B}^\top(a) \boldsymbol{\Psi}(0,a) \boldsymbol{\eta}(t-a) da + 2 \int_t^{+\infty} \mathbf{B}^\top(a) \boldsymbol{\Psi}(a-t,a) \boldsymbol{\mu}(a-t) da.
\end{equation}
In particular, for an initial condition $\boldsymbol{\mu} = (c_0 \delta_0, c_1 \delta_0 )$  consisting on $c_0$ initial individuals of type 0 with age 0, and $c_1$ initial individuals of type 1 with age 0, we have that $\boldsymbol{\eta}$ is solution to the linear Volterra equation of the second kind
\begin{equation}
    \boldsymbol{\eta}(t) = g(t)  + 2 \int_0^t \mathbf{K}^\top(0,a) \boldsymbol{\eta}(t-a) da .
\end{equation}
since indeed $\mathbf{B}^\top(a) \boldsymbol{\Psi}(0,a) = \mathbf{K}^\top(0,a)$ for $\mathbf{K}$ introduced in ~\eqref{eq:K}. Meanwhile, the inhomogeneous term is given by
\begin{equation}
    g(t) = 2 \mathbf{B}^\top(t) \boldsymbol{\Psi}(0,t) \begin{pmatrix} c_0 \\ c_1 \end{pmatrix} = 2 \mathbf{K}^\top(0,t) \mathbf{c} 
\end{equation}
We can interpret $\boldsymbol{\eta}(t)$ as the instantaneous number of offspring produced at time $t$. The term $g(t)$ gives the contribution of the initial individuals that have survived during $[0,t]$ before dividing. The integral term counts the contribution at time $t$ of the individuals of age $a$ (born $t-a$ ago).

\bigskip

We are interested in the long-time behaviour of the semigroup $\mathbf{M}_t$ in the case when we have survival of the population, using some classical ideas from the spectral theory of $C_0$-semigroups, adapting the approach followed by Webb in \cite{webb1985theory} to age-structured population dynamics, and more recently applied in \cite{Meleard2019}  to study the equilibrium of a birth-death model of ageing, also formulated as an individual-based stochastic model.  To this end, we set ourselves on the Banach space $(L^1(\rr_+))^2$ equipped with the norm $\norm{\mathbf{f}}_{1} = \int_0^{+\infty} (|f(a,0)| + |f(a,1)|)da $. We also write $\norm{\cdot}_1$ for vectors and matrices, meaning, as usually: $\norm{\mathbf{x}}_1 = \sum_i |x_i| $, and $\norm{\mathbf{A}}_1 = \max_j \sum_{i} {|A_{ij}|}$. 
We then consider $\mathbf{M}_t : (L^1(\rr_+))^2 \to (L^1(\rr_+))^2$, which is the mild solution of ~\eqref{eq:FE} on $(L^1(\rr_+))^2$. The existence of such semigroup is a direct consequence of the well-posedness of the measure-valued process $Z_t$ and the control of its first moment as stated in Prop. \ref{prop:nonexplosion}. We then obtain:

\begin{theorem}
Under Assumptions \ref{ass:assumptions} and if the survival conditions established by ~\eqref{eq:extinctionCondition} are verified, there is a unique triplet of a positive function $\mathbf{h} \in (L^1(\rr_+))^2$, a positive Radon measure $\boldsymbol{\nu} \in (\mathcal{M}(\rr_+))^2$, normalised such that $\ap{\boldsymbol{\nu}}{1} =1 $ and  $\ap{\boldsymbol{\nu}}{\mathbf{h}} =1 $, and a positive constant $\lambda > 0$ such that for all $\mathbf{f} \in (L^1(\rr_+))^2$
\begin{equation}
    \norm{ e^{-\lambda t} \mathbf{M}_t \mathbf{f} - \ap{\boldsymbol{\nu}}{\mathbf{f}} \mathbf{h} }_1 \leq c e^{(\omega - \lambda) t } \norm{\mathbf{f} - \ap{\boldsymbol{\nu}}{\mathbf{f}} \mathbf{h} }_1.
\end{equation}
The positive number $\lambda$ is the population growth rate and is often called the Malthusian parameter
or the population fitness; it is the largest real root of the characteristic equation
\begin{equation}
\det \pars{\mathbf{F}(\lambda) - \mathbf{I}} = 0,
\label{eq:malthusianParam}
\end{equation}
where
\[
\mathbf{F}(\lambda) := 2 \int_0^{+\infty} e^{-\lambda a } \mathbf{K}(0,a) da .
\]
Moreover, both coordinates of $\boldsymbol{\nu}$ admit a density with respect to the Lebesgue measure. 
\label{prop:main}
\end{theorem}

An important role is played by the matrix
\[
\mathbf{K}_\infty := \mathbf{F}(0) = 2 \int_0^{+\infty} \mathbf{K}(0,a) da,
\]
whose spectral properties determine the long-time behaviour of $\mathbf{M}_t$. Lemmas \ref{lemma:rhoKinf} and \ref{lemma:uniquenessLambda} will be useful to prove Theorem \ref{prop:main} in Section \ref{sec:proofs3}. Furthermore, they show how the conditions for survival with positive probability derived in Theorem \ref{prop:survival} and the existence of a positive eigenvalue $\lambda > 0$ are linked through the spectral properties of $\mathbf{K}_\infty$.

\begin{lemma}
The survival condition ~\eqref{eq:extinctionCondition} is equivalent to have to $\rho(\mathbf{K}_\infty) > 1$
\label{lemma:rhoKinf}
\end{lemma}

\begin{lemma}
Under Assumptions \ref{ass:assumptions}, there is a unique $\lambda > 0$ such that $\rho(\mathbf{F}(\lambda)) = 1$ (in particular, $\lambda$ is solution to the characteristic equation \eqref{eq:malthusianParam}) if and only if the survival conditions established by ~\eqref{eq:extinctionCondition} are verified.
\label{lemma:uniquenessLambda}
\end{lemma}

We conclude this section by noticing a useful bound for the value of $\lambda$, which is natural to obtain when the division rates are uniformly bounded.
\begin{remark}
    We have $\lambda \leq \bar b$, with $\bar b$ the bound on the division rate of  (A1) of Assumptions~\ref{ass:assumptions}. 
    \label{rmk:boundLambda}
\end{remark}

\section{Sensitivity of the population growth rate and the survival probability with respect to phenotypic switching strategies}
\label{sec:dlambda}

 In the following we denote by $\mathcal{Q}_{\alpha, \gamma}$ the generator \eqref{eq:Q} and $(\lambda_{\alpha, \gamma}, \boldsymbol{\nu}_{\alpha, \gamma}, \mathbf{h}_{\alpha, \gamma})$ the triplet of elements verifying Theorem~\ref{prop:main} for a given pair of parameters $(\alpha, \gamma)$ in the survival region defined by Proposition~\ref{prop:survival}. First, we show that the eigenfunction $(\alpha, \gamma) \in \rr_+ \times [0,1] \mapsto \mathbf{h}_{\alpha,\gamma} \in (L^1(\rr_+))^2$ is indeed continuous in $\alpha$ and $\gamma$. This will allow us to study the variations of the population growth rate with respect to $\alpha$ and $\gamma$. 

The proofs of next lemmas and propositions are postponed to Section \ref{sec:proofs4}

\begin{lemma}
    \label{lemma:continuity}
    Under Assumptions~\ref{ass:assumptions} and if $\beta_0, \beta_1 \in C(\rr_+)$,  then $\mathbf{h} \in C^1(\rr_+ , \; \rr_+^2)$. 
\end{lemma}

\begin{lemma}
\label{lemma:continuity2}
Under Assumptions~\ref{ass:assumptions}, for all fixed $a \geq 0$, the map $(\alpha, \gamma) \mapsto \mathbf{h}_{\alpha,\gamma}(a)$  is continuous for the uniform norm.
\end{lemma}

Proposition~\ref{prop:calculDlambda} characterises the partial variations of the population growth rate $\lambda_{\alpha, \gamma}$ with respect to $\alpha$ and $\gamma$. 

\begin{proposition}
\label{prop:calculDlambda}
For fixed $(\alpha, \gamma)$ and $(\lambda_{\alpha, \gamma}, \boldsymbol{\nu}_{\alpha, \gamma}, \mathbf{h}_{\alpha, \gamma})$ the triplet of eigenelements associated to $\mathcal{Q}_{\alpha, \gamma}$ we have that both $\alpha \mapsto \lambda_{\alpha, \gamma}$ and $\gamma \mapsto \lambda_{\alpha, \gamma}$ are continuously differentiable functions such that
\begin{align}
    \partial_{\alpha} \lambda_{\alpha, \gamma} &= \int_0^{+\infty} \pars{ h_{\alpha, \gamma}(a,1) - h_{\alpha, \gamma}(a,0) } \nu_{\alpha,\gamma}(da,0), 
    \label{eq:dlambda_dalpha} \\ 
    \partial_{\gamma} \lambda_{\alpha, \gamma} &= 2 \pars{ h_{\alpha, \gamma}(0,0) - h_{\alpha, \gamma}(0,1) } \int_0^{+\infty} \beta_{1}(a) \nu_{\alpha,\gamma}(da,1)   
    \label{eq:dlambda_dgamma} .
\end{align}
\end{proposition}

\bigskip

Let us recall that the eigenfunction $\mathbf{h}$ corresponds to Fisher's reproductive value \cite{Fisher1930}: $h(a,i)$ is a measure of the contribution of an individual of age $a$ and type $i$ to the future growth of the population. Indeed, the long time behaviour of the expected total number of individuals issued from an individual of age $a$ and type $i$ is $M_t 1(a,i)$ and by Theorem~\ref{prop:main} is given by $e^{\lambda t } h(a,i)$. Thus, Equations \eqref{eq:dlambda_dalpha} and \eqref{eq:dlambda_dgamma} show that the value of the fitness response to variations in $\alpha$ and $\gamma$ depends on the difference in the reproductive values of type 1 and type 0. 

In particular, in the case of the variations with respect to parameter $\gamma$, we see that the sign of $\partial_{\gamma} \lambda_{\alpha, \gamma}$ depends only on the sign of the difference between the newborn's reproductive values of type 0 and type 1, which are given by the vector $\mathbf{h}_{\alpha, \gamma}(0)$. Moreover, Proposition~\ref{prop:dgamma} below shows that increasing the probability $\gamma$ is beneficial for the growth of the population if and only if the population growth rate is already larger than the population growth rate associated with the subpopulation of type 1. Then, exploiting the equivalency of Lemma \ref{lemma:rhoKinf}, we can link the variations of the population growth rate with the variations of the establishment probability. This will allow to show that increasing $\gamma$ is detrimental from a Malthusian (population growth) point of view only if the death probability $p$ is greater than some critical value $\bar p$, i.e. if the environmental stress is high enough. In contrast, increasing $\gamma$ is always detrimental from the point of view of the establishment probability. This is also shown by the numerical simulations presented in Fig. \ref{fig:regimeChange}. We will see later that it is not true in the more general case where the environment changes in time.

\begin{proposition}
\label{prop:dgamma}
For $(\alpha,\gamma)$ in the survival region defined by Proposition~\ref{prop:survival}, we have the following implicit equivalence
\begin{align*}
    \partial_{\gamma} \lambda_{\alpha, \gamma} > 0 \iff \lambda_{\alpha, \gamma} > \lambda^{*}_1 \\
    \partial_{\gamma} \lambda_{\alpha, \gamma} < 0 \iff \lambda_{\alpha, \gamma} < \lambda^{*}_1 
\end{align*}
where $\lambda^{*}_1$ is the population growth rate of subpopulation 1 alone, i.e., the unique solution to
\[
1 = 2 \int_0^{+\infty} e^{- \lambda^*_1 a } \beta_1(a) \psi_1(0,a) da.
\]
\end{proposition}

\begin{corollary}
\label{cor:dlambdage0}
If $\lambda_{\alpha,\gamma=0} > \lambda_1^*$ then for all $\gamma \in [0,1[$, $\partial_{\gamma} \lambda_{\alpha, \gamma} > 0$. 
\end{corollary}

\begin{proof}
Let us recall that we denote $\xi_i(\lambda) = \esperd{0,i}{e^{-\lambda T_{div}}}$ the Laplace transform associated to the division time of type $i \in \{0,1\}$. In the case $\gamma = 0$, the matrix $\mathbf{F}(\lambda)$ becomes triangular, and thus the characteristic equation 
$\ 
\det \pars{ \mathbf{I} - \mathbf{F}(\lambda _{\alpha, \gamma}) } = 0\ $
reduces to
\[
\pars{2(1-p) \xi_0(\alpha + \lambda) - 1 }\pars{ 2 \xi_1(\lambda) -1 } = 0,
\]
which admits as solutions $\lambda_1^*$ and some $\tilde{\lambda}_0 \in \rr$ such that
$\ 
2(1-p) \xi_0(\alpha + \tilde{\lambda}_0 ) - 1 = 0$. By Theorem \ref{prop:main}, the fitness $\lambda_{\alpha,\gamma=0}$ corresponds then to the maximum value between $\lambda_1^*$ and $\tilde \lambda_0$. If $\lambda_1^* < \tilde \lambda_0$ (or equivalently, $\lambda_{\alpha,\gamma=0} > \lambda_1^*$) , by Proposition \ref{prop:dgamma}, $\left. \partial_{\gamma} \lambda_{\alpha,\gamma} \right|_{\gamma = 0} > 0$. And then, by the continuity of $\gamma \mapsto \partial_{\lambda} \lambda_{\alpha,\gamma}$, which follows easily from the continuity properties exhibited in the proof of Proposition \ref{prop:calculDlambda}, $ \partial_{\gamma} \lambda_{\alpha,\gamma} > 0$ for all $\gamma \in [0,1[$. 
\end{proof}

Notice however that if $\lambda_1^* \geq\tilde \lambda_0$, then $\lambda_{\alpha,\gamma=0} = \lambda_1^*$ and thus by Proposition \ref{prop:dgamma}, $\left. \partial_{\gamma} \lambda_{\alpha,\gamma} \right|_{\gamma = 0} = 0$. Therefore, we cannot conclude about the sign of $\partial_{\gamma} \lambda_{\alpha,\gamma}$ for $\gamma \in (0,1)$. As shown by Fig. \ref{fig:laplace01}, this can happen if $(1-p)(1-q)$ is small enough, for example, if $p > 1/2$ or $q > 1/2$ (but also in more general cases, as the one illustrated). Indeed, as proven above, when $\gamma = 0$ the population growth rate $\lambda$ is the largest solution of $\pars{2(1-p) \xi_0(\alpha + \lambda) - 1 }\pars{ 2 \xi_1(\lambda) -1 } = 0$, i.e., the largest value at which $\lambda \mapsto (1-p) \xi_0 (\alpha + \lambda)$ (solid blue curve in Fig. \ref{fig:laplace01}) or $\lambda \mapsto \xi_1(\lambda)$ (solid red curve in Fig. \ref{fig:laplace01})  pass through $1/2$. Moreover, at $\lambda = 0$, the first function equals $(1-p) \xi_0(\alpha) = \int_0^{\infty} \beta_0(a) \psi_0(0,a) da = (1-p)(1-q)$, by the definition of $q$ given in \eqref{eq:r0}. In the case represented in Fig. \ref{fig:laplace01}, $(1-p)(1-q)$ is small enough to impose $\lambda_1^* > \tilde \lambda_0$, so the largest solution to the characteristic equation \eqref{eq:malthusianParam} is the type 1 growth rate $\lambda_1^*$. In order to study $ \partial_{\gamma} \lambda_{\alpha,\gamma} $ in this more general case, we use the results derived for the establishment probability in Section \ref{sec:extinctionProb} to obtain the following result.

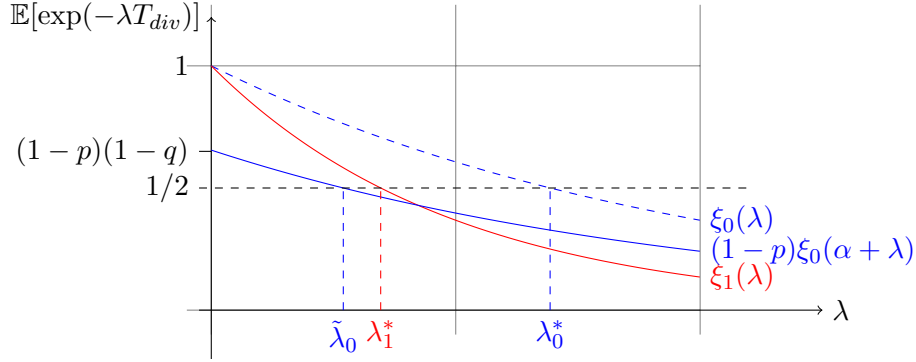
\begin{figure}
\centering
\begin{tikzpicture}[domain=0:2,scale=3.25] 
\draw[very thin,color=gray] (-0.1,-0.1) grid (2,1.25);

\draw[->] (-0.05,0) -- (2.5,0) node[right] {$\lambda$}; \draw[->] (0,-0.2) -- (0,1.2) node[left] {$\mathbb{E} [\exp ( - \lambda T_{div} ) ]$};
\draw[dashed, color=blue, samples=150] plot (\x,{exp(- 0.5 * \x )}) node[right, color=blue] {$ \xi_0(\lambda)$} ;
\draw[color=red, samples=150] plot (\x,{exp(- \x )}) node[right, color=red] {$ \xi_1(\lambda)$} ;
\draw[color=blue, samples=150]	plot (\x,{0.8 * exp(- 0.5 * ( 0.4 + \x  ) )}) node[right, color=blue] {$(1-p)\xi_0(\alpha + \lambda)$}  ;  
\draw[dashed] (-0.05,0.5) -- (2.2,0.5) ;
\draw[dashed, color=red] (0.693,0) node[below] {$\lambda_1^*$}  -- (0.693,0.5) ;
\draw[dashed, color=blue] (1.386,0) node[below] {$\lambda_0^*$} -- (1.386,0.5) ;
\draw[dashed, color=blue] (0.54,0) node[below] {$\tilde \lambda_0$} -- (0.54,0.5) ;
\draw (-0.05, 1) node[left]    {$1$};
\draw[-] (-0.05, 0.65) node[left]    {$(1-p)(1-q)$} -- (0, 0.65);
\draw (-0.05, 0.5) node[left]    {$1/2$};
\end{tikzpicture}
\caption{Laplace transforms of division times of type 0 ($\xi_0$, blue dashed line) and type 1 ($\xi_1$, red solid line), and the translation $(1-p)\xi_0(\cdot + \alpha)$ (blue solid line). When $\gamma = 0$, the population growth rate is the largest value of $\lambda$ at which one of the two latter functions (solid lines) pass through $1/2$.}
\label{fig:laplace01}
\end{figure}

\begin{proposition}
    \label{prop:dlambdaGamma_equiv}
    Under the supplementary Assumption \ref{ass:domination}, for all $\alpha \geq 0$ there exists a unique critical value $\bar p \leq 1/2$ such that for all $\gamma \in (0,1)$
    \begin{align*}
    \partial_\gamma \lambda_{\alpha,\gamma} > 0 &\iff p < \bar p , \\
    \partial_\gamma \lambda_{\alpha,\gamma} < 0 &\iff p > \bar p .
    \end{align*}
\end{proposition}

The proof are postponed to Section \ref{sec:proofs4}. Figure \ref{fig:prop54} summarises the results of Propositions \ref{prop:dgamma} and \ref{prop:dlambdaGamma_equiv}. It illustrates the dependence of $\lambda_{\alpha,\gamma}$ with respect to $\gamma$, for some fixed value of $\alpha \in (0,1)$. The two basins delimited by $\lambda_1^*$ identified in Proposition \ref{prop:dgamma} are shown with arrows: $\gamma \mapsto \lambda_{\alpha,\gamma}$ is increasing in the region $\lambda > \lambda_1^*$, and decreasing in the region in the region $\lambda < \lambda_1^*$. Whenever $p < \bar p$ (blue curve), $\lambda_{\alpha, \gamma = 0} > \lambda_1^*$ and therefore .$\gamma \mapsto \lambda_{\alpha,\gamma}$ is increasing. This is the key idea of Proposition \ref{prop:dlambdaGamma_equiv}'s proof.

   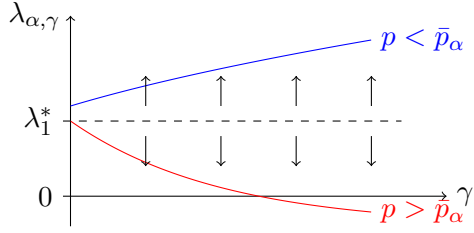
\begin{figure}
\centering
\begin{tikzpicture}[domain=0:2,scale=2] 
\draw[->] (-0.05,0) -- (2.5,0) node[right] {$\gamma$}; \draw[->] (0,-0.2) -- (0,1.2) node[left] {$\lambda_{\alpha,\gamma}$};
\draw[color=red, samples=150] plot (\x,{0.7 * exp(- \x ) - 0.2}) node[right, color=red] {$ p > \bar p_{\alpha}$} ;
\draw[color=blue, samples=150] plot (\x,{0.6 * sqrt( 1 + \x )}) node[right, color=blue] {$ p < \bar p_{\alpha}$} ;
\draw[dashed] (-0.05,0.5) -- (2.2,0.5) ;
\draw (-0.05, 0.5) node[left]    {$\lambda_1^*$};
\draw (-0.05, 0) node[left]    {$0$};
\draw[->] (0.5, 0.6) -- (0.5, 0.8);
\draw[->] (1, 0.6) -- (1, 0.8);
\draw[->] (1.5, 0.6) -- (1.5, 0.8);
\draw[->] (2, 0.6) -- (2, 0.8);
\draw[->] (0.5, 0.4) -- (0.5, 0.2);
\draw[->] (1, 0.4) -- (1, 0.2);
\draw[->] (1.5, 0.4) -- (1.5, 0.2);
\draw[->] (2, 0.4) -- (2, 0.2);
\end{tikzpicture}
\caption{Illustration of Propositions \ref{prop:dgamma} and \ref{prop:dlambdaGamma_equiv}.}
\label{fig:prop54}
\end{figure}

\bigskip

Surprisingly, if we look at the establishment probability, we do not necessarily observe the same variations with respect to $\gamma$. Proposition~\ref{prop:dgamma_pi} below shows that increasing $\gamma$ will increase the extinction probability, for all $p \in [0,1]$ (and thus even for $p < \bar p$):

\begin{proposition}
\label{prop:dgamma_pi}
For all $p \in [0,1]$, $\alpha > 0$, $\gamma \in [0,1[$, let $( \pi_0^{\alpha,\gamma}, \pi_1^{\alpha,\gamma}) \in [0,1]^2$ the minimal solution to \eqref{eq:pi0}-\eqref{eq:pi1}. Then 
\[
\partial_{\gamma} \pi_0^{\alpha,\gamma} > 0 \; \; \textrm{and} \; \; \partial_{\gamma} \pi_1^{\alpha,\gamma} > 0.
\]
\end{proposition}

\begin{proof}
Note that in the system \eqref{eq:pi0}-\eqref{eq:pi1} characterising the extinction probability, only ~\eqref{eq:pi1} depends on the value of $\gamma$. Moreover, for $(\pi_0, \pi_1) \in (0,1)^2$ verifying \eqref{eq:pi1} we have
\[
\pi_1 = \frac{1-2\gamma(1-\gamma)\pi_0 - \sqrt{1 - 4 \gamma(1-\gamma)\pi_0}}{2(1-\gamma)^2},
\]
and therefore, for all fixed value of $\pi_0$ we have
\[
\partial_{\gamma} \pi_1 = 
\frac
{\sqrt{1 - 4 \gamma (1-\gamma )\pi_0} + (1-\gamma) \pi_0 \pars{1 + 2 \gamma - \sqrt{1 - 4 \gamma (1-\gamma )\pi_0}} - 1}
{(1-\gamma)^3 \sqrt{1 - 4 \gamma (1-\gamma )\pi_0}},
\]
which is always non-negative for $\gamma \in [0,1]$ and $\pi_0 \in [0,1]$, since in that case $\pi_0 \leq 1/(4 \gamma(1-\gamma))$. 
Moreover, from ~\eqref{eq:pi0} we obtain that $\pi_1$ is an increasing function of $\pi_0$. Thus, for $(\pi_0, \pi_1)$ solution of the system, we have both $\partial_{\gamma} \pi_0 > 0$ and $\partial_{\gamma} \pi_1 > 0$. 
\end{proof}

In particular, only in the case $p > \bar p$, both the survival probability $1-\pi_i^{\alpha,\gamma}$ and the population growth rate $\lambda_{\alpha,\gamma}$ decrease with respect to the recovery probability $\gamma$. This result is in contrast to \cite{Fritsch2017}, where the monotonic dependence observed in the survival probability translated to a monotonic dependence also in the population growth rate. As discussed by \cite{Fritsch2017}, this can be explained by a loss of monotonicity in the survival probabilities. In our case, albeit the survival probability varies monotonically with respect to $\gamma$ (Proposition \ref{prop:dgamma_pi}), the monotonicity is lost by the presence of two types. In Appendix \ref{app:problink} we exhibit this loss of monotonicity using probabilistic arguments.

\bigskip

\begin{corollary}
For all $\alpha > 0$ and $\gamma \in (0,1)$, we have $h_{\alpha,\gamma}(0,1) >  h_{\alpha,\gamma}(0,0)$ if $p > \bar p$, and $h_{\alpha,\gamma}(0,1) \leq  h_{\alpha,\gamma}(0,0)$ otherwise.
\label{cor:diffH}
\end{corollary}
\begin{proof}
    It follows directly from Proposition \ref{prop:dlambdaGamma_equiv} and the explicit form of $\partial_{\gamma} \lambda_{\alpha,\gamma}$ obtained in Proposition \ref{prop:calculDlambda}. 
\end{proof}

Corollary \ref{cor:diffH} shows that under high stress, the reproductive value $h_{\alpha,\gamma}(0,0)$ of newborn individuals of type 0, i.e., the contribution of type 0 individuals to the asymptotic size of the population, is always less important than the reproductive value of individuals of type 1. Indeed, even if they take a long time to divide and if the switching events are rare, the fact that type 1 individuals are \textit{perfectly adapted} to the stress environment (in the sense that they reproduce without dying), make that they end up contributing statistically more to the population size. It is not hard to imagine that this is not generally true when individuals of type 1 happen to be less adapted. For example, under a changing environment, the subpopulation of type 0 can be allowed to proliferate if the stress is reduced ($p < \bar p$) at some intervals of time. This could be sufficient to make the contribution of type 0 larger in the asymptotic population. In the next section we explore an extension of the model where this happens. We will consider the case where the death probability $p$ is allowed to vary periodically in time.

\section{Sensitivity of the population growth rate and the survival probability under periodic stress}
\label{sec:floquet}

We  now consider the case where the death probability $p$ evolves periodically in time. Let $T > 0$ a time period and $t \in [0,T[ \mapsto p(t)\in [0,1]$ a $T$-periodic function. Only this term is allowed to fluctuate in time; $\alpha$ and $\gamma$ are considered fixed traits of the population and remain constant. Since deaths occur at birth, only the birth matrix $\mathbf{B}$ is affected and we define thereby
\[
\mathbf{B}(t,a) := \begin{bmatrix}
(1-p(t)) \beta_0(a) & 0  \\
\gamma \beta_1(a) & (1- \gamma) \beta_1(a) 
\end{bmatrix}.
\]
Analogously, we define the time-inhomogeneous generator
\begin{equation}
    \mathcal{Q}(t) \mathbf{f} (a) := \mathbf{f}'(a) + 2 \mathbf{B}(t,a) \mathbf{f}(0) - \mathbf{D}(a) \mathbf{f}(a),
    \label{eq:Q_NH}
\end{equation}
whose adjoint operator is for every $t \geq 0$
\[
\mathcal{Q}^*(t) \mathbf{n} (a) := -\mathbf{n}'(a) - \mathbf{D}^\top(a) \mathbf{n}(a),
\]
and has time-dependent domain
\[
\mathcal{D}(\mathcal{Q}^*(t)) = \set{\mathbf{n} \in W^{1,\infty}(\rr_+): \mathbf{n}(0) = 2 \int_0^{+\infty} \mathbf{B}^\top(t,a) \mathbf{n}(a) da}.
\]
We consider as well the time-inhomogenenous associated matrices
\[
\mathbf{F}(t,\lambda):= 2 \int_0^{+\infty} e^{-\lambda a} \boldsymbol{\Psi}(0,a) \mathbf{B}(t,a) da .
\]
Similarly to the use of Perron-Frobenious Theorem in the previous case, Floquet's Theorem (see for example \cite{Perthame2007}, p.163) now allows us to construct the eigenelements that will drive the long-time behaviour of the periodic dynamics. 

\begin{proposition}
\label{prop:floquet}
Let $p: \rr_+ \to [0,1]$ a $T$-periodic continuous function. Under the same set of Assumptions \ref{ass:assumptions} there exists a unique triplet of eigenelements $(\lambda_T, \boldsymbol{\nu}, \mathbf{h})$ where $\lambda_T \in \rr$ and $\boldsymbol{\nu} = \boldsymbol{\nu}(t,da)$ and $\mathbf{h} = \mathbf{h}(t,da)$ are time-dependent $T$-periodic positive continuous functions such that
\begin{subequations}
\begin{empheq}[left=\empheqlbrace]{align}
    - \partial_t \mathbf{h}(t,a) &=  \mathcal{Q}(t) \mathbf{h}(t,a) -\lambda_T \mathbf{h}(t,a),
    \label{eq:Floquet_h} \\
    \partial_t \boldsymbol{\nu}(t,a) &= \mathcal{Q}^*(t)  \boldsymbol{\nu}(t,a) - \lambda_T \boldsymbol{\nu}(t,a) \; , \; \boldsymbol{\nu}(t,\cdot) \in \mathcal{D}(\mathcal{Q}^*(t)),
    \label{eq:Floquet_nu} \\
     \int_0^T \int_0^{+\infty} \boldsymbol{\nu} (t,a) da dt &= 1 \; , \; \int_0^T \int_0^{+\infty }\mathbf{h}(t,a) \boldsymbol{\nu} (t,a) da dt = 1 .
    \label{eq:Floquet_norma}
  \end{empheq}
\end{subequations}    
\end{proposition}

\begin{proof}
    We adapt the general ideas developed on Section 5 of \cite{Michel2005} and in Appendix A of \cite{Clairambault2009}, extending Section \ref{sec:MtDynamics} to the case where $t \mapsto p(t)$ is a continuous $T$-periodic function. If we wish to allow discontinuities (for example to account for piecewise constant treatments) we might follow the approach of \cite{Clairambault2009}, that requires to estimate supplementary controls. 

    By variation of parameters, if $(\lambda_T, \boldsymbol{\nu}, \mathbf{h})$ are solutions to \eqref{eq:Floquet_h}-\eqref{eq:Floquet_norma} then
    \begin{align}
        \mathbf{h}(t,a) &= 2 \int_a^{+\infty} e^{-\lambda_T(s-a)}\boldsymbol{\Psi}(a,s) \mathbf{B}(t+s-a,s) \mathbf{h}(t+s-a,0) ds \label{eq:hta}, \\
        \boldsymbol{\nu}(t,a) &= e^{-\lambda_T a}  \boldsymbol{\Psi}^\top(0,a)  \boldsymbol{\nu}(t-a,0) ,
        \label{eq:nuta}
    \end{align}
    For the first equation, valuated the integral at $a=0$ we obtain that $\mathbf{h}(t,0)$ is solution to the integral fixed point problem
        \begin{equation}
            \mathbf{h}(t,0) = \mathcal{F}_{\lambda_T}(t) \mathbf{h}(\cdot,0) = 2 \int_0^{+\infty} e^{-\lambda_T a }\boldsymbol{\Psi}(0,a) \mathbf{B}(t+a,a) \mathbf{h}(t+a,0) da .
        \end{equation}
Analogously, since $\boldsymbol{\nu}$ must verify the boundary condition imposed by $\boldsymbol{\nu}(t,\cdot) \in \mathcal{D}(\mathcal{Q}^*(t))$, at $a=0$ we obtain that $\mathbf{h}(t,0)$ is solution to the integral fixed point problem
        \begin{equation}
          \boldsymbol{\nu}(t,0) = \mathcal{G}_{\lambda_T}(t) \boldsymbol{\nu}(\cdot,0) = 2 \int_0^{+\infty} e^{-\lambda_T a} \mathbf{B}^\top(t,a) \boldsymbol{\Psi}^\top(0,a) \boldsymbol{\nu}(t-a,0) da.  
        \end{equation}
        The sequel is classical. By Arzela-Ascoli Theorem, one shows that under Assumptions \ref{ass:assumptions}, for all $\lambda > 0$ and $t \geq 0$ the operators $\mathcal{F}_\lambda(t)$ and $\mathcal{G}_\lambda(t)$ are continuous, strictly positive and compact. Thus, by Krein-Rutman Theorem there is a simple dominant eigenvalue $\mu_\lambda > 0$ associated to eigenfunctions $\boldsymbol{\nu}_\lambda(t,0)$ and $\mathbf{h}_\lambda(t,0)$. Then, by the maximum principle and analogously as done in the proof of Lemma \ref{lemma:uniquenessLambda}, we obtain that $\lambda \mapsto \mu_\lambda$ is a continuous and strictly decreasing map. Hence there is a unique $\lambda_T$ such that $\mu_{\lambda_T} = 1$ with associated eigenfunctions  $\boldsymbol{\nu}_{\lambda_T}(t,0)$ and $\mathbf{h}_{\lambda_T}(t,0)$. We can finally recover $\boldsymbol{\nu}(t,a)$ and $\mathbf{h}(t,a)$ using the reconstruction formulae \eqref{eq:hta} and \eqref{eq:nuta}. 
    
\end{proof}

\begin{remark}
    We recall that the Floquet dominant eigenvalue $\lambda_T$ gives indeed the growth rate of the population. Set
    \[
    M_{s,t} f(a,i) = \esper{ \left. \ap{Z_t}{f}\right| Z_s = \delta_{(a,i)}}
    \]
    the time-inhomogenenous semigroup associated to $\mathcal{Q}(t)$ and denote $\mathbf{M}_{s,t}$ its vectorial form, as in the previous sections.  Then, for $\mathbf{h}(t,\cdot)$ solving \eqref{eq:Floquet_h},  we have immediately that
    \[
    \mathbf{M}_{s,t} \mathbf{h}(s,a) = e^{\lambda_T (t-s)} \mathbf{h}(t,a).
    \]
\end{remark}

Now, as in Section \ref{sec:dlambda}, we study the variations of $\lambda_T$ with respect to the model parameters. Using the normalisation conditions \eqref{eq:Floquet_norma} and repeating the same calculations as in the proof of Lemma \ref{lemma:continuity} and Proposition \ref{prop:calculDlambda} we obtain the Proposition \ref{prop:dgammaFloquet} below.

\begin{proposition}
  \label{prop:dgammaFloquet}
  Let $(\lambda_{T,\alpha,\gamma}, \boldsymbol{\nu}_{\alpha,\gamma}, \mathbf{h}_{\alpha, \gamma})$ the triplet of Floquet eigenelements associated to $T$-periodic $\mathcal{Q}_{\alpha, \gamma}(t)$. We have
\begin{align}
    \partial_{\alpha} \lambda_{\alpha, \gamma} &= \int_0^T \int_0^{+\infty} \pars{ h_{\alpha, \gamma}(t,a,1) - h_{\alpha, \gamma}(t,a,0) } \nu_{\alpha,\gamma}(t, a,0) dadt, 
    \label{eq:dlambda_dalpha_Floquet} \\ 
    \partial_{\gamma} \lambda_{\alpha, \gamma} &= 2 \int_0^T \pars{ h_{\alpha, \gamma}(t,0,0) - h_{\alpha, \gamma}(t,0,1) } \pars{ \int_0^{+\infty} \beta_{1}(a) \nu_{\alpha,\gamma}(t,a,1) da } dt.
    \label{eq:dlambda_dgamma_Floquet} 
\end{align}
\end{proposition}
\begin{proof}
    It follows directly from \eqref{eq:Floquet_norma} and repeating the same calculations as in the proof of Proposition \ref{prop:calculDlambda}. The continuity of $(\alpha, \gamma) \in \rr_+ \times [0,1] \mapsto \mathbf{h}_{\alpha,\gamma}(t,\cdot)$ for all fixed $t \geq 0$, follows also directly from Lemma \ref{lemma:continuity}. 
\end{proof}
 
    We compare this result with the one obtained in the constant environment case. We focus in the variations with respect to parameter $\gamma$. In the constant environment case we saw that if $p>\bar p$ increasing $\gamma$ is always detrimental from the point of view of the population growth rate (see Proposition \ref{prop:dlambdaGamma_equiv}, and the discussion of a numerical example in Section~
    \ref{sec:discuss1}). This occurs since the reproductive value $h_{\alpha,\gamma}(0,1)$ of type 1 is always larger than the reproductive value $h_{\alpha,\gamma}(0,0)$ of type 0 when the level of stress is constant. However, we see now that in the fluctuating case the sign of $\partial_{\gamma} \lambda_{\alpha, \gamma}$ depends on some time-average of the reproductive value difference. In particular, this difference is weighted proportionally to the mean division rate of type 1, $\bar \beta_1(t) := \int_0^{+\infty} \beta_{1}(a) \nu_{\alpha,\gamma}(t,a,1) da$, observed at that time. Hence, if there are times $t$ at which $p(t) \leq 1/2$, such that the reproductive value of type 0 is able to be larger than the reproductive value of type 1, it is possible for the average difference to be positive. In particular, the coincidence of these times with times at which the mean type 1 division rate $\bar \beta_1(t)$ is large, can lead to positive values of $\partial_{\gamma} \lambda_{\alpha, \gamma}$. This comes off naturally from a heuristic reasoning. Indeed, big values of $\bar \beta_1(t)$ and $\gamma$ would lead to a burst in the creation of individuals of type 0 at time $t$, which possesses the biggest relative advantage at that time. We discuss a numerical illustration of this phenomenon in Section \ref{sec:discuss3}.

    \smallskip
    
Finally, we can also compute the extinction probabilities. The time-dependent value of $p$ breaks the Markovian property at the division stopping times, preventing the problem from being reduced to a simple algebraic system as in the constant environment case. The new system, however, can be solved numerically. 

\begin{proposition}
    Let~$\ 
\pi_i(s,a) := \mathbb{P} \pars{\exists t \geq s : N_t = 0 | Z_s = \delta_{(i,a)}}$. Then, we have that $(\pi_0,\pi_1)$ is the minimal solution on $[0,1]^2$ of the system~\begin{align*}
        \pi_0(s,a) = & \int_0^{+\infty} p(s+t) \beta_0(a+t) \psi_0(a,a+t) dt + \int_0^{+\infty}  \alpha\,\pi_1(s+t,a+t) \psi(a,a+t) dt \\
        &+ \int_0^{+\infty} \pars{\pi_0(s+t,0)}^2 (1-p(s+t)) \beta_0(a+t) \psi_0(a,a+t) dt; \\
        \pi_1(s,a) = & \int_0^{+\infty} \pars{ \gamma \pi_0(s+t,0) + (1-\gamma) \pi_1(s+t,0)}^2 \beta_1(a+t) \psi_1(a,a+t) dt.
    \end{align*}~Moreover $t \mapsto \pi_i(t,\cdot)$ is $T$-periodic.
\end{proposition}
\begin{proof}
    It follows directly repeating the steps in Prop. \ref{prop:survival} 
\end{proof}

\section{Proofs of Section \ref{sec:extinctionProb}}
\label{sec:proofs2}

\subsection{Proof of Theorem \ref{thm:pisystem}}

We compute the extinction probability in a general way that will be useful when time inhomogeneity is included in Sections \ref{sec:floquet} and \ref{sec:proofs4}, this is, when $p$ is a periodic function instead of a constant. Note however that since for now $p$ is fixed, we could have shown that the extinction probability of $Z_t$ equals the extinction probability of some embedded discrete-time branching process giving the number of particles at the $n$-th generation, which is no other than a multitype Galton-Watson process \cite{Athreya1972}.

\begin{proof}
Conditioning with respect to the possible outcomes of the first jump, we have
\begin{align*}
    \pi(a,i) = \ & \probd{a,i}{Z_T = 0} + \probd{a,i}{\exists t > T: N_t = 0, Z_T \neq 0} \\
    = \ & \probd{a,i}{Z_T = 0} + \probd{a,i}{\exists t > T: N_t = 0, Z_T = \delta_{a+T,1}} \\
    &+ \probd{a,i}{\exists t > T: N_t = 0, Z_T = \delta_{(0,I_1)} + \delta_{(0,I_2)}}. \\
\intertext{Applying the strong Markov property on the stopping time $T$ gives}
  \pi(a,i)  = \ & \probd{a,i}{Z_T = 0} + \esperd{a,i}{\probd{a+T, 1}{\exists t > 0: N_t = 0} \mathds{1}_{Z_T = \delta_{a+T,1}}} \\
    &+ \esperd{a,i}{\mathbb{P}_{\delta_{(0,I_1)} + \delta_{(0,I_2)}} \pars{ \exists t > 0: N_t = 0} \mathds{1}_{Z_T = \delta_{(0,I_1)} + \delta_{(0,I_2)}}}, \\
\intertext{which by the independence of the processes starting from $\delta_{(0,I_1)}$ and $\delta_{(0,I_2)}$, gives}
  \pi(a,i)  = \ & \probd{a,i}{Z_T = 0} + \esperd{a,i}{\probd{a+T, 1}{\exists t > 0: N_t = 0} \mathds{1}_{Z_T = \delta_{a+T,1}}} \\
    &+ \esperd{a,i}{\probd{0, I_1}{\exists t > 0: N_t = 0}\probd{0, I_2}{\exists t > 0: N_t = 0} \mathds{1}_{Z_T = \delta_{(0,I_1)} + \delta_{(0,I_2)}}} \\
    = \ & \probd{a,i}{Z_T = 0} + \esperd{a,i}{\pi(a+T,1) \mathds{1}_{Z_T = \delta_{a+T,1}}} + \esperd{a,i}{\pi(0,I_1) \pi(0,I_2) \mathds{1}_{Z_T = \delta_{(0,I_1)} + \delta_{(0,I_2)}}}.
\intertext{Now, using Lemma \ref{lemma:probasvarias}, we obtain that}
\pi(a,i) = \ & \mathds{1}_{i=0} \ p \int_{0}^{+\infty} \beta_0(a+t) \exp \pars{ - \int_{0}^t \beta_0(a+u) du - \alpha t} dt \\
& + \mathds{1}_{i=0} \int_{0}^{+\infty} \pi(a+t,1) \  \alpha \exp \pars{ - \int_{0}^t \beta_0(a+u) du - \alpha t} dt \\
&+ \int_{0}^{+\infty} \beta_i(a+t) \exp \pars{ - \int_{0}^t \beta_i(a+u) du - (1-i) \alpha t} dt  \ \Big\{ \mathds{1}_{i=0} (1-p) \pi(0,0)^2 \\
&\qquad + \mathds{1}_{i=1} \pars{\gamma \pi(0,0) + (1-\gamma) \pi(0,1)}^2 \Big\}. \numberthis \label{eq:pi_as}
\end{align*}
In particular, doing $a=0$, we obtain 
\begin{empheq}[left=\empheqlbrace]{align}
    \pi_0 &= \pi(0,0) = p q' + (1-p) q' \pi_0^2 + \int_{0}^{+\infty} \pi(t,1) \  \alpha \exp \pars{ - \int_0^t \beta_0(u) du - \alpha t } dt, \label{eq:pi0_prelim} \\
    \pi_1 &= \pi(0,1) = (\gamma \pi_0 + (1-\gamma) \pi_1 )^2,
  \end{empheq}
with
\[
q' = \int_0^{+\infty} \beta_0(t) \exp \pars{- \int_0^t \beta_0(u) du - \alpha t} dt = 1-q,
\]
for $q$ defined by ~\eqref{eq:r0}, since
\[
q' + q = \int_{0}^{+\infty} \pars{ \alpha + \beta_0(t) } \exp \pars{ - \int_0^t \beta_0(u) du - \alpha t } dt
= \probd{0,0}{T < +\infty} = 1.
\]
Moreover, from ~\eqref{eq:pi_as} we have that for all $t \geq 0$
\begin{align*}
\pi(t,1) &= \pars{\gamma  \pi_0 + (1-\gamma) \pi_1}^2 \int_{0}^{+\infty} \beta_1(t+u) \exp \pars{ - \int_{0}^u \beta_1(t+w) dw} du  \\
&= \pars{\gamma \pi_0 + (1-\gamma) \pi_1}^2 \pars{1 - \exp\pars{- \int_0^{+\infty} \beta_1(t+u) du }}. \\
\intertext{Hence using the integrabilty Assumptions \ref{ass:assumptions}~\ref{eq:integrabilityBeta}, we obtain for all $t \geq 0$}
\pi(t,1) &= \pars{\gamma \pi_0 + (1-\gamma) \pi_1}^2 = \pi_1,
\end{align*}
which gives immediately ~\eqref{eq:pi1a}. Finally, after injecting these results back in ~\eqref{eq:pi0_prelim} we get the system \eqref{eq:pi0}-\eqref{eq:pi1}:
\begin{empheq}[left=\empheqlbrace]{align*}
 \pi_0 &= (1 - q) p + (1 - q) (1-p) \pi_0^2 + q \pi_1  \\
    \pi_1 &= (\gamma \pi_0 + (1-\gamma) \pi_1 )^2.
 \end{empheq}
Analogously, injecting these results back in ~\eqref{eq:pi_as} with $i=0$ we get ~\eqref{eq:pi0a} for all $a \geq 0$.

Next, we prove that $(\pi_0, \pi_1)$ is the minimal solution of this system. Set $(T_n)_{n \in \nn}$ the jump times of $Z$, with $T_0 = 0$, so in our previous notation $T_1 = T$. Define the extinction probabilities at the $n$-th jump by
\[
\pi^{(n)}(a,i) = \probd{a,i}{N_{T_n} = 0}.
\]
Therefore
\[
\lim_{n \to +\infty} \pi^{(n)}(a,i)  = \probd{a,i}{\bigcup_{n \in \nn} \set{N_{T_n} = 0} } = \probd{a,i}{\exists t > 0 : N_t = 0} = \pi(a,i).
\]

Now, suppose that we have some positive real solution $\tilde \pi$ of \eqref{eq:pi0}-\eqref{eq:pi1}. Then for both $i \in \{0,1\}$, $\tilde \pi_i \geq \pi^{(0)}(0,i) = 0$. We now show inductively that the same is verified for each $n \in \nn_*$ and in the limit $n \to +\infty$. As before, we condition with respect to the first jump and use the strong Markov property to obtain the following recursive equation
\begin{align*}
    \pi^{(n)}(a,i)  = \ & \probd{a,i}{Z_T = 0} + \esperd{a,i}{\probd{a+T, 1}{N_{T_{n-1}} = 0} \mathds{1}_{Z_T = \delta_{a+T,1}}} \\
    &+ \esperd{a,i}{\mathbb{P}_{\delta_{(0,I_1)} + \delta_{(0,I_2)}} \pars{ N_{T_{n-1}} = 0} \mathds{1}_{Z_T = \delta_{(0,I_1)} + \delta_{(0,I_2)}}} \\
    = \ & \pi^{(1)}(a,i) + \esperd{a,i}{\pi^{(n-1)}(a+T,1) \mathds{1}_{Z_T = \delta_{a+T,1}}} \\
    &+ \esperd{a,i}{\pi^{(n-1)}(0,I_1) \pi^{(n-1)}(0,I_2)  \mathds{1}_{Z_T = \delta_{(0,I_1)} + \delta_{(0,I_2)}}} , \qquad n \in \nn_*
\end{align*}
where we have again
\[
\pi^{(1)}(a,i) = \mathds{1}_{i=0} \ p \int_{0}^{+\infty} \beta_0(a+t) \exp \pars{ - \int_{0}^t \beta_0(a+u) du - \alpha t} dt .
\]
Suppose that $\pi^{(n-1)} \leq \tilde \pi$. Then 
\begin{align*}
    \pi^{(n)}(a,i)  
    \leq \ & \pi^{(1)}(a,i) + \esperd{a,i}{\tilde \pi(T,1) \mathds{1}_{Z_T = \delta_{T,1}}}  + \esperd{a,i}{\tilde \pi(0,I_1) \tilde \pi(0,I_2)  \mathds{1}_{Z_T = \delta_{(0,I_1)} + \delta_{(0,I_2)}}} \\
    = \ & \tilde \pi(a,i),
\end{align*}
since $\tilde \pi$ is a solution of ~\eqref{eq:pi_as}. Therefore, by induction, for all $n \in \nn$, $\pi^{(n)} \leq \tilde \pi$. Moreover, since $\set{\pi^{(n)}(a,i)}_{n \in \nn}$ is a sequence of probabilities of monotonic increasing events, we pass to the limit and conclude that $\pi \leq \tilde \pi$. Finally, notice that $\pi_0 = \pi_1 = 1$ is always an admissible solution, therefore the extinction probability are well contained in $[0,1]$.
\end{proof}

\subsection{Proof of Theorem \ref{prop:survival}}
\begin{proof}
First, note that ~\eqref{eq:pi0} and \eqref{eq:pi1} define two parabolic curves in the plane $(\pi_0, \pi_1)$ which intersect at least at the point $(1,1)$, since $\pi_0=\pi_1=1$ is always solution of the system. The proof consists in showing that other intersection occurs in the unit square $[0,1[ \times [0,1[$ if and only if Condition \eqref{eq:extinctionCondition} is verified. As in the classical characterisation of the extinction probability in Galton-Watson branching processes, this property can be obtained as a consequence of the value of the derivatives of the curves at the intersection point $(1,1)$. 

Note that the parametric curve defined by ~\eqref{eq:pi0} is a concave parabola whose intercept is located at $\pi_1 = - p \frac{1-q}{q} < 0$ and whose derivative is given by 
\[
\frac{d \pi_1 }{d \pi_0} (\pi_0) = \frac{1 - 2(1-p)(1-q)\pi_0}{q}.
\]
In particular, the derivative in $(1,1)$ equals
\[
\frac{d \pi_1 }{d \pi_0} (1) = \frac{1 - 2(1-p)(1-q)}{q}.
\]
Note that the curve defined by ~\eqref{eq:pi0} admits two solutions at $\pi_0 = 1$. However, using the implicit function theorem around $(1,1)$, we obtain a locally well defined function such that, by the implicit differentiation of ~\eqref{eq:pi1}, it has derivative
\[
\frac{d \pi_1 }{d \pi_0} (\pi_0) = 2 \pars {\gamma \pi_0 + (1-\gamma) \pi_1(\pi_0) } \pars{\gamma + (1-\gamma) \frac{d \pi_1 }{d \pi_0} (\pi_0) },
\]
and therefore at $(1,1)$ we have
\[
\frac{d \pi_1 }{d \pi_0} (1) = \frac{2\gamma }{2 \gamma - 1}.
\]
We can show that there is a second solution $\bar \pi_1$ of ~\eqref{eq:pi1} at $\pi_0 = 1$ comprised strictly between 0 and 1 if and only if $0 < \gamma < 1/2$. Moreover, ~\eqref{eq:pi1} also admits $(0,0)$ as solution. Thus, the trace of the curve described by ~\eqref{eq:pi1} connects $(0,0)$ to $(1,1)$ if $\gamma \geq 1/2$, or to $(1, \bar \pi_1)$, if $\gamma < 1/2$. Meanwhile, the curve of ~\eqref{eq:pi0} connects the negative ordinate $(0,- p \frac{1-q}{q})$ with $(1,1)$. Therefore, no intersection other than $(1,1)$ can occur inside the unit square if and only if $\gamma \geq 1/2$, and ~\eqref{eq:pi0} arrives at $(1,1)$ with non-negative derivative and whose value is at least as much as the value of the derivative of ~\eqref{eq:pi1}. Otherwise, by the continuity and strict monotonicity of the curves we would have some other intersection point below $(1,1)$ (see Fig. \ref{fig:pi_parab}). This is then:
\[
 \gamma  \geq 1/2  \quad \textrm{,} \quad 1 - 2(1-p)(1-q) \geq  0  \quad \textrm{and} \quad \frac{1 - 2(1-p)(1-q)}{q} \geq \frac{2\gamma }{2 \gamma - 1},
\]
which gives finally, for $p \neq 1/2$
\[
\gamma \geq \frac{1}{2} \pars {1 + \frac{q}{(2p-1)(1-q)}}.
\]
In particular, the condition cannot be verified if $p < 1/2$. In the case $p = 1/2$, the previous conditions cannot be verified either. Thus finally, extinction occurs almost surely if and only if $p > 1/2$ and condition \eqref{eq:extinctionCondition} is verified, which gives the result.

\begin{figure}
    \centering

\begin{tikzpicture}[scale=0.6]
\begin{axis}[
    axis lines = middle,
    xlabel={$\pi_0$},
    ylabel={$\pi_1$},
    xmin=-0.25, xmax=1.2, xtick={1}, xticklabels={1},
    ymin=-0.05, ymax=1.2, ytick={1}, yticklabels={1}]
\addplot[red, ultra thick, domain=0:1.1] {(x - 0.36 - 0.24 * x^2)/(0.4)};
\addplot[blue, ultra thick, domain=0:1.1] {1.0204 - 0.20408 * (25 - 21 * x)^(0.5) - 0.42857 * x};
\addplot[blue, ultra thick, domain=0:1.1] {1.0204 + 0.20408 * (25 - 21 * x)^(0.5) - 0.42857 * x};
\addplot[dotted,mark=none] coordinates {(0, 1) (1, 1)};
\addplot[dotted,mark=none] coordinates {(1, 0) (1, 1)};
\end{axis}
\end{tikzpicture}
\begin{tikzpicture}[scale=0.6]
\begin{axis}[
    axis lines = middle,
    xlabel={$\pi_0$},
    ylabel={$\pi_1$},
    xmin=-0.25, xmax=1.2, xtick={1}, xticklabels={1},
    ymin=-0.05, ymax=1.2, ytick={1}, yticklabels={1}]
\addplot[red, ultra thick, domain=0:1.1] {(x - 0.36 - 0.24 * x^2)/(0.4)};
\addplot[blue, ultra thick, domain=0:1.1] {3.125 - 0.625 * (25 - 24 * x)^(0.5) - 1.5 * x};
\addplot[dotted,mark=none] coordinates {(0, 1) (1, 1)};
\addplot[dotted,mark=none] coordinates {(1, 0) (1, 1)};
\end{axis}
\end{tikzpicture}
\begin{tikzpicture}[scale=0.6]
\begin{axis}[
    axis lines = middle,
    xlabel={$\pi_0$},
    ylabel={$\pi_1$},
    xmin=-0.25, xmax=1.2, xtick={1}, xticklabels={1},
    ymin=-0.05, ymax=1.2, ytick={1}, yticklabels={1}]
\addplot[red, ultra thick, domain=0:1.1] {(x - 0.98 * 0.6 - 0.98 * 0.4 * x^2)/(0.02)};
\addplot[blue, ultra thick, domain=0:1.1] {3.125 - 0.625 * (25 - 24 * x)^(0.5) - 1.5 * x};
\addplot[dotted,mark=none] coordinates {(0, 1) (1, 1)};
\addplot[dotted,mark=none] coordinates {(1, 0) (1, 1)};
\end{axis}
\end{tikzpicture}
    \caption{Parabolic curves defined by ~\eqref{eq:pi0} (red) and ~\eqref{eq:pi1} (blue) for $p = 0.6$. In the first case we have $\gamma = 0.3 < 1/2$ and $q = 0.4$. In the second case we have $\gamma = 0.6 > 1/2$ and $q = 0.4$. In the third case we have $\gamma = 0.6$ and $q = 0.02$, so the condition \eqref{eq:extinctionCondition} is violated and the only intersection in the unit square is (1,1).}
    \label{fig:pi_parab}
\end{figure}
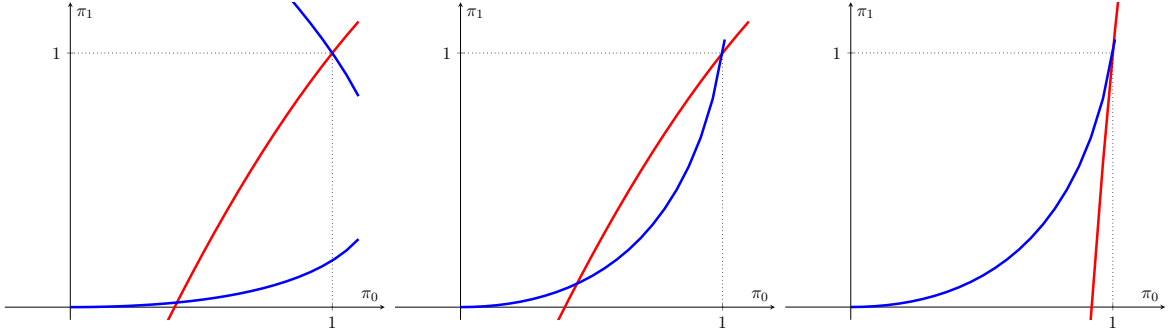

\end{proof}

\section{Proofs of Section \ref{sec:MtDynamics}}
\label{sec:proofs3}

\subsection{Proof of Proposition \ref{prop:FE}}
\begin{proof}
Let $(a,i) \in \rr_+ \times \set{0,1} $, $f \in \mathcal{B}_b(\rr_+ \times \set{0,1})$. Conditioning on the first jump event, we develop $M_t f(a,i)$ using Lemma \ref{lemma:probasvarias} to compute the expectations at each jump case, and the strong Markov property, similarly as we did in the Proof of Proposition 3.5. We obtain the following Duhamel's representation:
    \begin{equation}
    \begin{split}
        M_t f(a,i) =& f(a+t,i) \psi_i(a,a+t)  
        + (1-i) \alpha \int_0^t M_{t-s} f(a+s, 1) \psi_i(a,a+s) ds  \\
        &+ 2 \int_0^t \beta_i(a+s) \psi_i(a,a+s) \Big\{ (i-1)(1-p)M_{t-s}f(0,0) \\
        & \quad  i \pars{ \gamma M_{t-s}f(0,0) + (1-\gamma) M_{t-s}f(0,1) } \Big\} ds,
    \end{split}
    \label{eq:duhamel}
    \end{equation}
    where the first term of the RHS corresponds to the deterministic evolution when there are no events before time $t$, the second term corresponds to the case when the first jump is a type switch, and the third one to the case when the first jump is a division.  
    
    We iterate Duhamel's formula once more for the second line of the RHS, using ~\eqref{eq:duhamel} with $i=1$ and $a=a+s$, and then re-injecting the obtained result. We obtain therefore a representation that uses only the semigroup valuated at initial age 0. Rearranging the terms for $i = 0$ and $i = 1$ in a vector, we obtain ~\eqref{eq:FE}.
\end{proof}

\subsection{Proof of Lemma \ref{lemma:rhoKinf}}

\begin{proof}
We start by the direct integration of ~\eqref{eq:K} to obtain $\mathbf{K}_\infty$. The integration of the composition $\psi_0 \star \psi_1$ requires some attention. First, we can easily remark that for all $s < a$, $\psi_1(s,a) = \psi_1(0,a)/\psi_1(0,s)$, and then using Fubini's theorem we obtain that
\begin{align*}
\int_0^{+\infty} \alpha \beta_1(a) \psi_0 \star \psi_1(0,a) da 
&= \int_0^{+\infty} \alpha \beta_1(a) \int_0^a \psi_0(0,s) \psi_1(s,a) ds da \\
&= \int_0^{+\infty} \alpha  \frac{\psi_0(0,s)}{\psi_1(0,s)} \pars{ \int_s^{+\infty} \beta_1(a)  \psi_1(0,a) da } ds.
\end{align*}
But, by definition of the survival function, we know that $\psi_1(0,s) = \int_{s}^{+\infty} \beta_1(a) \psi_1(0,a) da$. Thus, using the definition of $q$ introduced by ~\eqref{eq:r0} we have that
\begin{align*}
   \int_0^{+\infty} \alpha \beta_1(a) \psi_0 \star \psi_1(0,a) da 
&=   \int_0^{+\infty} \alpha  \psi_0(0,s) ds = q.
\end{align*}
Thus, we obtain finally
\[
\mathbf{K}_\infty = 2 \begin{bmatrix}
(1-p) (1-q) + \gamma q & (1 - \gamma) q \\
\gamma & 1 - \gamma 
\end{bmatrix},
\]
which is a $2 \times 2$ matrix of non-negative terms. In particular, it has non-negative discriminant and non-negative eigenvalues which are given by
\[
\frac{1}{2} \pars{ \textrm{tr} ( \mathbf{K}_\infty  ) \pm \sqrt{ \textrm{tr} ( \mathbf{K}_\infty  )^2 - 4 \det ( \mathbf{K}_\infty  ) }}.
\]
with
$\
\textrm{tr} ( \mathbf{K}_\infty  ) =  2 \pars { 1+ (1-p-\gamma)(1-q) } \ $ and $\ 
\det ( \mathbf{K}_\infty  ) = 4(1-p)(1-q)(1-\gamma)$.
In particular, the largest eigenvalue is larger than 1 if and only if
\begin{align}
\sqrt{ \textrm{tr} ( \mathbf{K}_\infty  )^2 - 4 \det ( \mathbf{K}_\infty  ) } > 2 - \textrm{tr} ( \mathbf{K}_\infty  ) = (\gamma - (1-p))(1-q).
\label{eq:ineqRHO}
\end{align}
1) Case $\gamma \leq 1-p:$  Since the RHS of \eqref{eq:ineqRHO} is non-positive, \eqref{eq:ineqRHO} is trivially verified and we have immediately $\rho(\mathbf{K}_\infty) > 1$.\\
 2)  Case $\gamma > 1-p:$ Since the RHS of \eqref{eq:ineqRHO} is positive, by taking squares we have that \eqref{eq:ineqRHO} is equivalent to the following inequalities:
    \begin{align*}
        \textrm{tr} ( \mathbf{K}_\infty  )^2 - 4 \det ( \mathbf{K}_\infty  ) &> 4 + \textrm{tr} ( \mathbf{K}_\infty  )^2 - 4 \textrm{tr} ( \mathbf{K}_\infty  ) \\
     \Longleftrightarrow     \textrm{tr} ( \mathbf{K}_\infty  ) -  \det ( \mathbf{K}_\infty  ) -1  &>  0 \\
    \Longleftrightarrow      2 + 2(1-p-\gamma)(1-q) - 4(1-p)(1-q)(1-\gamma) - 1 &>0 \\
       \Longleftrightarrow   2(1-q)((1-\gamma)(1-2(1-p)) - p) + 1 &> 0 .
    \end{align*}
    Then, isolating the value of $\gamma$ we obtain
    \begin{equation}
    (2p-1)\gamma  < \frac{(2p-1)(1-q)+q}{2(1-q)}.
    \label{eq:lastIneq}
    \end{equation}
    We study the case $p > 1/2 $ and $p \leq 1/2$ separately. If $p > 1/2 $, the factor $2p - 1$ is positive, and then dividing \eqref{eq:lastIneq} by $2p-1$ we obtain directly condition \eqref{eq:extinctionCondition}:
    \[
    \gamma < \frac{1}{2} \pars{1 + \frac{q}{(2p-1)(1-q)}}.
    \]
    If $p= 1/2$, then \eqref{eq:lastIneq} is trivially verified. Finally, if $p < 1/2$, the factor $2p-1$ is negative and dividing \eqref{eq:lastIneq} by $2p-1$ we obtain
    \[
    \gamma > \frac{1}{2} \pars{1 + \frac{q}{(2p-1)(1-q)}}.
    \]
    However, since we are under the assumption $\gamma > 1-p$ and $p < 1/2$, we have $\gamma > 1/2$ and the inequality above is verified a fortiori.

Summarising,
\begin{align*}
\rho ( \mathbf{K}_\infty ) > 1 \iff \set{\gamma \leq 1 - p} &\cup  \pars { \set{\gamma > 1 - p} \cap \set{p > 1/2} \cap \{ \eqref{eq:extinctionCondition} \textrm{ is true} \} } \\
&\cup  \pars { \set{\gamma > 1 - p} \cap \set{p \leq 1/2} } \\
\iff \set{p \leq 1/2} &\cup 
\pars { \set{p > 1/2} \cap \{ \eqref{eq:extinctionCondition} \textrm{ is true} \} },
\end{align*}
which is exactly the condition ensuring survival with positive probability in Theorem \ref{prop:survival}.
\end{proof}

\subsection{Proof of Lemma \ref{lemma:uniquenessLambda}}

\begin{proof}
Notice that for all $\lambda \in \rr$ the matrix $\mathbf{F}(\lambda)$ is a $2 \times 2$ matrix with all strictly positive entries. In particular, this ensures irreducibility and by Perron-Frobenius Theorem we have the existence, for all fixed $\lambda \in \rr$, of a unique triplet of eigenelements $(\mu (\lambda), \mathbf{n}(\lambda), \mathbf{h}(\lambda))$ such that
\begin{empheq}[left=\empheqlbrace]{align*}
\mathbf{F}(\lambda) \mathbf{h}(\lambda) &= \mu (\lambda) \mathbf{h}(\lambda) \\
\mathbf{n}(\lambda) ^\top \mathbf{F}(\lambda)  &= \mu (\lambda) \mathbf{n}(\lambda) ^\top \\
\mathbf{n}(\lambda) ^\top\mathbf{h}(\lambda) &= 1 \\
\mathbf{n}(\lambda) ^\top (1,1) &= 1.
\end{empheq}
We start by the sufficiency direction of the equivalence. To show that there is a unique $\lambda > 0$ such that $\rho(\mathbf{F}(\lambda)) = 1$ we will use a classical monotonicity argument. First, we notice that $\mathbf{F}(0) = \mathbf{K}_\infty$, which under the survival conditions and thanks to Lemma \ref{lemma:rhoKinf} has spectral radius $\rho(\mathbf{F}(0)) = \mu(0) > 1$. Second, notice that as $\lambda \to +\infty$, $e^{-\lambda s} \mathbf{K}(0,s)$ decreases coordinate by coordinate to the null matrix. Thus, we have by monotone convergence that
$\mathbf{F}(+\infty) = \boldsymbol{0}_{2 \times 2}$ and therefore $\lim_{\lambda \to +\infty}\mu(\lambda) = 0$. It remains to show that $\lambda \mapsto \mu(\lambda)$ is a decreasing continuous function. This is a classical property that comes from Lemma \ref{lemma:ineqRhoAB}. The monotonicity is given by the assertion (i) of the latter, whilst the continuity comes from the estimation (ii) and the continuity of $\lambda \mapsto e^{-\lambda s}$. Finally, this implies that there exists a unique $\lambda^* > 0$ such that $\mu(\lambda^*) = 1$.

The necessity direction of the equivalence follows easily from the previous remarks. Since $\lambda \mapsto \mu(\lambda)$ is a continuous decreasing function, if $\lambda^* \in \rr$ such that $\mu(\lambda^*) = 1$ is strictly positive, then $\rho(\mathbf{F}(0)) = \mu(0) > 1$. Finally, by Lemma \ref{lemma:rhoKinf}, this is equivalent to verify the survival condition \eqref{eq:extinctionCondition}. 
\end{proof}

\subsection{Proof of Proposition \ref{prop:main}}
\begin{proof}
The proof follows classical arguments, presented for example in Section 4.3 of \cite{webb1985theory}. The first part of the proof consists in showing the existence of a spectral gap, this is, of a positive constant $\omega$ such that for the value of $\lambda$ defined by ~\eqref{eq:malthusianParam} we have
\[
\max \set{- \underline{b}, \sup_{z \in \sigma(\mathcal{Q}) \setminus \sigma_{ess}(\mathcal{Q}) \setminus \set{\lambda} } \textrm{Re}(z) } < \omega < \lambda.
\]
To do so, we show in Appendix \ref{app:mainproof} that the \emph{growth bound} $\omega_1(\mathbf{M}_t)$ associated with the measure of non-compactness of the semigroup $\mathbf{M}_t$ (see Definition \ref{def:noncompactnessmeasure}) is bounded by $ - \underline{b}$. Then, by Theorem 4.6 of \cite{webb1985theory} we have that the essential spectrum of $\mathcal{Q}$ must be contained within $\set{z \in \mathbb{C}: \textrm{Re}(z) \leq - \underline{b}}$. Therefore, for any $\lambda \in \mathbb{C}$ such that $\textrm{Re}(\lambda) > - \underline{b}$ which is a root of the characteristic equation \eqref{eq:malthusianParam}, we have that $\lambda$ is an eigenvalue of $\mathcal{Q}$. This is, there exists a non-zero $\mathbf{h} \in (L^1(\rr_+))^2$ such that $\mathcal{Q} \mathbf{h} = \lambda \mathbf{h}$. 

Indeed, suppose that the pair $(\lambda, \mathbf{h})$ is solution to $\mathcal{Q} \mathbf{h} = \lambda \mathbf{h}$, then $\mathbf{h}' \in (L^1(\rr_+))^2$ and is given almost everywhere by
\[
\mathbf{h}'(a) = (\lambda \mathbf{I} + \mathbf{D}(a)) \mathbf{h}(a) - 2 \mathbf{B}(a) \mathbf{h}(0).
\]
Then, we notice that
\begin{align}
\partial_a \pars{ e^{-\lambda a } \boldsymbol{\Psi}(0,a) \mathbf{h}(a)} &= \partial_a \pars{e^{-\lambda a } \boldsymbol{\Psi}(0,a)}  \mathbf{h}(a) + e^{-\lambda a } \boldsymbol{\Psi}(0,a) \mathbf{h}'(a) \nonumber \\
&= - e^{-\lambda a }   \boldsymbol{\Psi}(0,a) (\lambda \mathbf{I} + \mathbf{D}(a)) \mathbf{h}(a) +  e^{-\lambda a } \boldsymbol{\Psi}(0,a) (\lambda \mathbf{I} + \mathbf{D}(a)) \mathbf{h}(a) \nonumber \\
&\qquad - 2 e^{-\lambda a } \boldsymbol{\Psi}(0,a) \mathbf{B}(a) \mathbf{h}(0) \nonumber \\
&= - 2 e^{-\lambda a } \mathbf{K}(0,a) \mathbf{h}(0).\label{eq:diffCharacteristic}
\end{align}
Therefore, since $\mathbf{h} \in (L^1(\rr_+))^2$ and $\lim_{a \to +\infty} \boldsymbol{\Psi}(0,a) = \boldsymbol{0}$, we have
\begin{align*}
    \mathbf{h}(0) = - \int_0^{+\infty} \partial_a \pars{ e^{-\lambda a } \boldsymbol{\Psi}(0,a) \mathbf{h}(a)} da = 2 \int_0^{+\infty} e^{-\lambda a } \mathbf{K}(0,a) \mathbf{h}(0) \ da  = \mathbf{F}(\lambda) \mathbf{h}(0).
\end{align*}
This linear equation has a non-trivial solution $\mathbf{h}(0)$ if and only if $\det \pars{\mathbf{F}(\lambda) - \mathbf{I}} = 0$. Conversely if $\lambda$ is the largest root of $\det \pars{\mathbf{F}(\lambda) - \mathbf{I}} = 0$ and is simple and associated with some eigenvector $\mathbf{h}(0) \in \rr_+^2$, we have that
\begin{align}
\mathbf{h}(a) &= e^{\lambda a } \boldsymbol{\Psi}^{-1}(0,a) \pars{ \mathbf{I} -  2 \int_0^a e^{- \lambda s } \boldsymbol{\Psi}(0,s) \mathbf{B}(s) \ ds } \mathbf{h}(0) 
\label{eq:ha_h0_1}  \\
&= 2 e^{\lambda a} \int_{a}^{+\infty} e^{-\lambda s} \boldsymbol{\Psi}(a,s) \mathbf{B}(s) \ ds \mathbf{h}(0)
\label{eq:ha_h0} 
\end{align}
is a $(L^1(\rr_+))^2$ solution to $\mathcal{Q} \mathbf{h} = \lambda \mathbf{h}$ (both representations will be useful in the sequel). Finally, Lemma \ref{lemma:uniquenessLambda} allows us to conclude for the simplicity of the eigenvalue $\lambda$.

Then, thanks to Proposition 4.65 of \cite{webb1985theory}, we can identify $\ker \pars{\lambda \mathbf{I} - \mathcal{Q}}$ to the image of the projection operator $\mathbf{P} : (L^1(\rr_+))^2 \to (L^1(\rr_+))^2$ given by the resolvent
\[
\mathbf{P} \mathbf{f}(a) := \frac{1}{2 \pi i } \oint_{C_\lambda} \pars{z \mathbf{I} - \mathcal{Q}}^{-1} \mathbf{f}(a) \ dz,
\]
where $C_\lambda$ is a closed counterclockwise oriented curve of the complex plane enclosing $\lambda$ but no other point of the spectrum of $\mathcal{Q}$. 

Then, since $\lambda$ is a simple root of ~\eqref{eq:malthusianParam}, we have a unique pair $(\lambda, \mathbf{h})$ (up to normalisation of $\mathbf{h}$) such that $\mathcal{Q} \mathbf{h} = \lambda \mathbf{h}$, and so $\mathbf{P}$ is of rank 1. Therefore, for all $\mathbf{f}$ we can write $\mathbf{P} \mathbf{f}(a) = \bar \nu(\mathbf{f}) \mathbf{h}(a)$, where $\bar \nu(\mathbf{f})$ is a normalisation constant depending on $\mathbf{f}$.

Moreover, by the linearity of $\mathbf{P}$, we have that $\boldsymbol{\nu} : (C_c(\rr_+))^2 \to \rr_+^2$ given for all $\mathbf{f}$ of the form $\mathbf{f} = (f(\cdot, 0), f(\cdot, 1))$ by
\[
\boldsymbol{\nu}(\mathbf{f}) = (\bar \nu(f(\cdot,0),0), \bar \nu(0,f(\cdot,1)))
\]
is a linear application. Therefore, Riesz–Markov–Kakutani representation Theorem allows us to write $\bar \nu(\mathbf{f}) = \int_0^{+\infty} \mathbf{f}(a)^\top \boldsymbol{\nu}(da)$ where $\boldsymbol{\nu} \in (\mathcal{M}(\rr_+))^2$ is a (vector) positive Radon measure. Furthermore, by duality we can conclude that this measure $\boldsymbol{\nu}$ is the limiting distribution of the population ages in the sense that for all initial distribution $\boldsymbol{\mu}$ and all $\mathbf{f} \in (L^1(\rr_+))^2 $ we have
\begin{align*}
e^{-\lambda t} \ap{\boldsymbol{\mu} \mathbf{M}_t }{\mathbf{f}} = \int_0^{+\infty} \mathbf{f}(a) \ e^{-\lambda t} \boldsymbol{\mu} \mathbf{M}_t (da) &= \int_0^{+\infty} e^{-\lambda t}  \mathbf{M}_t \mathbf{f}(a)  \ \boldsymbol{\mu}(da) \\
& \underset{t \to +\infty}{\longrightarrow}  \pars{ \int_0^{+\infty} \mathbf{h}(a)  \boldsymbol{\mu}(da) }  \pars{ \int_0^{+\infty} \mathbf{f}(a)  \boldsymbol{\nu}(da)}.
\end{align*}

Finally, by Proposition 4.15 of \cite{webb1985theory} we have
\[
\norm{ \mathbf{M}_t \mathbf{f} - \mathbf{P} \mathbf{f}}_1 \leq c e^{\omega t } \norm{\mathbf{f} - \mathbf{P} \mathbf{f} }_1,
\]
from where we conclude the exponential rate of convergence.
\end{proof}

\subsection{Proof of Remark \ref{rmk:boundLambda}}
\begin{proof}
    Notice that taking as test function $\mathbf{f} \equiv \boldsymbol{1} = (1,1)$, we get 
    \begin{align*}
    \mathbf{M}_t \boldsymbol{1}(a) &= \boldsymbol{1} + \int_0^t \mathbf{M}_s \pars{\mathcal{Q} \boldsymbol{1}}(a) ds = \boldsymbol{1} + \int_0^t \mathbf{M}_s \pars{(2 \mathbf{B}(\cdot) - \mathbf{D}(\cdot))\boldsymbol{1}}(a) ds \\
    &= \boldsymbol{1} + \int_0^t  \begin{bmatrix}
        \mathbf{M}_s (\beta_0, 0) (a) ^\top \\
        \mathbf{M}_s (0, \beta_1) (a) ^\top
    \end{bmatrix} ds \begin{pmatrix}
        1 - 2p \\ 1
    \end{pmatrix} ds.
\end{align*}
Therefore for both coordinates $i \in \{0,1\}$ we have
    \[ M_t 1(a,i) \leq 1 + \bar b \int_0^t  \mathbf{M}_s 1(a,i) ds.\]
    Hence, by Grönwall 's inequality, for all $a \geq 0$, 
   $\ 
  M_t 1(a,u) \leq e^{\bar b t} $.
    Therefore, by duality, and since $\ \boldsymbol{\nu} \mathbf{M}_t (da) = e^{\lambda t} \boldsymbol{\nu} (da)$,
    \[
e^{\lambda t} \ap{ \boldsymbol{\nu}}{\boldsymbol{1}} = \ap{\boldsymbol{\nu}\mathbf{M}_t }{\boldsymbol{1}} =  \ap{\boldsymbol{\nu}}{\mathbf{M}_t \boldsymbol{1}} \leq e^{\bar b t} \ap{ \boldsymbol{\nu}}{\boldsymbol{1}}.
    \]
Finally, $\ap{ \boldsymbol{\nu}}{\boldsymbol{1}} = 1 \neq 0$ and the result follows.  
\end{proof}

\section{Proofs of Section \ref{sec:dlambda}}
\label{sec:proofs4}

\subsection{Proof of Lemma \ref{lemma:continuity}}

\begin{proof}
Since $\mathbf{h}$ is solution to the eigenproblem $\mathcal{Q} \mathbf{h} = \lambda \mathbf{h}$, representation \eqref{eq:ha_h0} gives us
\begin{equation*}
        \mathbf{h}(a) =  2  \int_{0}^{+\infty} e^{- \lambda s } \boldsymbol{\Psi}(a,a+s) \mathbf{B}(a+s)  \mathbf{h}(0) ds .
    \end{equation*}
We deduce the result by Dominated Convergence Theorem in each coordinate of the expression above. Set $\boldsymbol{\ell}(a,s) := 2 e^{- \lambda s } \boldsymbol{\Psi}(a,a+s) \mathbf{B}(s)  \mathbf{h}(0)$. Note that only matrix $\boldsymbol{\Psi}$ (defined in \eqref{eq:PsiMatrix}, and with $\psi_0, \psi_1, \psi_0 \star \psi_1$ in \eqref{psi}-\eqref{convol}) depends on $a$. First, we know by construction of the Perron eigenelements that for all $a \geq 0$, $\boldsymbol{\ell}(a,\cdot) \in (L^1(\rr_+))^2$. Indeed:
\[
\norm{\boldsymbol{\ell}(a,s)}_{\infty} \leq 2 \bar b \norm{\mathbf{h}(0)}_{\infty} e^{- \lambda s} \in L^1(\rr_+, ds).
\]
Moreover, since we assume additionally that $\beta_0, \beta_1$ are continuous, we deduce that for all $s \geq 0$, $\boldsymbol{\ell}(\cdot,s) \in C^1\pars{\rr_+, \rr_+^2}$. Indeed, for all $a \geq 0, s \geq 0$, we have explicitly
     \begin{align*}
     \partial_{a} \psi_{0}(a,a+s) &= (\beta_0(a) - \beta_0(a+s)) \psi_0(a,a+s),   \\
     \partial_{a} \psi_{1}(a,a+s) &= (\beta_1(a) - \beta_1(a+s)) \psi_1(a,a+s), \\
     \partial_{a} (\psi_{0}\star \psi_1) (a,a+s) &= (\alpha + \beta_0(a) - \beta_1(a+s))\psi_{0}\star \psi_1 (a,a+s) + \psi_0(a,a+s) - \psi_1(a,a+s),
     \end{align*}
where all the terms at the RHS are well defined and continuous for all $a \geq 0$. Finally, all these terms are uniformly bounded by $\bar{b} + \alpha + 1$, and therefore
\[
\norm{ \partial_{a} \boldsymbol{\ell}(a,s)}_{\infty} \leq  4 \bar{b} (\bar{b} + \alpha + 1) \norm{\mathbf{h}(0)}_{\infty} e^{-\lambda s} \in L^1(\rr_+, ds) .
\]
Hence, by the Dominated Convergence Theorem, $\mathbf{h} \in C^1(\rr_+ , \; \rr_+^2)$.
\end{proof}

\subsection{Proof of Lemma \ref{lemma:continuity2}}
\begin{proof}
    Writing the dependencies on $\alpha$ and $\gamma$ explicitly, we express $\mathbf{h}_{\alpha,\gamma}(a) $  for all $a \geq 0$ as
\begin{equation}
        \mathbf{h}_{\alpha, \gamma}(a) =  2  \int_{0}^{+\infty} e^{- \lambda_{\alpha, \gamma} s } \boldsymbol{\Psi}_{\alpha}(a,a+s) \mathbf{B}_{\gamma}(a+s)  \mathbf{h}_{\alpha, \gamma}(0) ds .
        \label{eq:h_alphaGamma}
    \end{equation}

Again by Dominated Convergence Theorem, it suffices to prove the continuity and domination by an integrable function of both coordinates of
\[
(\alpha, \gamma) \mapsto \boldsymbol{\ell}_a(\alpha, \gamma) := 2  e^{- \lambda_{\alpha, \gamma} s } \boldsymbol{\Psi}_{\alpha}(a,a+s) \mathbf{B}_{\gamma}(a+s)  \mathbf{h}_{\alpha, \gamma}(0) \; \in \rr_+^2.
\]
Hence, we show first that $\norm{\mathbf{h}_{\alpha, \gamma}(0)}_{\infty}$ is uniformly bounded in every open neighborhood of $(\alpha, \gamma) \in \rr_+ \times (0,1)$. Using the normalisation condition $\ap{\boldsymbol{\nu}_{\alpha, \gamma}}{\mathbf{h}_{\alpha,\gamma}} = 1$ with $\boldsymbol{\nu}_{\alpha, \gamma}$ the left eigenmeasure, we have that for all $a \geq 0$,
\begin{equation*} 
1 = \ap{\boldsymbol{\nu}_{\alpha, \gamma}}{\mathbf{h}_{\alpha,\gamma}}  = \mathbf{u}_{\alpha,\gamma}^\top \mathbf{h}_{\alpha,\gamma}(0),
\end{equation*}
where
\begin{align*}
\mathbf{u}_{\alpha,\gamma} &= 2 \int_0^{+\infty} \pars{ \int_0^{+\infty} e^{- \lambda_{\alpha,\gamma} \tau } \mathbf{B}_{\gamma}(a+\tau)^\top \boldsymbol{\Psi}_{\alpha}(a,a+\tau)^\top   d \tau }  \boldsymbol{\nu}_{\alpha, \gamma}(a) \ da .
\end{align*}
Therefore, to bound each coordinate of $\mathbf{h}_{\alpha,\gamma}(0) \in \rr_+^2$ it suffices to  bound by below both coordinates of $\mathbf{u}_{\alpha,\gamma}$. Set $\mathbf{u}_{\alpha,\gamma} = (u_{\alpha,\gamma}^0, u_{\alpha,\gamma}^1)$, such that
\begin{align*}
u_{\alpha,\gamma}^0 &= 2 \iint_{\rr_+^2} e^{- \lambda_{\alpha,\gamma} \tau } \boldsymbol{\nu}_{\alpha, \gamma}(a)^\top  \begin{pmatrix}
    (1-p) \beta_0(a+\tau) \psi_0(a+\tau) + \gamma \beta_1(a+\tau) \alpha \psi_0 \star \psi_1 (a, a+\tau) \\
    \gamma \beta_1(a+\tau) \psi_1(a+\tau)
\end{pmatrix} d \tau \ da, \\
u_{\alpha,\gamma}^1 &= 2  \iint_{\rr_+^2} e^{- \lambda_{\alpha,\gamma} \tau } \boldsymbol{\nu}_{\alpha, \gamma}(a)^\top  \begin{pmatrix}
    (1-\gamma) \beta_1(a+\tau) \alpha \psi_0 \star \psi_1 (a, a+\tau) \\
    (1-\gamma) \beta_1(a+\tau) \psi_1(a+\tau)
\end{pmatrix} \ d \tau \ da.
\end{align*}
By Remark \ref{rmk:boundLambda}, the value of $\lambda_{\alpha,\gamma}$ is bounded by $\bar b$ for all $(\alpha,\gamma)$. Then, using Assumptions \ref{ass:assumptions} to bound the division rates in matrix $\mathbf{B}$ and the survival functions in matrix $\boldsymbol{\Psi}$ we obtain :
\begin{align*}
u_{\alpha,\gamma}^0 &\geq 2 \gamma \int_0^{+\infty}  \pars{\int_{a_0 }^{A}  \underline{b} e^{-  2 \bar b \tau } ( 1 - e^{- \alpha \tau} ) d \tau } \boldsymbol{\nu}_{\alpha, \gamma}(a)^\top \boldsymbol{1} \ da,  \\
u_{\alpha,\gamma}^1 &\geq 2 (1- \gamma) \int_0^{+\infty}  \pars{\int_{a_0 }^{A}  \underline{b} e^{-  2 \bar b \tau } ( 1 - e^{- \alpha \tau} ) d \tau } \boldsymbol{\nu}_{\alpha, \gamma}(a)^\top \boldsymbol{1} \ da, 
\end{align*}
for some fixed arbitrary quantity $A > a_0$ with $a_0$ being given by (A3) in Assumptions \ref{ass:assumptions}. Then, the integral with respect to $\tau$ can be uniformly  bounded by below by some positive constant $\tilde c_{\alpha,A}$ dependent on $\alpha$ and the choice of $A$. Finally, by the normalisation condition $\ap{\boldsymbol{\nu}}{\boldsymbol{1}} = 1$, we obtain 
\begin{align*}
u_{\alpha,\gamma}^0 \geq  2 \gamma \tilde{c}_{\alpha,A} , \qquad 
u_{\alpha,\gamma}^1 \geq  2 (1-\gamma) \tilde{c}_{\alpha,A}, 
\end{align*}
and therefore,
\begin{align*}
h_{\alpha,\gamma}(0,0) \leq  \frac{1}{2 \gamma \tilde{c}_{\alpha,A}} , \qquad 
h_{\alpha,\gamma}(0,1) \leq  \frac{1}{2 (1-\gamma) \tilde{c}_{\alpha,A}}.
\end{align*}
Thus, for all fixed $a\geq 0$
\begin{align*}
\ell_{\alpha,\gamma}^0(a,s)  &\leq 2 \tilde{c}_{\alpha,A}^{-1}  \pars{
\frac{1-p}{\gamma} \beta_0(a+s) \psi_0(a,a+s) + \beta_1(a+s) \alpha \psi_0 \star \psi_1(a,a+s)  
} \; \in L^1(\rr_+, ds), \\
\ell_{\alpha,\gamma}^1(a,s)  &\leq 2 \tilde{c}_{\alpha,A}^{-1}  \beta_1(a+s) \psi_1(a,a+s) 
\; \in L^1(\rr_+, ds),
\end{align*}
which can be bound uniformly in any neighbourhood around $(\alpha, \gamma)$, whenever $\alpha, \gamma \notin \set{0,1}$. Finally, we show that $(\alpha, \gamma) \mapsto \boldsymbol{\ell}_a(\alpha, \gamma)$ is continuous for the uniform norm. Fix some couple $(\alpha, \gamma) \in \rr_+ \times (0,1)$ and let $(\alpha_n, \gamma_n)_{n \in \mathbb{N}}$ be some sequence converging to $(\alpha,\gamma)$ as $n \to +\infty$. We know already that $\mathbf{h}_{\alpha_n, \gamma_n}(0)$ and $\lambda_{\alpha_n,\gamma_n}$ are bounded for all $n$. Therefore we can extract some convergent subsequences with adherence values
\[
\lambda_{\alpha_{n_k}, \gamma_{n_k}} \to \lambda_{\infty} \geq 0, \qquad 
\mathbf{h}_{\alpha_{n_k}, \gamma_{n_k}}(0) \to \boldsymbol{\eta}_{\infty} \in \rr^2_+.
\]
Moreover, note that for all fixed vector $\mathbf{x} \in \rr^2_+$, the linear application $(\alpha, \gamma) \mapsto \mathbf{K}_{\alpha,\gamma}(a,a+s) \mathbf{x} = \boldsymbol{\Psi}_{\alpha}(a,a+s) \mathbf{B}_{\gamma}(a+s) \mathbf{x} $ is a continuous function of $(\alpha,\gamma)$ for the uniform norm. Therefore, we have that entry-wise
\[
 [\mathbf{K}_{\alpha_n,\gamma_n}(a,a+s)]_{i,j} \to [\mathbf{K}_{\alpha,\gamma}(a,a+s)]_{i,j}.
\]
To conclude, we identify $\lambda_{\infty}$ and $\boldsymbol{\eta}_{\infty}$ to the Perron eigenelements of the associated limit matrix $\mathbf{K}_{\alpha,\gamma}(a,a+s)$. Indeed, let
\[
\mathbf{h}_{\infty}(a) = e^{\lambda_\infty a } \boldsymbol{\Psi}_{\alpha}^{-1}(0,a) \pars{ \mathbf{I} -  2 \int_0^a e^{- \lambda_\infty s } \boldsymbol{\Psi}_{\alpha}(0,s) \mathbf{B}_{\gamma}(s) \ ds } \boldsymbol{\eta}_{\infty}.
\]
Therefore
$\ 
\mathbf{h}_{\infty}(0) = \boldsymbol{\eta_\infty}$,
and differentiating $\mathbf{h}_{\infty}(a)$ we obtain that
\[
\mathbf{h}_{\infty}'(a) = \lambda_{\infty}\mathbf{h}_{\infty}(a) - 2 \mathbf{B}_{\gamma}(a) \boldsymbol{\eta}_\infty + \mathbf{D}_\alpha (a) \mathbf{h}_{\infty}(a),
\]
or equivalently
\[
\mathcal{Q}_{\alpha,\gamma} \mathbf{h}_{\infty}(a) = \lambda_{\infty} \mathbf{h}_{\infty}(a) .
\]
By uniqueness of the solution to the eigenvalue problem associated to $\mathcal{Q}_{\alpha,\gamma} $ (see previous sections) we conclude that
\[
\mathbf{h}_{\infty}(a) = \mathbf{h}_{\alpha,\gamma}(a) , \qquad \lambda_{\infty} = \lambda_{\alpha,\gamma}.
\]
This allow us to conclude that for all fixed values of $a,s \geq 0$, both coordinates of $\boldsymbol{\ell}_{\alpha_n, \gamma_n}(a,s)$ converge to $\boldsymbol{\ell}_{\alpha, \gamma}(a,s)$. Thus, by Dominated Convergence Theorem, we can conclude the continuity of $(\alpha, \gamma) \mapsto \int_0^{+\infty } \boldsymbol{\ell}_{\alpha, \gamma}(a,s) ds = \mathbf{h}_{\alpha,\gamma}(a)$.
\end{proof}

\subsection{Proof of Proposition \ref{prop:calculDlambda}}
\begin{proof}
First, since $\mathbf{h}_{\alpha, \gamma}$ is normalised by $\ap{\boldsymbol{\nu}_{\alpha, \gamma}}{\mathbf{h}_{\alpha,\gamma}} = 1$, we can write 
\[
\lambda_{\alpha, \gamma} = \ap{\boldsymbol{\nu}_{\alpha, \gamma}}{\mathcal{Q}_{\alpha, \gamma} \mathbf{h}_{\alpha,\gamma}} = \ap{\boldsymbol{\nu}_{\alpha, \gamma} \mathcal{Q}_{\alpha, \gamma}}{ \mathbf{h}_{\alpha,\gamma}} .
\]
Then, for all $\delta \in \rr$ we have
\begin{align*}
\ap{\boldsymbol{\nu}_{\alpha, \gamma}}{\pars{ \mathcal{Q}_{\alpha + \delta, \gamma} - \mathcal{Q}_{\alpha, \gamma} }  \mathbf{h}_{\alpha+\delta,\gamma} } &= \ap{\boldsymbol{\nu}_{\alpha, \gamma}}{ \lambda_{\alpha + \delta} \mathbf{h}_{\alpha+\delta,\gamma} } - \ap{\lambda_{\alpha, \gamma }\boldsymbol{\nu}_{\alpha, \gamma}}{ \mathbf{h}_{\alpha+\delta,\gamma} } \\
&= \ap{\boldsymbol{\nu}_{\alpha, \gamma}}{\mathbf{h}_{\alpha+\delta,\gamma}} \pars{\lambda_{\alpha + \delta, \gamma} - \lambda_{\alpha, \gamma}} 
\end{align*}
and therefore
\begin{align*}
\lambda_{\alpha + \delta, \gamma} - \lambda_{\alpha, \gamma} &=  \frac{\ap{\boldsymbol{\nu}_{\alpha, \gamma}}{\pars{ \mathcal{Q}_{\alpha + \delta, \gamma} - \mathcal{Q}_{\alpha, \gamma}} \mathbf{h}_{\alpha+\delta,\gamma}}}{  \ap{\boldsymbol{\nu}_{\alpha, \gamma}}{\mathbf{h}_{\alpha+\delta,\gamma}} } \\
&= \frac{1}{  \ap{\boldsymbol{\nu}_{\alpha, \gamma}}{\mathbf{h}_{\alpha+\delta,\gamma}} } \int_0^{+\infty} \boldsymbol{\nu}_{\alpha, \gamma}(da) ^\top \begin{bmatrix}
-\delta & \delta \\ 0 & 0
\end{bmatrix}
\mathbf{h}_{\alpha+\delta,\gamma}(a) \\
&= \frac{\delta}{  \ap{\boldsymbol{\nu}_{\alpha, \gamma}}{\mathbf{h}_{\alpha+\delta,\gamma}} } \int_0^{+\infty} \pars{ h_{\alpha + \delta, \gamma}(a,1) - h_{\alpha + \delta, \gamma}(a,0) } \nu_{\alpha,\gamma}(da,0) .
\end{align*}
Thanks to Lemma \ref{lemma:continuity}, and once more by Dominated Convergence Theorem, we have that $\ap{\boldsymbol{\nu}_{\alpha,\gamma}}{\mathbf{h}_{\alpha+\delta,\gamma}} \to \ap{\boldsymbol{\nu}_{\alpha,\gamma}}{\mathbf{h}_{\alpha+\delta,\gamma}}$ and $\ap{\nu_{\alpha,\gamma}(\cdot, 0)}{h_{\alpha + \delta, \gamma}(\cdot,1) - h_{\alpha + \delta, \gamma}(\cdot,0)} \to \ap{\nu_{\alpha,\gamma}(\cdot, 0)}{h_{\alpha, \gamma}(\cdot,1) - h_{\alpha, \gamma}(\cdot,0)}$ as $\delta \to 0$. Then, dividing by $\delta$ and making $\delta \to 0$,  we obtain ~\eqref{eq:dlambda_dalpha}.
Analogously for any $\delta \in \rr$ small enough we have
\begin{align*}
\lambda_{\alpha, \gamma + \delta} - \lambda_{\alpha, \gamma} &=  \frac{\ap{\boldsymbol{\nu}_{\alpha, \gamma}}{\pars{ \mathcal{Q}_{\alpha , \gamma + \delta} - \mathcal{Q}_{\alpha, \gamma}} \mathbf{h}_{\alpha,\gamma + \delta}}}{  \ap{\boldsymbol{\nu}_{\alpha, \gamma}}{\mathbf{h}_{\alpha,\gamma + \delta}} } \\
&= \frac{1}{  \ap{\boldsymbol{\nu}_{\alpha, \gamma}}{\mathbf{h}_{\alpha,\gamma + \delta}} } \int_0^{+\infty} \boldsymbol{\nu}_{\alpha, \gamma}(da) ^\top 2 \begin{bmatrix}
0 & 0 \\ \delta \beta_1(a) & - \delta \beta_1(a)
\end{bmatrix}
\mathbf{h}_{\alpha,\gamma + \delta}(0) \\
&= \frac{\delta}{  \ap{\boldsymbol{\nu}_{\alpha, \gamma}}{\mathbf{h}_{\alpha,\gamma + \delta}} } \pars{ h_{\alpha, \gamma + \delta}(0,0) - h_{\alpha, \gamma + \delta}(0,1) }  \int_0^{+\infty} \beta_1(a) \nu_{\alpha}(da,1),
\end{align*}
from where ~\eqref{eq:dlambda_dgamma} is obtained by the same arguments.
\end{proof}

\subsection{Proof of Proposition \ref{prop:dgamma}}
\begin{proof}
From the proof of Theorem~\ref{prop:main}, if $(\alpha,\gamma)$ is in the survival region, then
$\ 
\det \pars{ \mathbf{I} - \mathbf{F}(\lambda _{\alpha, \gamma}) } = 0$, and $\mathbf{h}_{\alpha, \gamma}(0)$ is the unique non-trivial solution to the linear problem $ \pars{ \mathbf{I} - \mathbf{F}(\lambda _{\alpha, \gamma}) } \mathbf{h}_{\alpha, \gamma}(0) = 0.$
From the expression of $\mathbf{F}(\lambda)$ and the first equation, we have that for $\lambda = \lambda _{\alpha, \gamma} $
\[
\det \begin{bmatrix}
2(1-p) \xi_0(\alpha + \lambda) + 2 \gamma \xi_{01}(\lambda) - 1 & 2(1-\gamma) \xi_{01} (\lambda) \\
2 \gamma \xi_1(\lambda) & 2(1-\gamma) \xi_1(\lambda) -1 
\end{bmatrix}
= 0,
\]
where $\xi_0(\lambda) = \int_0^{+\infty} e^{-\lambda a} \beta_0(a) \exp (- \int_0^a \beta(s) ds) da$ and $\xi_1(\lambda) = \int_0^{+\infty} e^{-\lambda a} \beta_1(a) \psi_1(0,a) da$ are the Laplace transforms associated to the division times of types 0 and 1, and $\xi_{01}(\lambda) = \int_0^{+\infty} e^{-\lambda a} \psi_0 \star \psi_1 (0,a) da$. This implies the following implicit relation characterising $\lambda _{\alpha, \gamma}$:
\[
\xi_{01} \pars{\lambda _{\alpha, \gamma}} = \frac{(1-2(1-p)\xi_0(\alpha + \lambda _{\alpha, \gamma}))(1-2(1-\gamma)\xi_1(\lambda _{\alpha, \gamma}))}{2 \gamma},
\]
which allows to simplify the matrix $\mathbf{F}(\lambda_{\alpha,\gamma})$ in order to obtain that 
\[
\mathbf{h}_{\alpha, \gamma}(0) \in \textrm{span} \set{ \begin{pmatrix} 1 - 2 \xi_1(\lambda _{\alpha, \gamma}) + 2 \gamma \xi_1(\lambda _{\alpha, \gamma} ) \\ 2 \gamma \xi_1(\lambda _{\alpha, \gamma})  \end{pmatrix}}.
\]
Since $\mathbf{h}_{\alpha, \gamma}(0)$ is a non-negative vector, we have finally that
\[
\textrm{sign} \pars{ h_{\alpha, \gamma}(0,0) - h_{\alpha, \gamma}(0,1) } = \textrm{sign} \pars{1 - 2 \xi_1(\lambda_{\alpha,\gamma}) }.
\]
Notice that $\lambda \mapsto \xi_1(\lambda)$ is a continuous decreasing function, such that $\xi_1(0) = 1$ and $\xi_1 \to 0$ as $\lambda \to +\infty$. Moreover, the Malthusian parameter associated to $\xi_1$, $\lambda_1^*$, is the unique solution to $\xi_1(\lambda_1^*) = 1/2$ (this is classical, see for example \cite{webb1985theory}. Therefore, if $\lambda_{\alpha,\gamma} > \lambda_1^*$, then $\xi_1(\lambda_{\alpha,\gamma}) < 1/2$ and by \eqref{eq:dlambda_dgamma}, $\partial_{\gamma} \lambda_{\alpha,\gamma} > 0$. Analogously, if $\lambda_{\alpha,\gamma} < \lambda_1^*$, then $\partial_{\gamma} \lambda_{\alpha,\gamma} < 0$.

\bigskip

\end{proof}

\subsection{Proof of Proposition \ref{prop:dlambdaGamma_equiv}}
\begin{proof} We study the sign of $\partial_\gamma \lambda_{\alpha,\gamma}$ in the case where $p$ is big or small enough, and then relate this two partial analyses by continuity. Assumption \ref{ass:domination} plays a key role.
\smallskip
    \begin{enumerate}[label = \textbf{(\roman*)}]
        \item \textbf{Case $\boldsymbol{p > 1/2}$}. From Proposition \ref{prop:nonexplosion} and the equivalence established in Lemmas \ref{lemma:rhoKinf} and \ref{lemma:uniquenessLambda}, if $p>1/2$ we have that 
        \[
        \textrm{sign} \pars{  \frac{1}{2} \pars{ 1 + \frac{(2p-1)q}{1-q}} - \gamma } = \textrm{sign} \pars { \rho(\mathbf{K}_\infty^{\alpha,\gamma}) - 1 } = \textrm{sign} (\lambda_{\alpha,\gamma} ).
        \]
        In particular, for all $\alpha \geq 0$ it exists
        \[
        \bar \gamma_\alpha :=  \frac{1}{2} \pars{ 1 + \frac{(2p-1)q}{1-q}}
        \]
        such that $\lambda_{\alpha, \bar \gamma_{\alpha}} = 0$. Recall from Corollary \ref{cor:dlambdage0} and Fig. \ref{fig:laplace01} that $\lambda_{\alpha,0} \geq \lambda_1^* > 0 $. Therefore, by the continuity of $\gamma \mapsto \lambda_{\alpha,\gamma}$, for some $\hat \gamma > 0$ we must have $\lambda_{\alpha,\hat \gamma} < \lambda_1^*$, and by Proposition \ref{prop:dgamma}, $\left. \partial_\gamma \lambda_{\alpha, \gamma} \right|_{\gamma = \hat \gamma} < 0$. This implies that for all $\gamma \geq \hat \gamma$, $ \partial_\gamma \lambda_{\alpha, \gamma} < 0$. and $\partial_\gamma \lambda_{\alpha, \gamma} = 0$  for all $\gamma < \hat \gamma$. 
        \smallskip
        Thus, whenever $p > 1/2$, we have that for all $\alpha, \gamma$, $\partial_{\gamma} \lambda_{\alpha, \gamma} \leq 0$.
        \item \textbf{Small $\boldsymbol{p}$ case}. Now we prove that for all $\alpha$, if $p$ is small enough, then we have the opposite, namely: for all $\gamma \geq 0$, $\partial_{\gamma} \lambda_{\alpha, \gamma} \geq 0$. Consider the limit case $p = 0$ at $\gamma = 1$. The characteristic equation for $\lambda_{\alpha,\gamma=1}$ becomes
        \[
         \xi_0(\alpha + \lambda) + \xi_{01}(\lambda) = \frac{1}{2}.
        \]
        Recall that $T_{div}$ is the division time of a non-switching cell. By Jensen's inequality,
        \begin{align*} 
        \xi_{0}(\alpha + \lambda) &= \esperd{0,0}{\exp(-( \lambda + \alpha) T_{div})} \geq \exp \pars{-(\lambda+\alpha) \esperd{0,0}{T_{div}} } = \xi_0(\lambda) \xi_0(\alpha).
        \end{align*}
        In particular, for $\lambda = \lambda_1^*$, and since $\xi_0(\lambda) > \xi_1(\lambda)$ by Remark \ref{rmk:domination},
        \begin{align*}
        \xi_{0}(\alpha + \lambda_1^*) 
            &\geq \xi_0(\alpha ) \xi_{0}(\lambda_1^*) > \xi_0(\alpha ) \xi_{1}(\lambda_1^*) = \frac{1-q}{2}.
        \end{align*}
        On the other hand, letting $\tau \sim \textrm{Exp}(\alpha)$ independent from $T_{div}$ be the switching time of type 0, and recalling the definition of $\psi_0 \star \psi_1$ from \eqref{convol}, by Fubini's theorem we can write
        \begin{align*}
            \xi_{01}(\lambda) &:= \int_0^{+\infty} e^{-\lambda a} \alpha \beta_1(a) \psi_0 \star \psi_1(0,a) \; da \\ &= 
           \int_0^{+\infty}  \pars{\frac{1}{\psi_1(0,\tau)} \int_{\tau}^{+\infty} e^{-\lambda a } \beta_1(a) \psi_1(0,a) \; da } \alpha \psi_0(0,\tau) \; d\tau \\
            &= q \; \esperd{0,0}{\left. \esperd{0,1}{\left. e^{-\lambda T_{div}} \right| T_{div} \geq \tau} \right| \tau < T_{div}}.
        \end{align*}
        By Assumption \ref{ass:domination}, consider a monotone coupling $(\hat{T}_{div}^0, \hat{T}_{div}^1)$ on a common probability space such that for all $i \in \{0,1\}$, $\hat{T}_{div}^i$ has the same distribution under $\mathbb{P}$ as $T_{div}$ under $\mathbb{P}_{\delta(0,i)}$, and
        $\mathbb{P}(\hat{T}_{div}^1 > \hat{T}_{div}^0) = 1$. In particular $\{ \tau \leq \hat{T}_{div}^0 \} \subseteq \{ \tau \leq \hat{T}_{div}^1 \}$. Therefore
        \begin{align*}
            \xi_{01}(\lambda) &= q \; \esper{\left. \esper{\left. e^{-\lambda \hat{T}_{div}^1} \right| \hat{T}_{div}^1 \geq \tau} \right| \tau < \hat{T}_{div}^0}= q \; \esper{\left. \esper{ e^{-\lambda \hat{T}_{div}^1} } \right| \tau < \hat{T}_{div}^0} = q \; \xi_1(\lambda).
        \end{align*}
        In particular, at $\lambda = \lambda_1^*$, $\xi_{01}(\lambda_1^*) = q/2$. Finally,
        \[
        \xi_0 (\alpha + \lambda_1^*) + \xi_{01}(\lambda_1^*) > \frac{1}{2}
        \]
        and hence $\lambda_{\alpha,\gamma=1} > \lambda_1^{*}$. Then, analogously to Corollary \ref{cor:dlambdage0}, we have that $\left. \partial^{-}_{\gamma} \lambda_{\alpha,\gamma} \right|_{\gamma=1} > 0$, where $\partial^{-}_{\gamma}$ is the left derivative with respect to $\gamma$ (whose values are restricted to $[0,1]$), and by Proposition \ref{prop:dgamma}, that for all $\gamma \geq 0$, $\partial_{\gamma} \lambda_{\alpha, \gamma} \geq 0$.
        \item {\textbf{Conclusion}.} The continuity of $p \mapsto \partial_{\gamma} \lambda_{\alpha,\gamma}$ can be exhibited following the same arguments presented in the proof of Lemma \ref{lemma:continuity2}. Moreover, as we did in Proposition \ref{prop:dgamma} we can further show that $p \mapsto \partial_{\gamma} \lambda_{\alpha,\gamma}$ is continuously differentiable and that for all $\alpha, \gamma$, $p \in (0,1)$,
        \[
        \partial_p \lambda_{\alpha,\gamma} = \ap{\boldsymbol{\nu}}{ \pars{\partial_p \mathcal{Q}} \mathbf{h} } = - 2 h(0,0) \int \beta_0(a) \nu(a, 0) da < 0.
        \]
        Hence, for all $\alpha,\gamma$, the function $p \mapsto \partial_{\gamma} \lambda_{\alpha,\gamma}$ is continuous and strictly decreasing. Therefore, for all $\alpha \geq 0$, there is a unique $\bar p_{\alpha,\gamma}$ such that $\lambda_{\alpha,\gamma} = \lambda_1^{*}$ and $\partial_\gamma \lambda_{\alpha,\gamma} = 0$ for all $\gamma \in (0,1)$. In particular, this $\bar p_{\alpha,\gamma}$ is then constant in $\gamma$, and uniquely determined by $\alpha$, which gives the result. The conclusion is presented in Figure \ref{fig:prop54}.
    \end{enumerate}

\end{proof}

\section{Biological discussion and outlook}
\subsection{Under constant stress, high recovery might lead to increased extinction}
\label{sec:discuss1}
Fig. \ref{fig:hetamap_pi0} shows numerical solutions of the system~\eqref{eq:pi0}-\eqref{eq:pi1}, giving the values of the survival probabilities $1- \pi_i$ for a population initiated by a single vulnerable cell ($i=0$, in row \textbf{A}) and by a single tolerant cell ($i=1$, in row \textbf{B}). We can observe in row \textbf{A} that as expected, the survival of the population arising from a type 0 cell, increases as the probability of switching before dividing, $q$, increases. This is expected since we suppose that type 1 is not susceptible to death. On the other hand, as shown by Proposition ~\ref{prop:survival}, increasing the recovery probability $\gamma$ introduces a significant risk of extinction, particularly when the antibiotic-induced death probability $p$ is large. While this might appear biologically counter-intuitive, increasing the number of recovered individuals of type 0 introduces individuals who are prone to death and might lead, with positive probability, to the extinction of the whole population. This is also true, although for a higher level of stress, if we start with a type 1 individual as recovery can lead to the production of type 0 individuals who are susceptible to death (Figure \ref{fig:hetamap_pi0}B). When recovery ($\gamma$) is high, having a higher switching rate is of interest for the survival of the population, in order to escape from the risk region. Thus, under fixed environmental stress, the optimal strategy from the point of view of the survival probability of the population is to switch as fast as possible.

\begin{figure}
    \centering
    \includegraphics[width=0.9\textwidth]{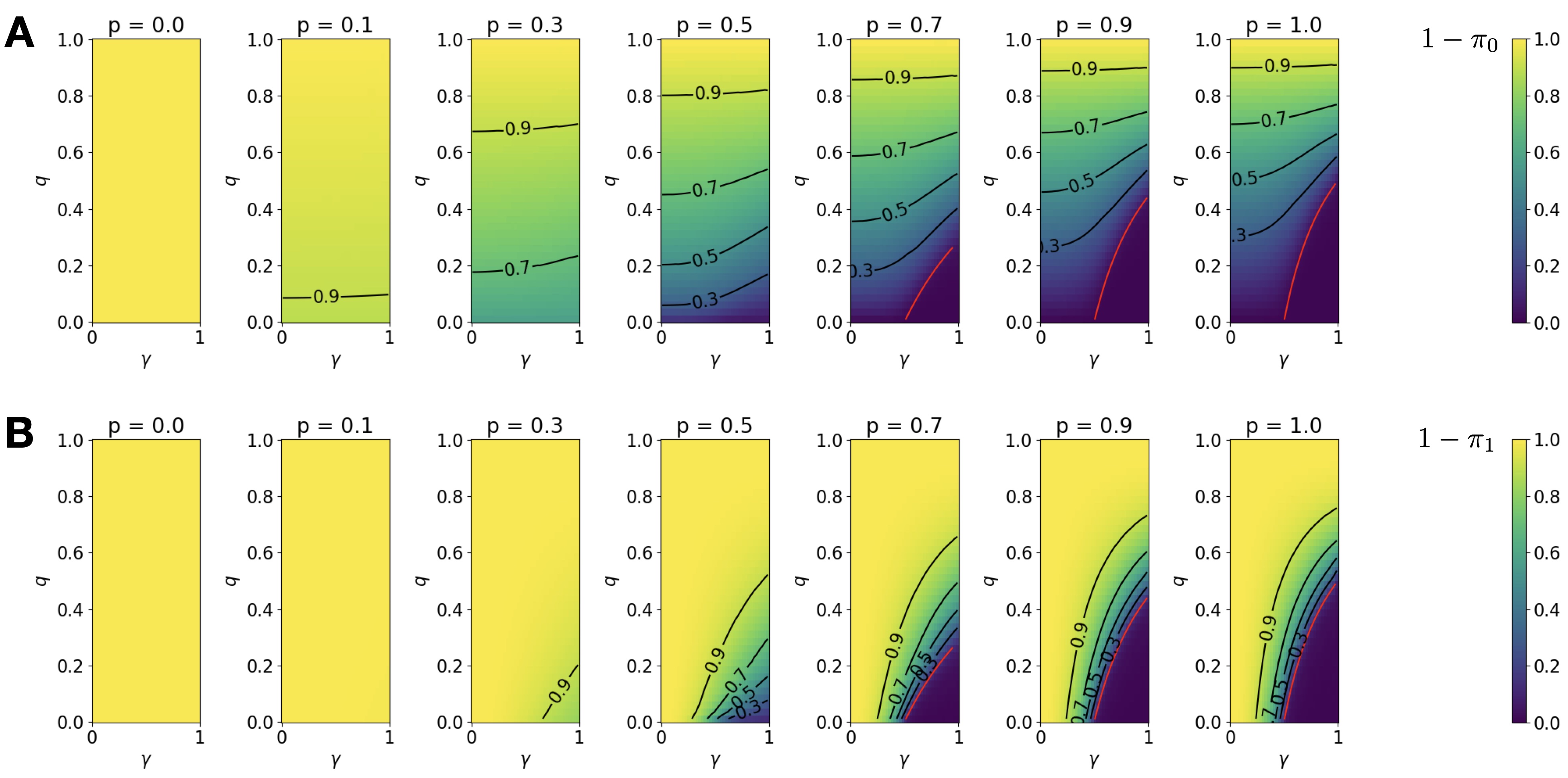}
    \caption{Value of the survival probability $1- \pi_i$ for a population issued from a single cell of type $i=0$ (\textbf{A}) and $i=1$ (\textbf{B}). The values were numerically computed as the minimal solution to the system \eqref{eq:pi0}-\eqref{eq:pi1}. When $p > 1/2$, the red line represents the critical case of equality in \eqref{eq:extinctionCondition}.}
    \label{fig:hetamap_pi0}
\end{figure}

\subsection{Under constant low stress, high recovery leads to opposite effects on the survival probability and the population growth rate}
We have shown in Propositions \ref{prop:dlambdaGamma_equiv} and \ref{prop:dgamma_pi} that measuring the response of the extinction probability of a population issued from a single cell and of the population growth rate with respect to variations in the recovery probability $\gamma$ can lead to seemingly opposite conclusions. On the one hand, as discussed in the paragraph above, increasing $\gamma$ will always decrease the population survival probability. On the other hand, increasing $\gamma$ can increase the population growth rate ($\lambda$) if $p$ is low enough. In that context, creating vulnerable but fast-proliferating type 0 individuals becomes advantageous. 

Fig.~\ref{fig:regimeChange} provides a quantitative representation of this behaviour. Row A shows numerical approximations of $\lambda$  obtained from the numerical solution to PDE \eqref{eq:PDE}. We observe that for the chosen $\beta_0, \beta_1$ (see caption), and for a stress weak enough such that $p \leq 0.47$, $\lambda$ increases with $\gamma$ for any $q$. However, between $p=0.47$ and $p=0.48$,  the direction of the fitness gradient changes, such as increasing $\gamma$ leads to a lower growth rate for the population. On the contrary, on row B we observe, as we already did in Fig.~\ref{fig:hetamap_pi0}, that for all values of $p$ (and $q$), the survival probability decreases, though ever so slightly, with $\gamma$.

This surprising result raises some methodological questions about comparing the growth response using single-cell and population techniques. Indeed, parameter variations might not have the same effects on the population growth rate in large population experiments and on the viability curves of small colonies. In other words, the optimal parameters that characterise the single lineages that survive by the end of the experiment and the parameters that maximise the exponential growth of the population might not be the same. 

\begin{figure}
    \centering
    \includegraphics[width=0.9\textwidth]{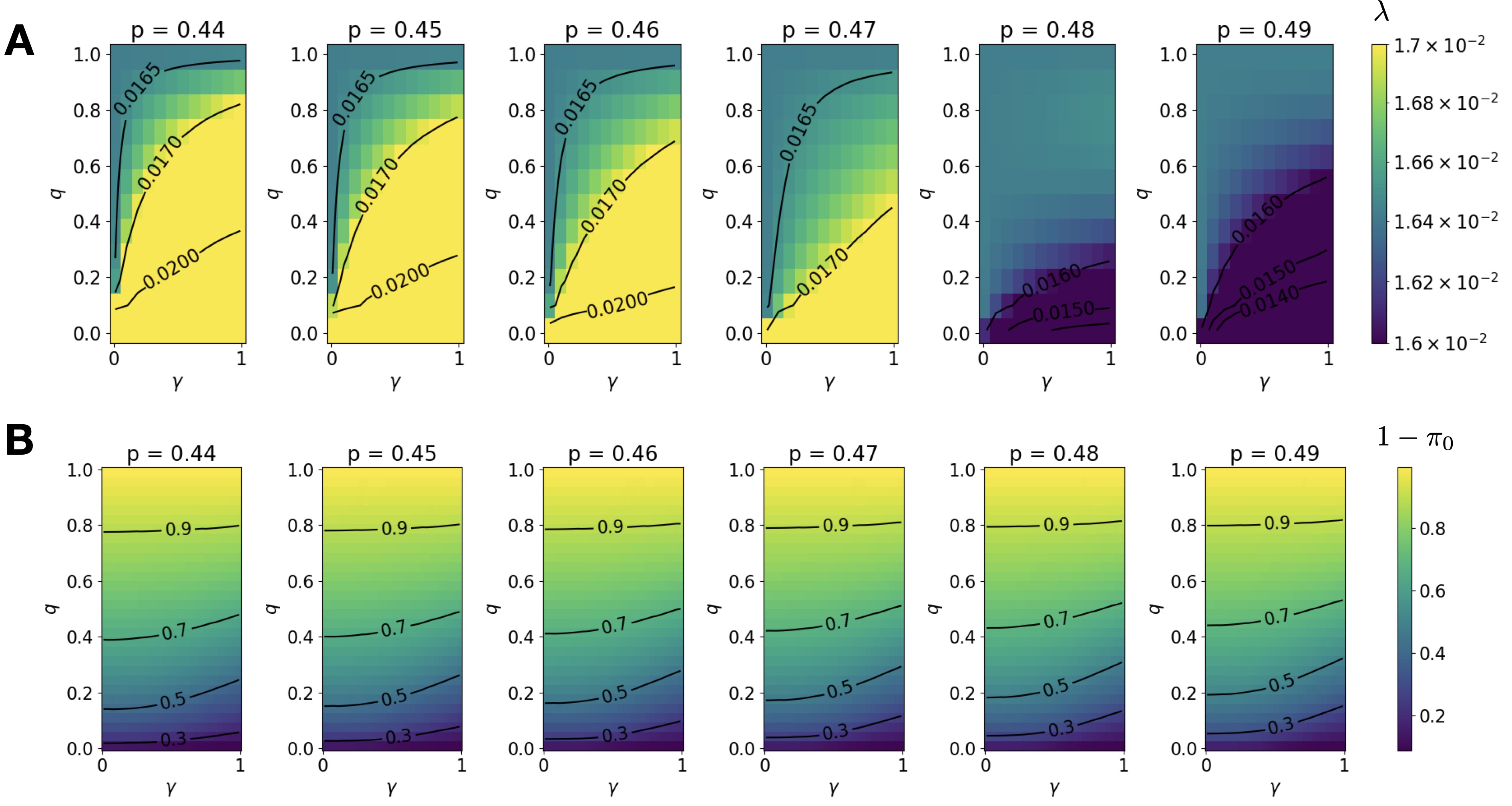}
    \caption{Variation of the population growth rate $\lambda$ (row \textbf{A}) and of the survival probability $1 - \pi_0$ starting from an initial cell of type 0 (row \textbf{B}), for some values of $p < 1/2$. The division rates $\beta_0$ and $\beta_1$ correspond to Gamma distributed division ages, as in Example \ref{ex:gammaBeta0}. For type 0, the parameters are $a_0 = 3, b_0 = 1$ (mean division time equals to $3$), and for type 1, the parameters are $a_0 = 3, b_0 = 0.1$ (mean division time equals to $30$).
    \label{fig:regimeChange}}
\end{figure}

\subsection{Under periodic stress, fine-tuning the switching rate and the recovery probability with respect to the period length can lead to higher growth}
\label{sec:discuss3}
\begin{figure}
    \includegraphics[width=0.9\textwidth]{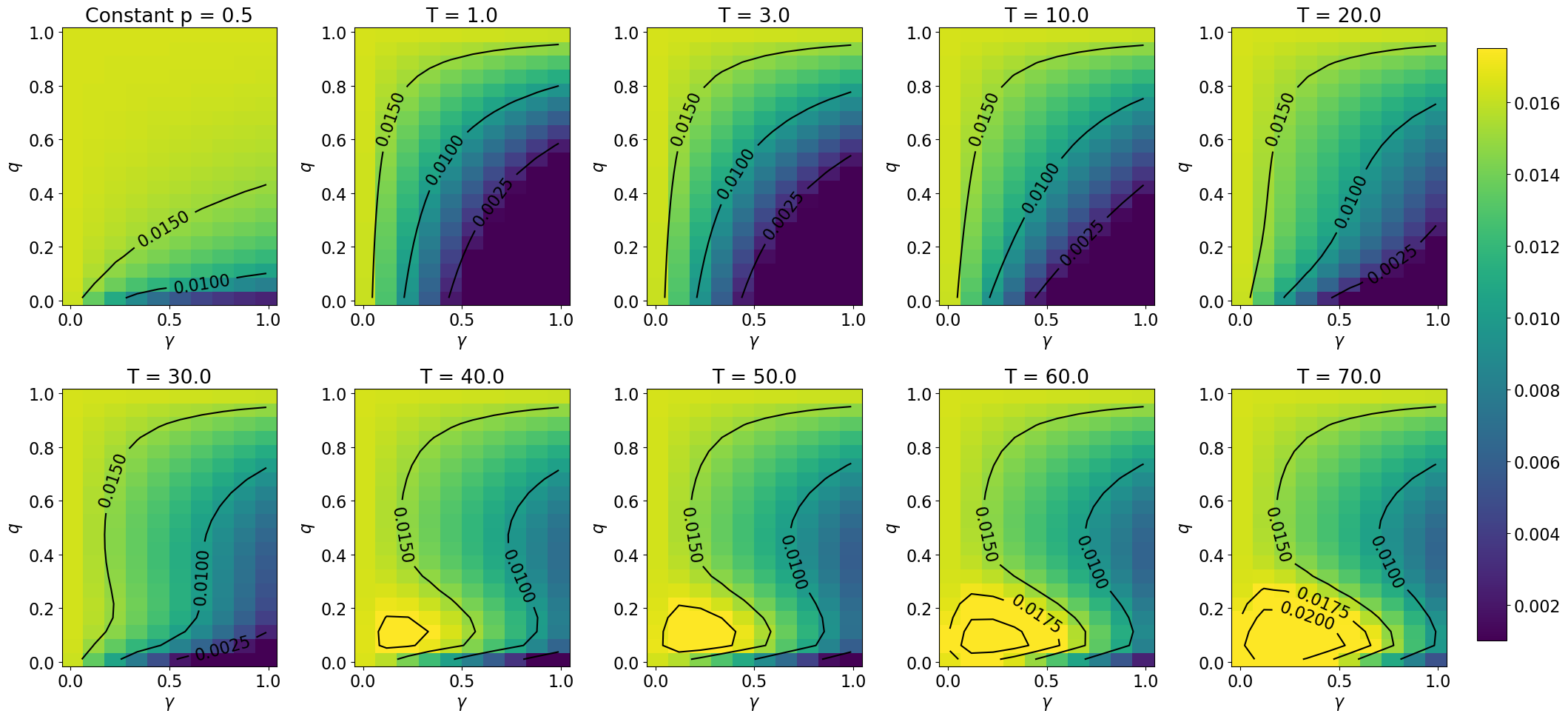}
    \caption{Value of the Floquet eigenvalue $\lambda_T$ for increasing values of $T$ as indicated in the head of the panel, as function of $\gamma$ and $q$, obtained from simulations of \eqref{eq:PDE}. The death probability $p$ varies between $0.25$ and $0.75$ spending a time $T/2$ in each. Division times are Gamma distributed as in Fig.~\ref{fig:regimeChange} (mean division time of type 0 is $3$, and of type 1 is $30$).
     } 
    \label{fig:malthus_floquet}
\end{figure}

Proposition \ref{prop:dgammaFloquet} shows that, under periodic stress, depending on the form of the eigenfunctions $h_{\alpha,\gamma}(t,\cdot)$, the gradient of the population growth rate $\lambda_{\alpha,\gamma}$ admits non-trivial changes of sign. This implies that the fitness landscape changes when the environmental stress is allowed to fluctuate. In particular, it is possible that phenotypic plasticity strategies that would have led almost surely to extinction in a constant environment become advantageous in the periodic case. 

As shown for illustrative purposes in Fig.~\ref{fig:malthus_floquet}, simulations of PDE \eqref{eq:PDE} with a $T$-periodic stress (i.e.$T$-periodic $p$) illustrate this result for different values of $T$. We use the same parameters as in the constant stress case showed in Fig. \ref{fig:regimeChange}, but with a $T$-periodic stress signal $p$ switching periodically between ``good" and ``bad" times, such that $p(t) = 0.25$ for $0 < t \leq T/2$ and $p(t) = 0.75$ for $T/2 < t \leq T$. We compare the simulations with the case where $p$ is fixed at its average value of $p \equiv 0.5$ (top left panel). 

We observe that under different values of $T$, the best stress response strategy changes. First, when the period $T$ is short ($T \leq 30$ in the illustrative case shown in Fig.~\ref{fig:malthus_floquet}), a non-growing region emerges in the fitness landscape, for a strategy combining low switching and high recovery. This suggests that this strategy, which leads to a high number of type 0 cells, is detrimental. As type 0 cells are prone to death under stress, and the ``good" periods are too short to let them thrive, the overall population stops growing.  

However, as the value of $T$ increases, the fitness landscape changes appreciably. Surprisingly, we see the emergence of a ''Goldilocks zone" where populations that recover and switch slowly (but still recover and switch, with $\gamma > 0$ and $q > 0$) grow even faster than in the constant case. At the same time, the previously well-adapted strategy of likely recovery ($\gamma > 1/2$ verifying \eqref{eq:extinctionCondition}) and fast switch ($q > 1/2$) becomes less adapted as the period increases. In other words, under periodic stress, bacteria can allow themselves to switch less (i.e. to reduce the induction of the stress response) as long as they do not repair very fast. This seems to lead to an optimal coexistence of type 0 and type 1, where the small switching rate allows the population to take advantage of the good times, while the slow repair allows them to keep a reserve of type 1 individuals that maintain the population during bad times. 
Overall, our results suggest that the population growth rate and the survival probability of a population issued from a single individual need to be analysed separately and can sometimes show opposite behaviour. Moreover, under the more realistic configuration of periodic stress, optimal population growth can only be achieved through fine-tuning both the repair and the switching rate simultaneously.

\section*{Acknowledgments}
This work has been supported by the Chair Modélisation Mathématique
et Biodiversité of Veolia Environnement - École polytechnique - Museum National d'Histoire Naturelle - Fondation X. This work has been supported in part by a Wellcome Trust Investigator Awards (Grant No. 205008/Z/16/Z) awarded to M.E.K. It is also funded by the European Union (ERC, SINGER, 101054787). Views and opinions expressed are however those of the author(s) only and do not necessarily reflect those of the European Union or the European Research Council. Neither the European Union nor the granting authority can be held responsible for them.

\bibliographystyle{abbrv}
\bibliography{biblio}

\newpage

\appendix
\section{Appendix}

\subsection{Construction and well-posedness of $Z_t$}
\label{app-process}

\begin{definition}[Pathwise representation of the population process]
\label{def:pathwise}
Let $Z_0$ be a counting measure on $\rr_+ \times \set{0,1}$, of the form of ~\eqref{eq:defZt}, and $\mathcal{N}(du,dk,dz,d\omega)$ an independent Poisson point measure over $\rr_+ \times \nn_* \times \rr_+ \times [0,1]^2$ with intensity $du \ n(dk) \ dz \ d\omega$, with $n$ the counting measure on $\mathbb{N}_*$. $Z_0$ represents the initial population and $\mathcal{N}$ clocks the division and switching times that occur in some time interval measured by the integrating variable $u$. These events are measured for each individual $k$ independently and happen proportionally to their division and switching rates, which are measured by the integrating variable $z$. Finally, the two independent uniform random variables $\omega = (\omega_1, \omega_2)$ determine, respectively, the outcome of the first and the second daughter.

Thus, under the canonical filtration $(\mathcal{F}_t)_{t\geq 0}$ generated by $(Z_0, \mathcal{N})$, we define the process $(Z_t)_{t\geq 0}$ as
\begin{equation} \begin{split}
Z_t = 
&\sum_{k=1}^{N_0} \delta_{(A_k(0) + t, I_k(0))}  \\
\ &+ \int_0^t \int_{\nn^* \times \rr_+ \times [0,1]^2 }  \mathds{1}_{ \{ k \leq N_{u^-} \} }  \Bigg\{ \mathds{1}_{  \set{ z \leq \alpha (1 - I_k(u^-))} }  \left[ 
\delta_{(A_k(u^-) + (t-u) , 1) } - \delta_{(A_k(u^-) + (t-u) , I_k(u^-)) } \right] \\
\ &+  \mathds{1}_{ \set{ 0 < z - \alpha (1 - I_k(u^-)) \leq \beta_{I_k(u^-)}(A_k(u^-))} }  \Big( - \delta_{( A_k(u^-) + (t-u) , I_k(u^-) )}\\& + \mathds{1}_{I_k(u^-) = 0} \left[ \mathds{1}_{\omega_2 > p} \delta_{(t-u, 0) } + \mathds{1}_{\omega_3 > p} \delta_{(t-u, 0) } \right] \\
\ & + \mathds{1}_{I_k(u^-) = 1} \Big[ \mathds{1}_{\omega_2 \leq \gamma} \delta_{(t-u, 0) } + \mathds{1}_{\omega_2 > \gamma } \delta_{(t-u, 1) } + \mathds{1}_{\omega_3 \leq \gamma } \delta_{(t-u, 0) } + \mathds{1}_{\omega_3 > \gamma} \delta_{(t-u, 1) } \Big]   \Big) \Bigg\} \mathcal{N} (du, dk, dz, d\omega).
\end{split} \end{equation}

Note  that 
$$Z_t(da,di) = Z_t(da,\{0\})\delta_0(di)  + Z_t(da,\{1\})\delta_1(di).$$
\label{eq:Zt_trajdef}
We explain now each term of the RHS of ~\eqref{eq:Zt_trajdef}. The first line represents the deterministic evolution of the population when no random events happen before time $t$. The second line represents the switching events, that occur only for individuals of type $I_k = 0$ at rate $\alpha$, and gives a new individual of type 1 with identical age while removing the previous individual of type 0 from the population. The third and four lines represent the divisions, which remove the divided cell from the population. If the mother type is 0 (third line), we add independently two cells of type 0, but only with probability $1-p$ each. If the mother type is 1 (fourth line) we add independently two cells whose type is decided by a Bernoulli random variable of parameter $1-\gamma$. 
\end{definition}

We show that the stochastic process $(Z_t)_t$ is well-defined under Assumptions \ref{ass:assumptions}:
\begin{proposition}[Well-posedness and first-moment control] 
\label{prop:nonexplosion}
Under Assumptions \ref{ass:assumptions}, and if $\esper{\int ( Z_0(da,\{0\})+ Z_0(da,\{1\})) } < \infty$ and $\esper{\int a( Z_0(da,\{0\})+ Z_0(da,\{1\})) } < \infty$, then the SDE~\eqref{eq:Zt_trajdef} has a well-posed solution $(Z_t)_{t\geq0} \in \mathbb{D}\pars{\rr_+, \mathcal{M}\pars{\set{0,1} \times \rr_+}}$ which verifies for every $t >0$
	\begin{eqnarray}
	\esper{ \sup_{s\in[0,t]}  \int (1+a)( Z_s(da,\{0\})+ Z_s(da,\{1\})) } &\leq& \esper{\int (1+a)( Z_0(da,\{0\})+ Z_0(da,\{1\})) }\exp \pars{(\bar b +1) t}\nonumber \\& <& + \infty.
 \label{eq:gronw_prop}
	\end{eqnarray}
\end{proposition}
\begin{proof}
The proof is classical for populations with uniformly bounded birth rates, a suitable sequence of stopping times and Gronwall inequality to conclude. See for example \cite{Tran2008}.
\end{proof}

\bigskip

Using the Compensated Poisson Point Measure associated to $\mathcal{N}$ and Ito's formula for semi-martingales \cite{Ikeda1981}  we can derive the following representation for $\ap{Z_t}{f}$. 

\begin{proposition}[Semi-martingale decomposition]
Under control assumptions for the moments of $Z_0$ and for control assumptions for $\beta$ { \color{red} 2.1}, $Z_t$ is well-posed for $t \in [0,T]$ for any $T > 0$. Moreover, for any $f \in C^{1,\cdot}(\rr_+ \times \{0,1\})$, we can write for any $t \geq 0$
\begin{align}
    \int (f(a,0) Z_t(da,\{0\})&+ f(a,1) Z_t(da,\{1\})) = \int (f(a,0) Z_0(da,\{0\})+ f(a,1) Z_0(da,\{1\})) \nonumber \\
    & + \int_0^t  \int ( \mathcal{Q}f(a,0) Z_s(da,\{0\})+  \mathcal{Q}f(a,1) Z_s(da,\{1\})) ds + \mathcal{M}_t^f,
    \label{eq:semimartinagaledecomposition}
\end{align}
where $\mathcal{M}_t^f$ is a squared-integrable $\mathcal{F}_t$-martingale and the infinitesimal generator $\mathcal{Q}$ is defined as
\begin{equation}
\begin{split}
    \mathcal{Q}f(a,i) =  & \ \partial_a f(a,i) + (1-i) \alpha (f(a,1) - f(a,0) ) - \beta_i(a) f(a,i) \\
    & + 2 (1-i) (1-p) \beta_i(a) f(0,0) + 2 i \beta_i(a)  ( \gamma f(0,0) + (1-\gamma ) f(0,1) ).
\end{split}
\end{equation}
\label{prop:semimartingale}
\end{proposition}

We can write Proposition \ref{prop:semimartingale} in an abbreviated vectorial form. For all $f:\rr_+ \times \{0,1\} \to \rr$ which is $C^1$ in the first coordinate, we set $\mathbf{f}(a) = (f(a,0),f(a,1))$ and analogously $\mathbf{f}'(a) = (\partial_a f(a,0),\partial_a f(a,1))$. Then, we can write for $\mathbf{f} \in (C^{1}(\rr_+))^2$, $\mathcal{Q}\mathbf{f} = (\mathcal{Q} f(a,0), \mathcal{Q}f(a,1))$ with $\mathcal{Q}$ given by \eqref{eq:Q}.

\subsection{Technical lemmas}

We recall some useful classical properties.
\begin{proposition}
Let $S, T$ be two bounded linear operators in a Banach space $\mathcal{X}$ and let $\alpha(S)$ the measure of non-compactness of $S$ as defined in Definition \ref{def:noncompactnessmeasure}. Then,
\begin{enumerate}[label=(\roman*)]
    \item $\alpha(S) \leq \norm{S}_{op}$. 
    \item $\alpha(TS) \leq \alpha(T)\alpha(S)$. 
    \item $\alpha(T+S) \leq \alpha(T) + \alpha(S)$.
    \item $\alpha(S) = 0$ if and only if $S$ is compact.
\end{enumerate}
\label{prop:webb_prop4-9}
\end{proposition}

\begin{lemma}
\label{lemma:ineqRhoAB}
Let $A$ and $B$ two positive matrices such that $A_{ij} \geq B_{ij}$ for all pairs $(i,j)$, then:
\begin{enumerate}[label=(\roman*)]
    \item $\rho(A) \geq \rho(B)$, and
    \item $ \displaystyle
\rho(B) \geq \rho(A) \pars{ 1 - \frac{\max_{i,j} (A_{i,j} - B_{i,j} ) }{\min_{i,j} A_{i,j}}} . $
\end{enumerate}
\end{lemma}

\subsection{Complements on the proof of Proposition \ref{prop:main}}
\label{app:mainproof}

We start our analysis by studying some estimates about the boundness and compactness of $\mathbf{M}_t$. To that purpose, we write the decomposition
\[
\mathbf{M}_t \mathbf{f}(a) = \mathbf{U}_t \mathbf{f}(a) + \mathbf{W}_t \mathbf{f}(a),
\]
where $\mathbf{U}_t$ is destinated to be small and $\mathbf{W}_t$ is destinated to be compact, in the senses that will be detailed further below (see for example Section 3.4 of \cite{webb1985theory} for a detailed motivation of this decomposition). In our case, the natural decomposition is given by Lemma \ref{lemma:decompMt}.

\begin{lemma}
Consider the semigroup $\mathbf{M}_t$ on $(L^1(\rr_+))^2$ introduced by Def. \ref{def:Mt}. Then, for all $\mathbf{f} \in (L^1(\rr_+))^2$ we can write $\mathbf{M}_t \mathbf{f}(a) = \mathbf{U}_t \mathbf{f}(a) + \mathbf{W}_t \mathbf{f}(a)$ where
\label{lemma:decompMt}
\begin{subequations}
\begin{align}
\mathbf{U}_t \mathbf{f}(a) &:= \boldsymbol{\Psi}(a,a+t) \mathbf{f}(a+t) \label{eq:def_Ut},\\
\mathbf{W}_t \mathbf{f}(a) &:= 2 \int_0^t \mathbf{K}(a,a+t-s) (\mathbf{I}-\mathbf{S}_1)^{-1} \mathbf{S}_2 \mathbf{f}(s) ds ,
\label{eq:def_Wt}
\end{align}
\end{subequations}
where $\mathbf{S}_1 : (L^1(\rr_+))^2 \to (L^1(\rr_+))^2$ is a locally compact linear operator of spectral radius 0, and $\mathbf{S}_2 : (L^1(\rr_+))^2 \to (L^1(\rr_+))^2$ is a bounded linear operator, both defined by
\begin{align*}
\mathbf{S_1} \mathbf{f}(t) &:= 2 \int_0^t \mathbf{K}(0,t-s) \mathbf{f}(s) ds, \\
\mathbf{S_2} \mathbf{f}(t) &:= \boldsymbol{\Psi}(0,t) \mathbf{f}(t).
\end{align*}
Moreover, for all fixed $t > 0$, $\mathbf{W}_t : (L^1([0,t]))^2 \to (L^1([0,t]))^2$ is a compact operator.
\end{lemma}
\begin{proof}
Fix $t > 0$ and let $\mathbf{f} \in (L^1([0,t]))^2$. When needed, we extend $\mathbf{f}$ to $(L^1(\rr))^2$ as $\mathbf{f}(s) = (0,0)$ whenever $s > t$ or $s < 0$. The Duhamel representation obtained in ~\eqref{eq:FE}, along with ~\eqref{eq:FE0} gives that for all $a \geq 0$,
\begin{equation*}
    \mathbf{M}_t \mathbf{f}(a) = \boldsymbol{\Psi}(a,a+t) \mathbf{f}(a+t) + 2 \int_0^t \mathbf{K}(a,a+t-s) \mathbf{g}(s) ds, 
\end{equation*}
where $\mathbf{g}$ solves the fixed point problem
\begin{align*}
    \mathbf{g}(s) &= \boldsymbol{\Psi}(0,s) \mathbf{f}(s) + 2 \int_0^s \mathbf{K}(0,s-u) \mathbf{g}(u) du = \mathbf{S_2} \mathbf{f}(s) + \mathbf{S_1} \mathbf{g}(s).
\end{align*}
Therefore, if $(\mathbf{I} - \mathbf{S}_1)^{-1}$ is well defined, we have that $\mathbf{g}$ can be obtained as
$\ 
\mathbf{g} = (\mathbf{I} - \mathbf{S}_1)^{-1} \mathbf{S}_2 \mathbf{f}\ $, 
from where we get ~\eqref{eq:def_Wt}. 

Hence, we start proving that $\mathbf{S_1} : (L^1([0,t]))^2 \to (L^1([0,t]))^2$ is a compact linear operator, to then conclude thanks to the Fredholm alternative. 

Since functions are defined on  $[0,t]$, the compactness  of $S_1$ is deduced from the equicontinuity assumption of the Riesz-Fréchet-Kolmogorov Theorem. Consider some bounded sequence $\set{\mathbf{f}_n}_n$ on the unit ball of $(L^1([0,t]))^2 $. For all $\mathbf{f} \in (L^1([0,t]))^2$ we have

    \begin{align*}
        \int_0^{t} \norm{\mathbf{S}_1 \mathbf{f}(s+a) - \mathbf{S}_1 \mathbf{f}(s)}_1 ds &= 2 \int_0^t \norm{ \int_0^{s+a} \mathbf{K}(0,s+a-u) \mathbf{f}(u) du -  \int_0^{s} \mathbf{K}(0,s-u) \mathbf{f}(u) du }_1 ds   \\
        &\leq 2 \int_0^t\int_0^{s}  \norm{\mathbf{K}(0,s-u+a) - \mathbf{K}(0,s-u) }_1 \norm{\mathbf{f}(u)}_1 du ds \\
        &\quad + 2 \int_0^t \int_{s}^{a+s} \norm{\mathbf{K}(0,s-u+a)  }_1 \norm{\mathbf{f}(u)}_1 du ds \\
        &\leq 2 \int_0^t\int_0^{s}  \norm{\mathbf{K}(0,s-u+a) - \mathbf{K}(0,s-u) }_1 \norm{\mathbf{f}(u)}_1 du ds \\
        &\quad + 2 \int_0^a \pars{ \int_{0}^{u} \norm{\mathbf{K}(0,s-u+a)  }_1  ds } \norm{\mathbf{f}(u)}_1 du \\
        &\quad + 2 \int_a^t \pars{ \int_{u-a}^{u} \norm{\mathbf{K}(0,s-u+a)  }_1  ds } \norm{\mathbf{f}(u)}_1 du \\
        &\leq 2 \int_0^t\int_0^{s}  \norm{\mathbf{K}(0,s-u+a) - \mathbf{K}(0,s-u) }_1 \norm{\mathbf{f}(u)}_1 du ds \\
        &\quad + 2 |a|  \sup_{0 < s < t} \norm{\mathbf{K}(0,s)}_1 \norm{\mathbf{f}}_1.
    \end{align*}
    The second term of the RHS can be controlled uniformly as $|a| \to 0$. We focus in the first term. By construction, as shown in ~\eqref{eq:marginalT}, the jump times are absolutely continuous random variables. Therefore $s \mapsto \mathbf{K}(0,s)$, which contains the probability densities of the different jump events at time $s$, is a continuous application from $\rr_+$ to the vector space of positive $2 \times 2$ matrices. In particular, it is uniformly continuous over the compact $[0,t]$, this is, for all $\varepsilon > 0$, there exists $\delta > 0$ such that
    \[
   \norm{\mathbf{K}(0,s-a) - \mathbf{K}(0,s) }_1 < \varepsilon  \qquad \textrm{ for all } s \in [0,t[ \quad \textrm{ whenever } a < \delta.
    \]
    Therefore, for all $ a < \delta$ we have
    \begin{align*}
        \int_0^{t} \norm{\mathbf{S}_1 \mathbf{f}(s-a) - \mathbf{S}_1 \mathbf{f}(s)}_1 ds \leq 2 \varepsilon t  + 2 |a|  \sup_{0 < s < t} \norm{\mathbf{K}(0,s)}_1 =: \tilde{\varepsilon},
    \end{align*}
    uniformly with respect to $\mathbf{f}$. Therefore $\set {\mathbf{S}_1 \mathbf{f}_n}$ is relatively compact.

 Since $\mathbf{S_1} : (L^1([0,t]))^2 \to (L^1([0,t]))^2$ is compact, any $\lambda \neq 0$ is either an eigenvalue of $\mathbf{S_1}$ or otherwise lies in the domain of the resolvent. Let's suppose by absurd that $\lambda \neq 0$ is an eigenvalue associated to some nonzero eigenfunction $\mathbf{u} \in (L^1([0,t]))^2$ such that $\mathbf{S}_1 \mathbf{u} = \lambda \mathbf{u}$, this is, for all $a \in [0,t]$,
\[
\mathbf{u}(a) = \frac{2}{\lambda} \int_0^a \mathbf{K}(0,a-s) \mathbf{u}(s) ds ,
\]
and thereby,
\[
 \norm{\mathbf{u}(a)}_1 \leq \frac{2}{\lambda} \sup_{0<s<t} \norm{\mathbf{K}(0,s)}_1 \int_0^a \norm{\mathbf{u}(s)}_1 ds.
\]
Thus, by Grönwall's inequality $\norm{\mathbf{u}(a)}_1 = 0$ for all $a \in [0,t]$, and therefore $\mathbf{u} \equiv 0$, which is a contradiction. Therefore the spectrum of $\mathbf{S}_1$ consists only in $\set{0}$. In particular the constant $\lambda = 1$ is in the domain of the resolvent and $(\mathbf{I} - \mathbf{S}_1)^{-1}$ is then well defined bounded linear operator.

We prove that $\mathbf{S}_2 : (L^1([0,t]))^2 \to (L^1([0,t]))^2$ is bounded and then, rejoining all the previous results, we obtain the compactness of $\mathbf{W}_t : (L^1([0,t]))^2 \to (L^1([0,t]))^2$ which allows us to conclude the Lemma.
\begin{enumerate}[resume]
    \item \textbf{Boundness of $\mathbf{S}_2$. }  For all $\mathbf{f} \in (L^1([0,t]))^2$ we have
    \begin{align*}
        \int_0^{t} \norm{\mathbf{S}_2 \mathbf{f}(s)}_1 ds &\leq \int_0^t  \norm{\boldsymbol{\Psi}(0,s)}_1 \norm{\mathbf{f}(s)}_1 ds  \leq  \norm{\mathbf{f}}_1,
    \end{align*}
    since $t \mapsto \boldsymbol{\Psi}(0,t)$ is decreasing in each coordinate and therefore $\sup_{t>0} \norm{\boldsymbol{\Psi}(0,t)}_1 = \norm{\boldsymbol{\Psi}(0,0)}_1 = 1$. 
    \item \textbf{Compactness of $\mathbf{W}_t$. } As we did for $\mathbf{S}_1$, for all $\mathbf{f} \in (L^1([0,t]))^2$ we have
    \begin{align*}
        \int_0^{t} \pars{\mathbf{W}_t \mathbf{f}(a+s) - \mathbf{W}_t \mathbf{f}(a)} da = 2 \int_0^t \int_0^t & \pars{ \mathbf{K}(a+s,a+s+t-u) - \mathbf{K}(a,a+t-u)} \\
        & (\mathbf{I}-\mathbf{S}_1)^{-1} \mathbf{S}_2 \mathbf{f}(u) \ du \ da,
    \end{align*}
    and therefore
    \begin{align*}
        \norm{\mathbf{W}_t \mathbf{f}(\cdot+s) - \mathbf{W}_t \mathbf{f}(\cdot)}_1 \leq 2 \int_0^t & \pars { \int_0^t \norm{ \mathbf{K}(a+s,a+s+t-u) - \mathbf{K}(a,a+t-u)}_1 da } \\
        & \norm { (\mathbf{I}-\mathbf{S}_1)^{-1} \mathbf{S}_2 \mathbf{f}(u) }_1 du.
    \end{align*}
    By our previous calculations, we can bound $\norm { (\mathbf{I}-\mathbf{S}_1)^{-1} \mathbf{S}_2 \mathbf{f}(u) }_1$ uniformly for $u \in [0,t]$ by some constant $M > 0$ times $\norm{\mathbf{f}}_1$. By the absolutely continuity of the jump time densities, $(a,s) \mapsto \mathbf{K}(a,s)$ is a continuous application in both coordinates, and therefore uniformly continuous on $[0,t]$. Thus, for all $\varepsilon > 0$ there exists a $\delta > 0$ such that whenever $|s| < \delta$, then
    \[
    \norm{\mathbf{W}_t \mathbf{f}(\cdot+s) - \mathbf{W}_t \mathbf{f}(\cdot)}_1 \leq 2 M t^2 \varepsilon,
    \]
    uniformly for all $\mathbf{f} \in (L^1([0,t]))^2$ with $ \norm{\mathbf{f}}_1 \leq 1$. Therefore $\mathbf{W}_t$ maps bounded sets to relatively compact sets and we conclude the proof.
\end{enumerate}
\end{proof}

We recall now some useful definitions.
\begin{definition}[Measure of non-compactness and growth bounds]
\label{def:noncompactnessmeasure}
We define the Kuratowski's measure of non-compactness of a bounded set $A$ of a Banach space $\mathcal{X}$ as
\begin{equation*}
    \alpha(A) := \inf \set{\varepsilon > 0 : A \textrm{ can be covered by a finite number of subsets of } \mathcal{X} \textrm{ of diameter} \leq \varepsilon } ,
\end{equation*}
and the associated measure of non-compactness of a bounded operator $S$ in $\mathcal{X}$ as
\[
\alpha(S) := \inf \set{\varepsilon > 0 : \alpha(S(A)) \leq \varepsilon \alpha(A) \textrm{ for all bounded sets } A \subset \mathcal{X} }.
\]
Finally we define the $\alpha$-growth bound of the semigroup $\mathbf{M}_t$ as
$\ 
\omega_1(\mathbf{M}_t) := \lim_{t\to +\infty} \frac{1}{t} \log \pars{ \alpha \pars{\mathbf{M}_t} }$.
\end{definition}

Thanks to the decomposition proven in Lemma \ref{lemma:decompMt}, we obtain the following estimate on the non-compactness of $\mathbf{M}_t$:
\begin{proposition}
Under Assumptions \ref{ass:assumptions} we have that $\omega_1(\mathbf{M}_t) \leq - \bar b < 0$.
\end{proposition}
\begin{proof}
From Proposition \ref{prop:webb_prop4-9} applied to the decomposition proven on Lemma \ref{lemma:decompMt}, since $\mathbf{W}_t$ is compact we have that
\[
\alpha(\mathbf{M}_t) \leq \alpha(\mathbf{U}_t) \leq \norm{\mathbf{U}_t}.
\]
Now, for all fixed $t \geq 0$, and for all $a \geq 0$, we have term by term
\[
\boldsymbol{\Psi}(a,a+t) \leq \begin{bmatrix} e^{-(\alpha + \underline{b} ) t} & e^{-\underline{b}t} (1 - e^{-\alpha t}) \\ 0 & e^{-\underline{b}t}
\end{bmatrix},
\]
so $\norm{\boldsymbol{\Psi}(a,a+t)}_1 \leq e^{- \underline{b} t } \pars { 2 - e^{-\alpha t} } $. Therefore, for all $\mathbf{f} \in (L^1(\rr_+))^2$, we have
\begin{align*}
    \norm{\mathbf{U}_t \mathbf{f}}_1 &=  \int_0^{+\infty} \norm{ \boldsymbol{\Psi}(a,a+t) \mathbf{f}(a+t)}_1 da  
    \leq e^{- \underline{b} t } \pars { 2 - e^{-\alpha t} } \norm{\mathbf{f}}_1,
\end{align*}
and therefore
\[
\omega_1(\mathbf{M}_t)  \leq \lim_{t\to +\infty} \frac{1}{t} \log \pars { e^{- \underline{b} t } \pars { 2 - e^{-\alpha t} } } = - \underline{b}.
\]
\end{proof}

Then we conclude with the rest of the proof of Proposition \ref{prop:main} as stated in the main text.

\subsection{Links between $\gamma \mapsto \lambda_{\alpha,\gamma}$ and the establishment probability conditions}
\label{app:problink}

    Fix $\gamma^* \in (0,1)$. We will introduce an auxiliary process $\tilde{Z}_t$ that will allow to rederive Proposition \ref{prop:dlambdaGamma_equiv} using the population establishment conditions. Let $\tilde{Z}_t$ be the measure-valued process characterised under $\mathbb{P}^{\alpha,\gamma}_{\mu}$ by an initial condition $\tilde{Z}_0 = \mu$ and by the infinitesimal generator $\tilde{\mathcal{Q}}_{\alpha, \gamma} $ defined as
    \[
    \tilde{\mathcal{Q}}_{\alpha, \gamma} \mathbf{f}(a) := \mathbf{f}'(a) + 2 \mathbf{B}_{\gamma}(a) \mathbf{f}(0) - (\mathbf{D}_{\alpha}(a) + \lambda_{\alpha,\gamma^*} \mathbf{I}) \mathbf{f}(a).
    \]
    The only difference with respect to $\mathcal{Q}_{\alpha,\gamma}$ defined in \eqref{eq:Q} is an additional death rate of value $\lambda_{\alpha,\gamma^*}$ for both types. Consider the associated survival functions
    \[
    \tilde{\psi}_i(s,t) := \psi_i(s,t) e^{- \lambda_{\alpha, \gamma^*} (t-s)}.
    \]
    Then, following the proof of Theorem \ref{thm:pisystem} we have that the extinction probabilities
    \[
    \tilde{\pi}_i^{\alpha,\gamma} := \mathbb{P}_{\delta_{(0,i)}}^{\alpha, \gamma} \pars{ \exists t \geq 0 : \tilde{Z}_t = 0} 
    \]
    are the minimal solutions in $[0,1]^2$ to the quadratic system
    \begin{subequations}
\begin{empheq}[left=\empheqlbrace]{align}
    \tilde{\pi}_0^{\alpha,\gamma} &= \omega_0 + (1-\tilde{q}-\omega_0) p + (1-\tilde{q}-\omega_0) (1-p) (\tilde{\pi}_0^{\alpha,\gamma})^2 + \tilde{q} \tilde{\pi}_1^{\alpha,\gamma}, 
    \label{eq:pi0_tilde} \\
    \tilde{\pi}_1^{\alpha,\gamma} &= \omega_1 + (1-\omega_1) \pars{\gamma \tilde{\pi}_0^{\alpha,\gamma} + (1-\gamma) \tilde{\pi}_1^{\alpha,\gamma} }^2,
    \label{eq:pi1_tilde}
  \end{empheq}
\end{subequations}
where
\begin{align}
    \omega_i & := \mathbb{P}_{\delta_{(0,i)}}^{\alpha,\gamma}\pars{\textrm{Die by death rate}} = \lambda_{\alpha,\gamma^*} \int_0^{+\infty} \tilde{\psi}_i(0,s) ds , \label{eq:omegai} \\
    \tilde{q} & := \mathbb{P}_{\delta_{(0,0)}}^{\alpha,\gamma} \pars{\textrm{Switch before dividing}} = \alpha \int_0^{+\infty} \tilde{\psi}_0(0,s) ds. \label{eq:qtilde}
\end{align}
We can further reduce \eqref{eq:pi0_tilde} to
\begin{equation}
\tilde{\pi}_0^{\alpha,\gamma} = (1 - \tilde{q})\tilde{p}   + (1-\tilde{q}) (1-\tilde{p}) (\tilde{\pi}_0^{\alpha,\gamma})^2 + \tilde{q} \tilde{\pi}_1^{\alpha,\gamma},
\label{eq:pi0_tilde_new}
\end{equation}
where
$\ 
\tilde{p} := p + (1-p) \frac{\omega_0}{1-q}\ $ 
is the \textit{effective} death probability at division of type 0 induced by the additional death rate $\lambda_{\alpha, \gamma^*}$. Note that since $0 \leq \omega_0 + q \leq 1$, we have indeed $ 0 \leq \tilde{p} \leq 1$. This way, \eqref{eq:pi0_tilde_new} has the same form as \eqref{eq:pi0} studied previously. Then, following the proof of Theorem \ref{prop:survival}, system \eqref{eq:pi0_tilde}-\eqref{eq:pi1_tilde} admits a minimal solution different than $(1,1)$ if and only if 
\begin{equation}
    \set{\tilde{p} \leq \frac{1}{2}} \textrm{  or  } \set{ \tilde{p} > \frac{1}{2} \textrm{ and }  \omega_1 \leq \frac{1}{2} \textrm{ and } \gamma < \frac{(1-2 \omega_1)(\tilde q + (1- \tilde q)(2 \tilde p -1)}{2(\omega_1 \tilde q + (1 + \omega_1)(1-\tilde q) (2 \tilde p -1))}  }.
    \label{eq:extinctionCondition_tilde}
\end{equation}
 We can recover the previous case (condition \eqref{eq:extinctionCondition}) by making $\omega_1 = 0$. 

\bigskip

We exhibit a monotonicity result analogous to Proposition \ref{prop:dgamma_pi}, but for the system \eqref{eq:pi0_tilde}-\eqref{eq:pi1_tilde} that has an additional extinction probability $\omega_1$, absent in the previously studied system \eqref{eq:pi0}-\eqref{eq:pi1}. 

\begin{lemma}
\label{lemma:monotonicTilde}
For any solution of \eqref{eq:pi0_tilde}-\eqref{eq:pi1_tilde}, for any $i \in \{0,1\}$:
\begin{itemize}
    \item If $\omega_1 = \tilde{p}$, then for all $\gamma \in (0,1)$, $\partial_\gamma \tilde{\pi}_{i}^{\alpha,\gamma} = 0$.
    \item If $\omega_1 > \tilde{p}$, then for all $\gamma \in (0,1)$, $\partial_\gamma \tilde{\pi}_{i}^{\alpha,\gamma} \leq 0$.
    \item If $\omega_1 < \tilde{p}$, then for all $\gamma \in (0,1)$, $\partial_\gamma \tilde{\pi}_{i}^{\alpha,\gamma} \geq 0$.
\end{itemize}
\end{lemma}

\begin{proof}
    Fix $x \in (0,1)$. To study \eqref{eq:pi1_tilde}, we consider the minimal solution to the quadratic equation
    \begin{equation*}
        y = \omega  + (1 - \omega) (\gamma x + (1-\gamma) y)^2 , \quad y \in [0,1] ,\tag{$E_1$}
        \label{eq:y_pi1}
    \end{equation*}
    which is given by
    \begin{equation}
    y_x(\gamma, \omega) = \frac{1 - 2 \gamma(1-\gamma)(1- \omega) x - \sqrt{1 - 4 \gamma(1-\gamma) (1-\omega) x  - 4 (1 - \gamma)^2 (1-\omega) \omega } }{2(1-\omega) (1-\gamma)^2}.
    \label{eq:yx}
    \end{equation}
    We show first that there exists a unique $\bar \omega_x \in (0,1)$ such that for all $\gamma \in (0,1)$, $\partial_{\gamma } y_x (\gamma, \bar \omega_x)  = 0$. By differentiation of \eqref{eq:yx} we have that, for $x \neq 1$, $\partial_{\gamma } y_x (\gamma, \bar \omega_x) = 0$ for all $\gamma \in (0,1)$ if and only if 
    \[
    \bar \omega_x = \frac{x}{1+x}.
    \]
    Now, to determine the value of $p = p_x$ compatible with $y_x(\gamma, \bar \omega_x ) = x$, we study the equation 
    \begin{equation*}
        x = (1-q)p_x + (1-q)(1-p_x)x^2 + q y_x(\gamma, \omega), \tag{$E_0$}
        \label{eq:y_pi0}
    \end{equation*}
    which admits a solution $x = y_x(\gamma, \omega) $ if and only if
$\ 
    p_x = \frac{x}{1+x}$, from where we deduce $\bar \omega_x = p_x$. Then, by implicit differentiation of \eqref{eq:y_pi0}, valuating under $x = y_x(\gamma, \omega) $, we have for all $\gamma \in (0,1)$, 
    \[
    \partial_{\omega} \pars{\partial_\gamma y_x(\gamma,\bar \omega_x) } = - \frac{2(1-x)x(1+x)^2}{(1-(1-2\gamma)x)^2} < 0.
    \]
    In the final case, if $x = 1$ then from \eqref{eq:y_pi0}, for all $\gamma \in (0,1)$, $y_x(\gamma,\omega) = 1$ too. In particular, $\partial_{\gamma } y_x (\gamma, \omega) = 0$ for all $\gamma \in (0,1)$ and $\omega \in (0,1)$. 

    Finally, since from \eqref{eq:y_pi0}, $x \in (0,1) \mapsto y_x(\gamma,\omega) \in (0,1)$ is increasing, if $\partial_{\omega} \partial_\gamma y_x(\gamma,\bar \omega_x) \leq 0$ for all fixed $x$, by the implicit function theorem, $\partial_{\omega} \partial_\gamma \tilde{\pi}_{1}^{\alpha, \gamma} \leq 0$.
\end{proof}

\bigskip

Then, as in Lemma \ref{lemma:rhoKinf}, we show now that there is an equivalence relation between condition \eqref{eq:extinctionCondition_tilde} and the spectral properties of the matrix 
\[ 
\tilde{\mathbf{K}}_{\infty}^{\alpha,\gamma} := 2 \int_0^{+\infty} \tilde{\boldsymbol{\Psi}}_{\alpha}(0,a) \mathbf{B}_{\gamma}(a)  da .
\]
As in Lemma \ref{lemma:rhoKinf} we obtain 
\[
\tilde{\mathbf{K}}_\infty^{\alpha,\gamma} = 2 \begin{bmatrix}
(1-\tilde{p}) (1-\tilde{q}) + \gamma \tilde{q} & (1 - \gamma) \tilde{q} \\
(1-\omega_1) \gamma & (1-\omega_1)(1 - \gamma )
\end{bmatrix},
\]
from which we can compute explicitly the eigenvalues, so that we are able to show that
\[
\rho(\tilde{\mathbf{K}}_\infty^{\alpha,\gamma} ) > 1 \iff \textrm{\eqref{eq:extinctionCondition_tilde} is verified} \iff \tilde{\pi}_0^{\alpha,\gamma^*}, \tilde{\pi}_1^{\alpha,\gamma^*} < 1.
\tag{$\star$}
\label{eq:equivTilde}
\]
Then, following Lemma \ref{lemma:uniquenessLambda}, there is a unique $\tilde \lambda_{\alpha,\gamma} \in \rr$ such that $\rho(\tilde{\mathbf{F}}_{\alpha,\gamma}( \tilde \lambda_{\alpha,\gamma})) = 1$. By construction, $\tilde{\lambda}_{\alpha,\gamma} = \lambda_{\alpha,\gamma} - \lambda_{\alpha,\gamma^*}$. Indeed, $\lambda_{\alpha,\gamma}$ is the largest solution to \eqref{eq:malthusianParam} and then,
    \[
    \det \pars{ 2 \int_0^{+\infty} e^{- (\lambda_{\alpha,\gamma} - \lambda_{\alpha,\gamma^*}) s } \tilde{\mathbf{K}}_{\alpha,\gamma}(a,s) \ ds \ - \mathbf{I}} = \det \pars{ 2 \int_0^{+\infty} e^{ -\lambda_{\alpha,\gamma}  s } \mathbf{K}_{\alpha,\gamma}(a,s) \ ds \ - \mathbf{I}} = 0.
    \]
In particular, for $\gamma = \gamma^*$, $\tilde{\lambda}_{\alpha,\gamma^*} = 0$, and therefore by \eqref{eq:equivTilde},
\[
\tilde{\pi}_0^{\alpha,\gamma^*} = \tilde{\pi}_1^{\alpha,\gamma^*} = 1.
\]

Now, suppose that $\gamma^*$ is such that $\omega_1 < \tilde p.$ Then, by Lemma \ref{lemma:monotonicTilde}, for all $\gamma \geq \gamma^*$, $\tilde{\pi}_0^{\alpha,\gamma} = \tilde{\pi}_1^{\alpha,\gamma} = 1$. Then, by \eqref{eq:equivTilde}, $\rho(\tilde{\mathbf{K}}_\infty^{\alpha,\gamma} ) \leq 1$ for all $\gamma \geq \gamma^*$. Therefore, again by Lemma \ref{lemma:uniquenessLambda}, we have $\tilde{\lambda}_{\alpha,\gamma} \leq 0$ for all $\gamma \geq \gamma^*$, or equivalently, $\lambda_{\alpha,\gamma}  \leq \lambda_{\alpha,\gamma^*}$ for all $\gamma \geq \gamma^*$. We proceed analogously for the case $\omega_1 > \tilde p$ and $\omega_1 = \tilde p$. Finally, recalling from Proposition \ref{prop:calculDlambda} that $\gamma \mapsto \lambda_{\alpha,\gamma}$ is continuously differentiable, we obtain that
\[
\textrm{sign} \pars{\partial_\gamma \left. \lambda_{\alpha, \gamma} \right|_{\gamma = \gamma^*}} = \textrm{sign} \pars{ \omega_1 - \tilde{p} } = \textrm{sign} \pars{ 
\mathbb{P}^{\alpha, \gamma^*}_{\delta_{(0,1)}} \pars{\tilde Z_{T_1} = 0} - \mathbb{P}^{\alpha, \gamma^*}_{\delta_{(0,0)}} \pars{\left. \tilde Z_{T_1} = 0 \right| \textrm{Do not switch}} 
} .
\]
Finally, from \eqref{eq:omegai} an \eqref{eq:qtilde} and by an integration by parts we obtain
\[
\omega_1 = \tilde p  \iff p = \bar p (\alpha, \gamma) := 
\frac{\alpha \xi_0(\lambda_{\alpha,\gamma}) \xi_1(\lambda_{\alpha,\gamma}) + \lambda_{\alpha,\gamma} (\xi_0(\lambda_{\alpha,\gamma}) + \xi_1(\lambda_{\alpha,\gamma}) - 1) }{ (\alpha + \lambda_{\alpha,\gamma}) \xi_0(\lambda_{\alpha,\gamma})  } ,
\]
where 
\[ 
\xi_i(\lambda) := \int_0^{+\infty} e^{-\lambda a} \beta_i(a) \exp \pars{ - \int_0^a \beta_i(s) ds} da
\]
is the Laplace transform of the division time distribution of type $i \in \{0,1\}$. Thus, for all $\alpha, \gamma$, $\bar p(\alpha,\gamma)$ is uniquely defined.

\end{document}